\documentclass[11pt, twoside]{article}
\usepackage{ntheorem}
\theorembodyfont{\itshape}

\usepackage{amssymb}  
\usepackage{dsfont}
\usepackage{amsmath}
\usepackage{graphicx}
\usepackage{threeparttable, float}  
\usepackage[mathscr]{euscript}
\usepackage{enumitem}
\setlist[enumerate]{topsep=2pt,partopsep=0pt,itemsep=3pt, parsep=0pt}

\usepackage{indentfirst, multirow, lineno}
\usepackage[figuresright]{rotating}
\usepackage{setspace}
\usepackage{tikz}
\usepackage{accents}
\usepackage{yhmath,bm}

\usepackage{caption}    
\usepackage{subcaption} 

\definecolor{RoyalBlue}{RGB}{65,105,225}
\definecolor{ForestGreen}{RGB}{34,139,34}

\usepackage{yhmath}
\usepackage{accents}
\usepackage{bm}

\usepackage{cite}
\usepackage{appendix}

\newcommand{\undt}[1]{\underaccent{\raisebox{0.2ex}{\scalebox{0.5}{$\bm{\sim}$}}}{#1}}
\newcommand{\undl}[1]{\underaccent{\raisebox{0.2ex}{\scalebox{0.65}{$\bm{\sim}$}}}{#1}}

\allowdisplaybreaks[4]\allowdisplaybreaks

\usepackage[colorlinks=true]{hyperref}
\hypersetup{urlcolor=red, citecolor=blue}

\newtheorem{theorem}{Theorem}[section]

\newtheorem{lemma}[theorem]{Lemma}
\newtheorem{proposition}[theorem]{Proposition}

\newtheorem{remark}{Remark}[section]

\numberwithin{equation}{section}

\newcommand\bee{\begin{equation}}
\newcommand\eee{\end{equation}}

\newcommand\ol{\overline}
\newcommand\kk{\left}
\newcommand\rr{\right}
\newcommand\dd{\displaystyle}
\newcommand\oo{\Omega}
\newcommand\boo{\overline\Omega}
\newcommand\qq{\eqref}
\newcommand\www{\vspace{-2mm}}
\newcommand\zzz{\vspace{-1mm}}

\newcommand\aaa{\vspace{0.5mm}}

\newcommand\ud{\underline}
\newcommand\dx{\mathrm dx}
\newcommand\dt{\mathrm dt}
\newcommand\yy{\infty}
\newcommand\ep{\varepsilon}
\newcommand\nm{\nonumber}

\newenvironment{proof}[1][Proof]{{\noindent\it #1.}}{\hfill \fontsize{10pt}{10pt}$\Box$\vskip 4pt}

\usepackage{lineno}

\begin{document}\thispagestyle{empty}
 \setlength{\baselineskip}{16pt}
\setlength{\abovedisplayskip}{7pt}
\setlength{\belowdisplayskip}{7pt}

\begin{center}{\Large\bf Diffusion contrast induced dynamics in a time-periodic}\\[2mm]
 {\Large\bf competition–diffusion system with equal total resources}
 \footnote{The work of the first author was supported by the Natural Science Foundation of Hunan Province (No. 2026JJ60002) and Hunan Basic Science Research Center for Mathematical Analysis: 2024JC2002. The work of the second author was supported by the National Natural Science Foundation of China: 12571221. The authors are extremely grateful to Professors Binxiang Dai and Zhi-An Wang for their insightful comments and suggestions.}\\[4mm]
Zhenzhen Li\\
{\small School of Mathematics and Statistics, HNP-LAMA, Central South University, Changsha 410083, China}\\[2mm]
Mingxin Wang\footnote{Corresponding author. {\sl E-mail}: mxwang@sxu.edu.cn (M. Wang).}\\
  {\small School of Mathematics and Statistics; Key Laboratory of Complex Systems and Data Science of the Ministry of Education; Shanxi Key Laboratory for Mathematical Technology in Complex Systems, Shanxi University, Taiyuan 030006, China}
\end{center}

\begin{quote}
\noindent{\bf Abstract.} This paper investigates a time-periodic, spatiotemporally heterogeneous competition-diffusion system, where two species have different intrinsic growth rates but equal total resources per period. Using asymptotic analysis and principal eigenvalue theory, we systematically examine its dynamics. When both diffusion rates are small, spatial heterogeneity of the difference in time-averaged resources leads to uniform persistence, stable coexistence, and asymptotic spatial segregation; whereas its spatial uniformity produces multiple regimes—fast-diffuser selection, parameter-dependent coexistence, or slow-diffuser dominance—separated by smooth threshold curves under different structural conditions. In the regime where at least one diffusion rate is large, we obtain a detailed classification of the local dynamics, with sharp delineation of parameter regions for stable coexistence, competitive exclusion, and a novel bistable scenario, where both semi-trivial periodic solutions are linearly stable while an unstable coexistence state also exists. We further characterize the asymptotic profiles of positive solutions in the mixed-scale diffusion limit. Our results show that the interplay between temporal periodicity and spatial heterogeneity can overturn the classical "slower diffuser always prevails" principle, leading to a substantially richer range of ecological outcomes.


\end{quote}

\section{Introduction and main results}\label{sec1}
\setcounter{equation}{0}\setlength\arraycolsep{2pt}
\markboth{\rm$~$ \hfill Time-periodic competition-diffusion system\hfill $~$}{\rm$~$ \hfill Z. Li \& M. Wang\hfill $~$}

Nonlinear periodic-parabolic equations arise naturally in population ecology, where environmental conditions display time-periodic variations driven by seasonal or diurnal cycles. Temporal and spatial fluctuations in physical factors---such as temperature, light intensity, and nutrient availability---are widely acknowledged as critical factors that can disrupt and even reverse interspecific competitive interactions.

To elucidate how dispersal evolves in temporally periodic and spatially heterogeneous environments, Hutson et al.~\cite{HMP01} proposed and systematically analyzed the following time-periodic Lotka--Volterra competition-diffusion system  for species with a common resource:
  \begin{equation}\label{1.1}
 \begin{cases}
  u_t=d_1 \Delta u+u\bigl(m(x,t)-u-v\bigr), & x\in\Omega,\; t>0,\\
  v_t=d_2 \Delta v+v\bigl(m(x,t)-u-v\bigr), & x\in\Omega,\; t>0,\\
\partial_\nu u=\partial_\nu v=0, & x\in\partial\Omega,\; t>0,\\
u(x,0)=u_0(x) \geq 0,\not\equiv 0,\; v(x,0)=v_0(x) \geq 0,\not\equiv 0, & x \in \overline{\Omega},
  \end{cases}\end{equation}
where $\Omega \subset \mathbb{R}^n$ is a bounded domain with smooth boundary, $\nu$ is the outward unit normal vector on $\partial\Omega$, $d_1,d_2 > 0$ denote the diffusion rates of $u$ and $v$, respectively. The function $m(x, t)$ is $T$-periodic in time. Representing the local carrying capacity or intrinsic growth rate of the two species, it reflects the environmental influence on the species.

When $m(x,t)=m(x)$ is independent of time $t$,  Dockery et al. \cite{DHMP98} confirmed that
the slower diffuser actually always wipes out its faster counterpart regardless of their initial values---this is known as “the slower diffuser always prevails”. Since then, the effects
of spatial heterogeneity have been studied extensively in the past few decades, from various  perspectives. See \cite{CC03,Lou06,Lou08,LL14,HLM02,GHN20,GHN20a,LL23} and references therein.

To state the result of Dockery et al. \cite{DHMP98} precisely, we define
$Q_T=\Omega \times (0,T]$ and $S_T=\partial\Omega \times (0,T]$ in the usual way. 
For $k,l\geq 0$, we define
  \[C_T^{k,l}(\overline Q_T)=\big\{\phi\in C^{k,l}(\overline Q_T): \phi(x,0)=\phi(x, T),\; \forall\, x\in\boo\big\},\]
and the temporal and spatial averages of a function $m\in C_T(\overline Q_T)$ by
  \[\widehat{m}(x):=\frac{1}{T}\int_0^T m(x,t)\mathrm dt,\quad \ol m(t):=\frac{1}{|\Omega|}\int_{\Omega}m(x,t)\dx.\]
Consider the following
(single species) logistic equation:
\bee\label{1.2}
 \begin{cases}
\theta_t=d\Delta \theta+\theta(m(x,t)-\theta) &\text{in}\;\;Q_T,\\
\partial_{\nu}\theta=0&\text{on}\;\;S_T,\\
\theta(x,0)=\theta(x,T)&\text{on}\;\;\boo,
  \end{cases} \eee
where $d>0$, $m\in C^{\alpha,\alpha/2}_T(\overline Q_T)$. Let $\mu(d,m)$ be the principal eigenvalue of
 \bee\label{1.3}
\begin{cases}
\varphi_t-d\Delta\varphi-m(x,t)\varphi=\mu\varphi&\text{in}\;\;Q_T,\\
\partial_{\nu}\varphi=0&\text{on}\;\;S_T,\\
\varphi(x,0)=\varphi(x,T)&\text{on}\;\;\ol\Omega.
\end{cases}
 \eee

It is well-known that \qq{1.2} admits a positive solution, denoted by $\theta_{d,m}(x,t)$, if and only if $\mu(d,m)<0$. Whenever $\theta_{d,m}$ exists, it belongs to $C^{2+\alpha,1+\alpha/2}_T(\overline Q_T)$, is uniquely determined, and is globally asymptotically stable (with respect to positive initial data). If $\mu(d,m)\geq 0$, then $\theta=0$ is globally asymptotically stable for all nonnegative initial values.

When $m>0$ in $Q_T$, system \eqref{1.1} admits two semi-trivial $T$-periodic solutions $(\theta_{d_1,m},0)$ and $(0,\theta_{d_2,m})$. The following theorem demonstrates that slow diffusion dominance always arises when $m$ does not depend on time $t$.\vskip 4pt

\noindent{\bf Theorem A} (\!\!\cite{DHMP98}) {\it Let $m(x,t)$ be spatially heterogeneous but temporally homogeneous, namely $m(x,t)=m(x)$. If $d_1<d_2$, then $(\theta_{d_1,m},0)$ is globally asymptotically stable for the system \eqref{1.1}. That is, solution $(u,v)$ of \eqref{1.1} converges to  $(\theta_{d_1,m},0)$ as $t\to\infty$. }\vskip 4pt

Theorem {\bf A} indicates that in a spatially varying yet temporally constant environment, faster dispersal is always selected against under completely random dispersal. In particular, coexistence of the two species is impossible. However, the scenario changes drastically when additional temporal periodicity is incorporated into the model. In the pioneering work in 2001, Hutson et al. \cite{HMP01} demonstrated that by considering various choices of the resource function $m(x,t)$ and the dispersal rates $d_1, d_2$ of the two species, three outcomes are possible: selection of the lower dispersal rate, selection of the higher dispersal rate, or stable coexistence of the two species.

More recently, Bai et al.~\cite{BHN23} provided a detailed classification for the dynamics of \eqref{1.1}:
\begin{enumerate}
\item[(a)] If one diffusion rate is sufficiently large, the slower diffuser (provided its rate is not extremely low) drives the faster one to extinction, analogous to the time-independent case.
\item[(b)] When one rate is large and the other relatively small, both semi-trivial periodic solutions lose stability, leading to coexistence.
\item[(c)] When both diffusion rates are sufficiently small, they almost completely characterize the local dynamics of system \eqref{1.1} in terms of $m, d_1$ and $d_2$.
\end{enumerate}
In addition, they also showed that when $I(\Phi)>0$ with $I(\Phi)$ defined in \eqref{eq1.9}, there exists $0<\ep\ll 1$ such that the faster diffuser $(0,\theta_{d_2,m})$ may become globally asymptotically stable when $0<d_1<\ep$ and $d_2>d_1$ sufficiently close to $d_1$. This contrasts sharply with the spatially heterogeneous but 
temporally homogeneous case in Theorem {\bf A}.

These findings highlight a fundamental difference between time-periodic and time-independent environments: unlike the latter case where slower diffusers always dominate~\cite{DHMP98}, the time-periodic setting permits richer outcomes, including faster-diffuser dominance and coexistence, depending on parameter choices.

The above works concern ecologically equivalent competitors sharing the same spatio-temporal resource function. In the present paper, we consider a substantially different situation in which the two species experience distinct intrinsic growth rates and one resource may be autonomous. Specifically, we study the population dynamics of the following system:
  \begin{equation}\label{1.4}
\begin{cases}
u_t=d_1 \Delta u+u\bigl(m_1(x,t)-u-v\bigr), & x\in\Omega,\; t > 0,\\
v_t=d_2 \Delta v+v\bigl(m_2(x,t)-u-v\bigr), & x\in\Omega,\; t > 0,\\
\partial_\nu u=\partial_\nu v=0, & x\in\partial\Omega,\; t > 0,\\
u(x,0)=u_0(x) \geq 0,\not\equiv 0,\; v(x,0)=v_0(x) \geq 0,\not\equiv 0, & x \in \overline{\Omega},
\end{cases}
\end{equation}
where $m_i(x,t)$ ($i=1,2$) are positive $T$-periodic functions representing the intrinsic growth rates of species $u$ and $v$, respectively. The corresponding $T$-periodic problem is
  \begin{equation}\label{1.5}
 \begin{cases}
 U_t=d_1 \Delta U+U\bigl(m_1(x,t)-U-V\bigr) &\text{in }\; Q_T,\\[2pt]
V_t=d_2 \Delta V+V\bigl(m_2(x,t)-U-V\bigr) &\text{in }\; Q_T,\\[2pt]
\partial_\nu U=\partial_\nu V=0 &\text{on }\; S_T,\\[2pt]
U(x,0)=U(x,T),\; V(x,0)=V(x,T) &\text{on }\; \overline{\Omega}.
  \end{cases}\end{equation}

When $m_i$ are time-independent ($m_i=m_i(x)$), system \eqref{1.4} has attracted considerable attention over the past two decades; see, e.g.,~\cite{HN13a, HN13b, DHMP98,HN16a,HN16b,HN17}. Additional contributions can be found in~\cite{CL84,HN13b,HLM02,Lou06,LN12,ZX18,ZTX21} and references therein. However, much less is known for the time-periodic case with $m_1\not\equiv m_2$. From the analytical point of view, this setting is not a routine extension of the common-resource case ($m_1\equiv m_2$), because the two semi-trivial states are now governed by different scalar periodic-parabolic equations, and the corresponding invasion analysis becomes genuinely asymmetric. The purpose of this paper is to investigate how diffusion rates, temporal periodicity, and spatial heterogeneity jointly affect the competition outcome in system \eqref{1.4} when $m_1\not\equiv m_2$.
As a first step, we focus on a biologically meaningful scenario described by the following condition:
\setlist[enumerate]{topsep=4pt,partopsep=1pt,itemsep=3pt, parsep=1pt}
 \begin{enumerate}[leftmargin=15mm]
\item[{\bf(M)}] For $i=1,2$, $m_i\in C_T^{2,1}(\overline{Q}_T)$ satisfies $m_i > 0$, and $\nabla m_i \not\equiv 0$ on $\overline{\Omega}\times[0,T]$;
 $m_1 \not\equiv m_2$ but $\int_{Q_T}(m_1-m_2)\mathrm{d}x\mathrm{d}t=0$.
\end{enumerate}
Condition {\bf(M)} assumes that the two species experience different spatiotemporal resource distributions, yet the \emph{total amount of resources} available to the two species over a full cycle is equal. In particular, it also covers both the autonomous case and the mixed case where one growth rate is autonomous while the other is time-periodic, provided that each $m_i$ remains spatially heterogeneous. This hypothesis is motivated by the works of He and Ni~\cite{HN13a,HN17} and allows us to isolate the effect of spatio-temporal variation from mere resource imbalance.

For system \eqref{1.4}, the trivial solution $(0,0)$ is always linearly unstable for all parameter values. By condition {\bf(M)}, system \eqref{1.5} admits semi-trivial solutions $(\theta_{d_1,m_1},0)$ and $(0,\theta_{d_2,m_2})$.

Let $m\in C^{\alpha,\alpha/2}_T(\overline Q_T)$ be spatially heterogeneous. For each $x\in\boo$, the periodic problem
 \bee\begin{cases}\label{1.6}
  \Phi_t=\Phi(m(x,t)-\Phi),\;\; 0<t\le T,\\[0.1mm]
 \Phi(x,0)=\Phi(x,T)
  \end{cases}\eee
admits a positive solution $\Phi_m$ if and only if $\widehat{m}(x)>0$, and $\Phi_m$ is unique and $\Phi_m\le \max_{\ol Q_T}|m|$ when it exists. Moreover, we infer from the continuous dependence of the solutions to ODE on the parameters that $\Phi_{m}\in C^{\alpha,1}(\overline Q_T)$ (cf. \cite{H09}). Moreover, 
 \bee
 \nabla\Phi_m\not\equiv 0,\;\;\widehat{m}(x)=\widehat \Phi_m(x).
 \label{1.7}\eee

The main contributions are summarized as follows.\vskip 1pt

{\bf Both diffusion rates are small}: When $d_1$ and $d_2$ are sufficiently small, the dynamics are primarily governed by the difference in time-averaged resource quantities $\widehat{m}_1(x)-\widehat{m}_2(x)$.\vspace{-2mm}

\begin{theorem}\label{th1.1} Assume that
  \bee \widehat{m}_1(x)-\widehat{m}_2(x)\not\equiv 0\;\;\;\text{on}\;\;\boo,\label{1.8}\eee
i.e., the averaged difference $\widehat{m}_1(x)-\widehat{m}_2(x)$ is spatially heterogeneous. Define
  \begin{align*}
  \Omega_1&=\{x\in \overline{\Omega}: \widehat{m}_1(x)>\widehat{m}_2(x)\},\\
  \Omega_2&=\{x\in \overline{\Omega}: \widehat{m}_1(x)<\widehat{m}_2(x)\},\\
  \Omega_0&=\overline{\Omega} \setminus (\Omega_1 \cup \Omega_2).
   \end{align*}
Then there exists $\varepsilon_0 > 0$ such that for all $d_1,d_2\in (0,\varepsilon_0)$, system \eqref{1.4} is uniformly persistent, and system \eqref{1.5} admits a linearly stable positive solution. Moreover, in the space $[C((\overline{\Omega}\setminus\Omega_0)\times[0,T])]^2$, any positive solution $(U,V)$ of \eqref{1.5} satisfies
    \bee\lim_{(d_1,d_2)\to(0,0)}(U,V)=\begin{cases}
    (\Phi_{m_1}(x,t), 0), & x\in \Omega_1,\\[2pt]
    (0, \Phi_{m_2}(x,t)), & x\in \Omega_2,
    \end{cases}\label{a.9}\eee
where $\Phi_{m_i}$ is the unique positive solution to \qq{1.6} with $m=m_i$.\vspace{-2mm}
\end{theorem}

This theorem reveals that at small diffusion rates, the system achieves uniform persistence and stable coexistence. However, as the diffusion rate approaches zero, the limiting profile becomes spatially segregated, which depends on the spatial dominance of $\widehat{m}_i(x)$---specifically, the size relationship between $\widehat{m}_1(x)$ and $\widehat{m}_2(x)$. In this scenario, the dominant species will occupy favorable regions and exclude the other species.

When condition \qq{1.8} fails, one only has
$\widehat m_1(x)\equiv \widehat m_2(x)$ on $\overline\Omega$, which does not
in general imply that $m_1(x,t)-m_2(x,t)$ is independent of $x$.
For example, $m_1(x,t)-m_2(x,t)=f(x)h(t)$ with $h$ being $T$-periodic and
$\int_0^T h(t)\,\mathrm dt=0$ also leads to the failure of \qq{1.8}.
In what follows, we focus on the special case where $m_1(x,t)-m_2(x,t)$ is
independent of $x$, and obtain a detailed classification paralleling that
of~\cite{BHN23} for the common-resource case. For $i=1,2$, let $\Phi_i=\Phi_{m_i}$ and define
\begin{equation}\label{eq1.9}
I(\Phi_i)=\int_{\Omega}\kk[\exp\kk(\frac{2}{T}\int_0^T\ln\Phi_i
  \,\mathrm{d}t\rr)\int_0^T \frac{\Delta \Phi_i}{\Phi_i} \,\mathrm{d}t\rr] \mathrm{d}x.
\end{equation}

\begin{theorem}\label{th1.2} Suppose that \eqref{1.8} does not hold, i.e., $\widehat{m}_1(x)\equiv\widehat{m}_2(x)$, and $m_1-m_2$ is independent of $x$. We further assume the following condition:
\begin{enumerate}[leftmargin=15mm]
\item[{\bf(M1)}] $\partial_{\nu}m_i(x,t)=0 \quad \forall\,(x,t)\in \partial\Omega\times[0,T].$
\end{enumerate}
Then there exists $\varepsilon_0>0$ with $\varepsilon_0\ll 1$ such that all the assertions below hold.
  \begin{enumerate}
\item[\rm(1)] {\rm(}Selection for faster diffusion{\rm)} If $I(\Phi_i)>0$ and $\min_{x\in\overline\Omega}\int_0^T\frac{\Delta \Phi_i}{\Phi_i}(x,t)\mathrm dt>0$ for $i=1,2$, then $(\theta_{d_1,m_1},0)$ is linearly stable and $(0,\theta_{d_2,m_2})$ is linearly unstable for all $(d_1,d_2)\in(0,\ep_0)\times(0,d_1)$; see the blue region in Figure {\rm\ref{fig1}(b)}.
\item[\rm(2)] If $I(\Phi_i)>0$ and $\min_{x\in\overline\Omega}\int_0^T\frac{\Delta \Phi_i}{\Phi_i}(x,t)\mathrm dt<0$ for $i=1,2$, then there exists a $C^1$ function $\ol{\mathsf{d}}_2:(0,\ep_0)\to(0,\ep_0)$ with $0<\ol{\mathsf{d}}_2(d_1)<d_1$ such that
\begin{itemize}
\item[$(2a)$] {\rm(}Stable coexistence{\em)} System \eqref{1.4} is uniformly persistent and system \qq{1.5} admits a linearly stable positive solution for all $(d_1,d_2)\in(0,\ep_0)\times(0, \ol{\mathsf{d}}_2(d_1))$; see the yellow region in Figure {\rm\ref{fig1}(c)}.
\item[$(2b)$]  {\rm(}Selection for faster diffusion{\rm)} $(\theta_{d_1,m_1},0)$ is linearly stable and $(0,\theta_{d_2,m_2})$ is linearly unstable for all $(d_1,d_2)\in(0,\ep_0)\times(\ol{\textstyle\mathsf{d}}_2(d_1), d_1)$; see the blue region in Figure {\rm\ref{fig1}(c)}.
    \end{itemize}
Moreover,
 \[\lim_{d_1\to 0^+}\ol{\mathsf{d}}_2(d_1)=0,\quad\lim_{d_1\to 0^+}\overline{\mathsf{d}}_2'(d_1)=\overline{\mathsf{d}}_2'(0)\in(0,1).\aaa\]
\item[\rm(3)] If $I(\Phi_i)<0$ for $i=1,2$, and the condition
  \begin{enumerate}[leftmargin=15mm]
\item[{\bf(M2$'$)}] for $i=1,2$, $m_i(x,t)\not\equiv a(x){\rm e}^{\int_0^tb(s)\mathrm ds}+b(t)$ for any functions $a(x)\in C(\overline\Omega)$ and $b(t)\in C([0,T])$
  \end{enumerate}
holds, then there exists a $C^1$ function $\ud{\mathsf{d}}_2:(0,\ep_0)\to(0,\ep_0)$ with $0<\ud{\mathsf{d}}_2(d_1)<d_1$ such that
\begin{enumerate}
\item[$(3a)$]  {\rm(}Stable coexistence{\em)} System \eqref{1.4} is uniformly persistent and system \qq{1.5} admits a linearly stable positive solution for all $(d_1,d_2)\in(0,\ep_0)\times(0,\ud{\mathsf{d}}_2(d_1))$; see the yellow region in Figure {\rm\ref{fig1}(d)}.
\item[$(3b)$]  {\rm(}Selection for slower diffusion{\rm)} $(\theta_{d_1,m_1},0)$ is linearly unstable and $(0,\theta_{d_2,m_2})$ is linearly stable for all $(d_1,d_2)\in(0,\ep_0)\times(\ud{\mathsf{d}}_2(d_1),d_1)$; see the black region in Figure {\rm\ref{fig1}(d)}.
\end{enumerate}
 Moreover,
  \[\lim_{d_1\to 0^+}\ud{\mathsf{d}}_2(d_1)=0,\quad\lim_{d_1\to 0^+}\ud{\mathsf{d}}_2'(d_1)=\ud{\mathsf{d}}_2'(0)\in(0,1).\]
\end{enumerate}
\end{theorem}
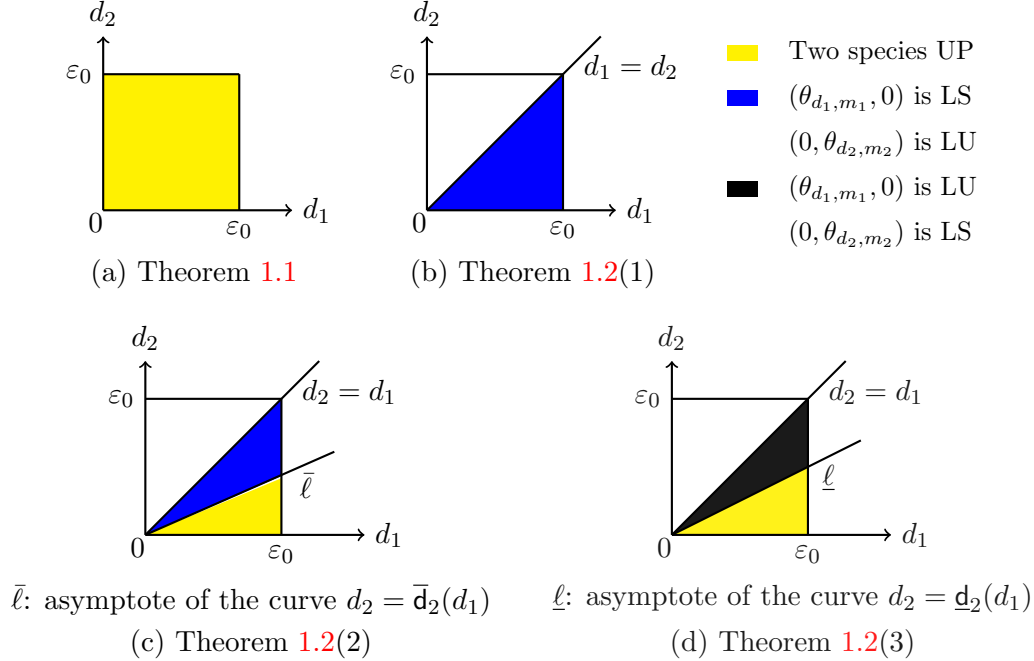
\begin{figure}[H]\centering
\begin{tabular}{cc}
\begin{tikzpicture}[thick,fill opacity=1]
\fill[yellow] (0,0)--(0,1.8) --(1.8,1.8)--(1.8,0)--cycle;
\draw[thick, black][->] (0,0)--(2.5,0);\draw[thick, black][->] (0,0)--(0,2.3);
\node[right] at (2.5,0) {$d_1$};
\node[above] at (0,2.3) {$d_2$};
\draw[thick, black] (0,1.8)--(1.8,1.8);
\node[left] at (0,1.8) {$\ep_0$};
\draw[thick, black] (1.8,0)--(1.8,1.8);
\node[below] at (1.8,0) {$\ep_0$};
\node[below] at (-0.1,0.1) {$0$};
\node[below] at (1.2,-0.5){(a) Theorem \ref{th1.1}};
\end{tikzpicture}\hspace{4mm}
\begin{tikzpicture}[thick,fill opacity=1]
\fill[blue] (0,0)--(1.8,1.8) --(1.8,0)--cycle;
\draw[thick, black][->] (0,0)--(2.5,0);\draw[thick, black][->] (0,0)--(0,2.3);
\node[right] at (2.5,0) {$d_1$};
\node[above] at (0,2.3) {$d_2$};
\draw[thick, black] (0,0)--(2.3,2.3);
\draw[thick, black] (0,1.8)--(1.8,1.8);
\node[left] at (0,1.8) {$\ep_0$};
\draw[thick, black] (1.8,0)--(1.8,1.8);
\node[below] at (1.8,0) {$\ep_0$};
\node[below] at (-0.1,0.1) {$0$};
\node[below] at (2.7,2.2) {$d_1=d_2$};
\node[below] at (1.4,-0.5){(b) Theorem \ref{th1.2}(1)};
\end{tikzpicture}\hspace{4mm}
\begin{tikzpicture}[thick,fill opacity=1]
\fill[yellow] (0,0.6) rectangle (0.4,0.8);
\fill[blue] (0,0.0) rectangle (0.4,0.2);
\fill[black] (0,-1.0) rectangle (0.4,-1.2);
{\large\node[right] at (0.65,0.7) {\small Two species UP};
\node[right] at (0.65,0.1) {\small$(\theta_{d_1,m_1},0)$ is LS};
\node[right] at (0.65,-0.5) {\small$(0,\theta_{d_2,m_2})$ is LU};
\node[right] at (0.65,-1.1) {\small$(\theta_{d_1,m_1},0)$ is LU};
\node[right] at (0.65,-1.7) {\small$(0,\theta_{d_2,m_2})$ is LS};
\node[right] at (0.65,-2.4) {};}
\end{tikzpicture}\\[1mm]
\begin{tikzpicture}[thick,fill opacity=1]
\fill[yellow] (0,0)--(1.8,0) --(1.8,0.75)--cycle;
\fill[blue] (0,0)--(1.8,1.8) --(1.8,0.78)--cycle;
\draw[thick, black] (0,0)--(2.5,1.1);
\node[right] at (1.9,0.65) {$\bar\ell$};
\draw[thick, black][->] (0,0)--(2.9,0);\draw[thick, black][->] (0,0)--(0,2.3);
\node[right] at (2.9,0) {$d_1$};
\node[above] at (0,2.3) {$d_2$};
\draw[thick, black] (0,0)--(2.3,2.3);
\draw[thick, black] (0,1.8)--(1.8,1.8);
\node[left] at (0,1.8) {$\ep_0$};
\draw[thick, black] (1.8,0)--(1.8,1.8);
\node[below] at (1.8,0) {$\ep_0$};
\node[below] at (-0.1,0.1) {$0$};
\node[below] at (2.7,2.2) {$d_2=d_1$};
\node[below] at (1.4,-0.5){ $\bar\ell$: asymptote of the curve $d_2=\ol{\mathsf{d}}_2(d_1)$};
\node[below] at (1.4,-1.1){(c) Theorem  \ref{th1.2}(2)};
\end{tikzpicture}
\hspace{3mm}
\begin{tikzpicture}[thick,fill opacity=0.9]
\fill[black] (0,0)--(1.8,1.8) --(1.8,0.9)--cycle;
\fill[yellow] (0,0)--(1.8,0) --(1.8,0.9)--cycle;
\draw[thick,black] (0,0)--(2.5,1.25);
\node[right] at (1.85,0.75) {$\ud\ell$};
\draw[thick, black][->] (0,0)--(2.9,0);\draw[thick, black][->] (0,0)--(0,2.3);
\node[right] at (2.9,0) {$d_1$};
\node[above] at (0,2.3) {$d_2$};
\draw[thick, black] (0,0)--(2.3,2.3);
\draw[thick, black] (0,1.8)--(1.8,1.8);
\node[left] at (0,1.8) {$\ep_0$};
\draw[thick, black] (1.8,0)--(1.8,1.8);
\node[below] at (1.8,0) {$\ep_0$};
\node[below] at (-0.1,0.1) {$0$};
\node[below] at (2.7,2.2) {$d_2=d_1$};
\node[below] at (1.6,-0.5){$\ud\ell$: asymptote of the curve $d_2=\ud{\mathsf{d}}_2(d_1)$};
\node[below] at (1.6,-1.1){(d) Theorem \ref{th1.2}(3)};
\end{tikzpicture}
\end{tabular}\vspace{-3mm}
\caption{Graphical Illustration of Theorem \ref{th1.1} and Theorem  \ref{th1.2}. Abbreviations: UP=uniform persistence, LS=linearly stable, LU=linearly unstable.}
\label{fig1}\vspace{-3mm}
\end{figure}

\begin{remark}\label{rem1.1}
Theorems \ref{th1.1} and \ref{th1.2} are illustrated in Figure~\ref{fig1}. We give several comments on these small-diffusion regimes.
\begin{enumerate}
\item[{\rm(i)}] Theorem \ref{th1.1} establishes that for small diffusion rates $d_1, d_2>0$, the system is uniformly persistent whenever condition \eqref{1.8} is satisfied (see Figure \ref{fig1}(a)). The asymptotic behaviors of positive $T$-periodic solutions as $d_1, d_2 \to 0$ are consistent with those in the spatially heterogeneous but temporally static case (cf. \cite[Theorem 1.10]{HN17}).

\item[{\rm(ii)}] The results of Theorem \ref{th1.2} herein are consistent with those for the case $m_1\equiv m_2$ {\rm(cf. \cite[Theorem 2.1]{BHN23})}. Condition {\bf(M2$'$)} (first introduced in \cite{HMP01} when $m_1\equiv m_2$) is a technical assumption ensuring that $m_i-\Phi_i$ is spatially heterogeneous for $i=1,2$. Theorem \ref{th1.2} classifies the dynamic behaviors of the model system under the conditions that the difference $m_1(x,t)-m_2(x,t)$ is spatially homogeneous, the boundary condition {\bf(M1)} is satisfied, both diffusion rates $d_1$ and $d_2$ are sufficiently small ($0<\ep_0\ll1$), and $d_2<d_1$. The core criterion for this classification hinges on the sign of the integral $I(\Phi_i)$ and the sign of the minimum value of $\int_0^T\frac{\Delta \Phi_i}{\Phi_i}(x,t)\mathrm dt$ (or the additional condition {\bf(M2$'$)} when $I(\Phi_i)<0$); these factors collectively determine whether the two species can attain stable coexistence or one species is competitively excluded by the other.\vspace{-2mm}
\end{enumerate}
\end{remark}

{\bf At least one diffusion rate is large}: When one diffusion rate is large, we derive distinct dynamical regimes that hinge on whether the difference $m_1-m_2$ is spatially dependent. For $m\in C_T^{\alpha,\alpha/2}(\overline Q_T)$ and each $0\le t\leq T$, we denote by $\Gamma_m(\cdot,t)$ the unique solution to
  \bee\begin{cases}
-\Delta \Gamma=m(x,t)-\ol m(t), & x\in\Omega,\\
\partial_{\nu}\Gamma=0, & x\in\partial\Omega,\\[0.5mm]
\ol\Gamma(t)=0.
 \end{cases}\label{1.10}
   \eee
The existence and uniqueness of solutions to \eqref{1.10} follow from the Fredholm alternative theorem in \cite[Theorem 6.2.4]{E10}. It is straightforward to verify that $\Gamma_m(\cdot,t)$ is $T$-periodic in time $t$, and
 \[\|\Gamma_m(\cdot,t)\|_{L^2(\Omega)}\le C\|\nabla\Gamma_m(\cdot,t)\|_{L^2(\Omega)}\leq C \|m(\cdot,t)-\ol m(t)\|_{L^2(\Omega)},\;\,\forall\, t\in[0,T].\]

 For $m\in C_T(\overline Q_T)$ with $m>0$ on $\overline Q_T$, let $l_m(t)$ be the unique positive solution to
  \bee\begin{cases}
 l_t=l(\ol m(t)-l), \;\;0<t\le T,\\[0.5mm]
 l(0)=l(T).\end{cases}\label{1.11} \eee
Hereafter, for brevity, we use the abbreviated integral notations
 \[\int_\oo f=\int_\oo f\dx,\;\;\;\int_{Q_T}f=\int_{Q_T}f\dx\dt.\]
Define
 \begin{align}\label{1.12}
 \mathcal{E}(m_i)=\;&\int_{Q_T}|\nabla\Gamma_{m_i}|^2,\;\;\;i=1,2,\\[0.2mm]
  \label{1.13}
\Pi=\;&\frac 2{\mathcal{E}(m_1)}\int_{Q_T}\kk(\frac{\mathcal{E}(m_2)}{\mathcal{E}(m_1)} l_{m_1}\Gamma_{m_1}^2-(l_{m_1}+ l_{m_2})\Gamma_{m_1}\Gamma_{m_2}+\frac{\mathcal{E}(m_1)}{\mathcal{E}(m_2)} l_{m_2}\Gamma_{m_2}^2\rr).\vspace{-1mm}
 \end{align}

\begin{theorem}[Selection for slower diffusion]\label{th1.3}
The following assertions hold:
\begin{enumerate}
\item[$(1)$]{\rm(}Spatially homogeneous resource difference{\rm)} If $m_1(x,t)-m_2(x,t)$ is spatially homogeneous, then for any $\varepsilon>0$, there exists $d_2^\varepsilon>0$ such that $(\theta_{d_1,m_1},0)$ is globally asymptotically stable for all $\ep<d_1<d_2$ and $d_2>d_2^\varepsilon$.
\item[$(2)$]{\rm(}Spatially heterogeneous resource difference{\rm)} If $m_1(x,t)-m_2(x,t)$ is spatially heterogeneous, then for any given $0 < a < b$, there exists $d_2^{a,b} > 0$ such that $(\theta_{d_1,m_1},0)$ is globally asymptotically stable for all $(d_1,d_2)\in [a,b] \times (d_2^{a,b},\infty)$.
\end{enumerate}\vspace{-3mm}
\end{theorem}

\begin{theorem}[Detailed classification]\label{th1.4} Assume that the condition {\bf(M1)} and the following hold:
  \begin{enumerate}[leftmargin=14mm]
\item[{\bf(M2)}] Both $m_2(x,t)-\Phi_{m_1}(x,t)$ and $m_1(x,t)-\Phi_{m_2}(x,t)$ are spatially heterogeneous.
    \end{enumerate}
Suppose further that $\widehat{m}_i(x)\not\equiv \text{constant}$ on $\boo$ for $i=1,2$. Then there exist constants $\bar{d}_1, \bar{d}_2 > 0$, two strictly decreasing $C^1$ functions $\mathsf{d}_1: (\bar{d}_2,\infty)\to (0,\infty)$ and $\mathsf{d}_2: (\bar d_1,\infty)\to (0,\infty)$, and two strictly increasing $C^1$ functions $\widehat{\mathsf{d}}_2, \widetilde{\mathsf{d}}_2: (\bar{d}_1,\infty)\to (0,\infty)$ satisfying
  \[ \mathsf{d}_1(d_2)=O\kk(d_2^{-1}\rr) \;\;\;{\rm as}\;\;d_2 \to\infty,\quad \mathsf{d}_2(d_1)=O\kk(d_1^{-1}\rr) \;\;\;{\rm as}\;\;d_1 \to\infty, \]
and
 \[\lim_{d_1\to\infty}\frac{\widehat{\mathsf{d}}_2(d_1)}{d_1}=\lim_{d_1\to\infty} \frac{\widetilde{\mathsf{d}}_2(d_1)}{d_1}=\frac{\mathcal{E}(m_2)}
 {\mathcal{E}(m_1)},\]
such that the following statements hold with reference to the regions in Figure~$\ref{fig2}$:
\begin{enumerate}
\item[$(1)$]{\rm(}Spatially homogeneous  resource difference{\rm)} Assume that $m_1(x,t)-m_2(x,t)$ is spatially homogeneous.
\begin{enumerate}
\item[$({1}a)$]  {\rm(}Stable coexistence{\em)} System \eqref{1.4} is uniformly persistent and system \qq{1.5} admits a linearly stable positive solution for all
$(d_1,d_2)\in (0,\mathsf{d}_1(d_2))\times(\bar d_2,\infty)$; see the yellow region in Figure {\rm\ref{fig2}(a)}.
\item[$({1}b)$]  {\rm(}Selection for slower diffusion{\rm)} $(\theta_{d_1,m_1},0)$ is linearly stable and $(0,\theta_{d_2,m_2})$ is linearly unstable for all
$(d_1,d_2)\in (\mathsf{d}_1(d_2),d_2)\times(\bar d_2,\infty)$; see the blue region in Figure {\rm\ref{fig2}(a)}.
\end{enumerate}
\item[$(2)$]{\rm(}Spatially heterogeneous resource difference{\rm)} Assume that $m_1(x,t)-m_2(x,t)$ is spatially heterogeneous, and further that condition {\bf(M3)} holds:
    \begin{enumerate}[leftmargin=10mm]
\item[{\bf(M3)}] $m_i\in H^2((0,T),L^p(\Omega))$ for some $p>n$.
\end{enumerate}
\begin{enumerate}
\item[{$({2}a)$}] If either $\ol m_1(t)\equiv\ol m_2(t)$ on $[0,T]$, or $\ol m_1(t)\not\equiv\ol m_2(t)$ on $[0,T]$ with $\Pi>0$, then the following properties hold:
  \begin{enumerate}
\item[$({2a}_1)$] $\widehat{\mathsf{d}}_2(d_1)>\widetilde{\mathsf{d}}_2(d_1)$ for all $d_1\in(\bar d_1,\infty)$; see Figure {\rm\ref{fig2}(b)}.
\item[$({2a}_2)$]  {\rm(}Stable coexistence{\em)} System \eqref{1.4} is uniformly persistent and system \qq{1.5} admits a linearly stable positive solution for parameters in the yellow region of Figure {\rm\ref{fig2}(b)}, which is defined by the union:
    \begin{equation}(d_1,d_2)\in[(0,\mathsf{d}_1(d_2))\times(\bar d_2,\infty)]\cup[(\bar d_1,\infty)\times(\widetilde{\mathsf{d}}_2(d_1), \widehat{\mathsf{d}}_2(d_1))]\cup[(\bar d_1,\infty)\times(0,\mathsf{d}_2(d_1))].\label{1.x14}\end{equation}
\item[$({2a}_3)$]  $(\theta_{d_1,m_1},0)$ is linearly stable, and  $(0,\theta_{d_2,m_2})$ is linearly unstable in the blue region of  Figure {\rm\ref{fig2}(b)}:
  \begin{equation}(d_1,d_2)\in [(\bar d_1,\infty)\times(\widehat{\mathsf{d}}_2(d_1),\infty)]
  \cup[(\mathsf{d}_1(d_2),\bar d_1]\times(\bar d_2,\infty)].\label{1.x15}\end{equation}

\item[$({2a}_4)$]  $(\theta_{d_1,m_1},0)$ is linearly unstable and  $(0,\theta_{d_2,m_2})$ is linearly stable for all $(d_1,d_2)\in(\bar d_1,\infty)\times(\mathsf{d}_2(d_1), \widetilde{\mathsf{d}}_2(d_1))$; see the black region of Figure {\rm\ref{fig2}(b)}.
   \end{enumerate}\end{enumerate}
\begin{enumerate}
\item[{$({2}b)$}]  If $\ol m_1(t)\not\equiv\ol m_2(t)$ on $[0,T]$ and $\Pi<0$, then the following assertions hold:
 \begin{enumerate}
\item[$({2b}_1)$] $\widehat{\mathsf{d}}_2(d_1)<\widetilde{\mathsf{d}}_2(d_1)$ for all $d_1\in(\bar d_1,\infty)$; see Figure {\rm\ref{fig2}(c)}.
\item[$({2b}_2)$] {\rm(}Stable coexistence{\em)} System \eqref{1.4} is uniformly persistent and system \qq{1.5} admits a linearly stable positive solution for all
    \bee(d_1,d_2)\in[(0,\mathsf{d}_1(d_2))\times(\bar d_2,\infty)]\cup[(\bar d_1,\infty)\times(0,\mathsf{d}_2(d_1))];\label{1.x16}\eee
see the yellow region in Figure {\rm\ref{fig2}(c)}.
\item[$({2b}_3)$] $(\theta_{d_1,m_1},0)$ is linearly stable, and $(0,\theta_{d_2,m_2})$ is linearly unstable for all \bee(d_1,d_2)\in[(\mathsf{d}_1(d_2),\bar d_1]\times(\bar d_2,\infty)]\cup [(\bar d_1,\infty)\times(\widetilde{\mathsf{d}}_2(d_1),\infty)];\label{1.x17}\eee see the blue region in Figure {\rm\ref{fig2}(c)}.
\item[$({2b}_4)$] $(\theta_{d_1,m_1},0)$ is linearly unstable and  $(0,\theta_{d_2,m_2})$ is linearly stable for all
     \bee(d_1,d_2)\in(\bar d_1,\infty)\times(\mathsf{d}_2(d_1), \widehat{\mathsf{d}}_2(d_1));\label{1.x18}\eee  
see the black region of Figure {\rm\ref{fig2}(c)}.
\item[$({2b}_5)$] {\rm(}Unstable coexistence{\em)} Both $(\theta_{d_1,m_1},0)$ and $(0,\theta_{d_2,m_2})$ are linearly stable and system \eqref{1.5} has a linearly unstable positive solution for all  
    \bee (d_1,d_2)\in(\bar d_1,\infty)\times(\widehat{\mathsf{d}}_2(d_1), \widetilde{\mathsf{d}}_2(d_1));\label{1.x19}\eee  
see the red region in Figure {\rm\ref{fig2}(c)}.
\end{enumerate}\end{enumerate}\end{enumerate}\vspace{-5mm}
\end{theorem}
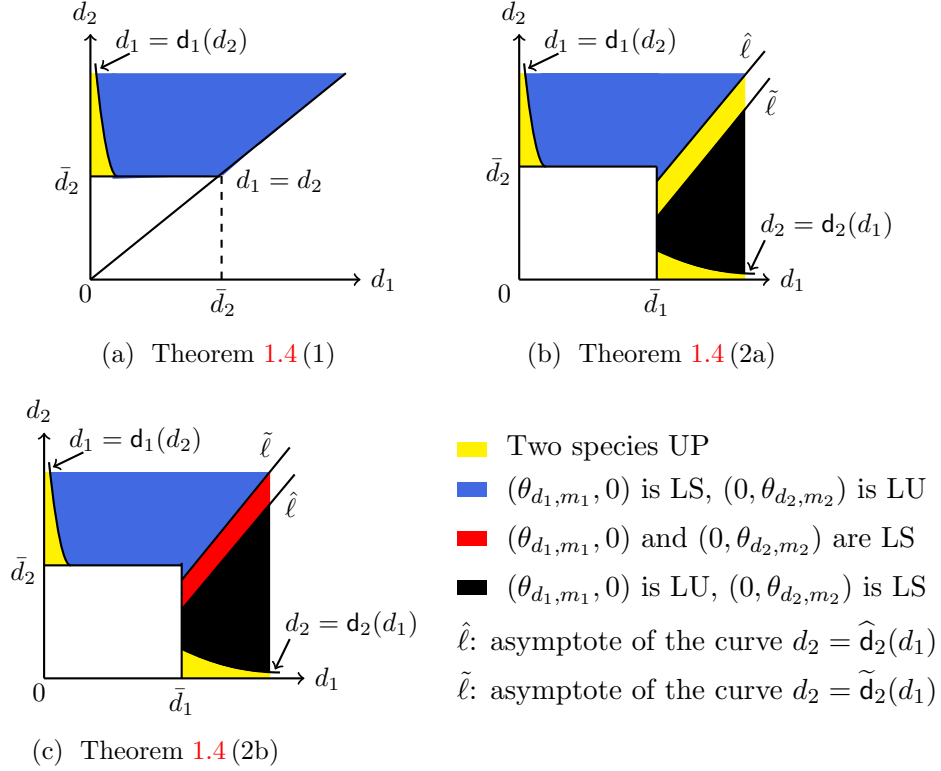
\begin{figure}[H]\centering
\begin{tabular}{cc}
\begin{tikzpicture}[scale=0.65][thick,fill opacity=1]
\fill[RoyalBlue] (0.34,3.1) parabola (-0.1,5.2)--(0.3,5.2)--(0.3,3.8)--cycle;
\fill[RoyalBlue] (0.3,3.8)--(3.27,3.8)--(2.5,3.12)--(0.25,3.05)--cycle;
\fill[yellow] (0.35,3.1) parabola (-0.1,5.2) --(-0.2,5.2)--(-0.2,3.1)--cycle;
\fill[RoyalBlue] (0.25,3.75)--(3.25,3.75)--(5.0,5.2)--(0.25,5.2)--cycle;
\draw[thick, black][->] (-0.2,1)--(5.3,1);\draw[thick, black][->] (-0.2,1)--(-0.2,6);
\node[right] at (5.3,1) {\small $d_1$};
\node[above] at (-0.3,6) {\small $d_2$};
\draw[thick, black] (-0.2,1)--(5,5.2);
 \draw[thick, dashed] (2.47,1)--(2.47,3.2);
\draw[thick, black] (-0.2,3.1)--(2.45,3.1);
\draw[thick, black](0.35,3.1) parabola (-0.1,5.4);
\draw[thick,->] (0.6,5.6)--(0,5.3);
\node[right] at (0.1,5.85) {\small $d_1=\mathsf{d}_1(d_2)$};
\node[below] at (-0.3,1.1) {\small $0$};
\node[below] at (2.5,1) {\small $\bar d_2$};
\node[left] at (-0.2,3.0) {\small $\bar d_2$};
\node[below] at (3.65,3.45) {\small $d_1=d_2$};
\end{tikzpicture}
&
\hspace{5mm}\begin{tikzpicture}[scale=0.65][thick, fill opacity=1]
\fill[RoyalBlue] (1.15,3.3) parabola (0.7,5.2) --(1.1,5.2)--(1.1,3.9)-- cycle;
\fill[RoyalBlue] (1.1,3.9)--(3.4,3.9)--(3.4,5.2)--(5.2,5.2)--(3.4,2.9)-- (3.4,3.3)--(1.05,3.3)--cycle;
\fill[RoyalBlue] (1.05,3.8)--(3.45,3.8) --(3.45,5.2)--(1.05,5.2)--cycle;
\fill[yellow] (1.15,3.3) parabola (0.7,5.2) --(0.6,5.2)--(0.6,3.3)--cycle;
\fill[yellow] (3.4,2.2)--(3.4,3)--(5.2,5.15)--(5.2,5.1)--(5.2,4.45)-- (3.4,2.2)--cycle;
\fill[black] (3.4,2.3)--(5.2,4.5)--(5.2,1.12) parabola (3.4,1.6)--cycle;
\fill[yellow] (3.4,1)--(5.2,1)--(5.2,1.12) parabola (3.4,1.6)--cycle;
\draw[thick, black][->] (0.6,1)--(5.7,1);\draw[thick, black][->](0.6,1)--(0.6,6);
\node[right] at (5.7,1) {\small $d_1$};
\node[above] at (0.5,6) {\small $d_2$};
\draw[thick, black] (0.6,3.3)--(3.4,3.3);
 \draw[thick, black] (3.4,1)--(3.4,3.3);
\draw[thick, black](1.14,3.3) parabola (0.7,5.4);
\draw[thick, black](5.4,1.12) parabola (3.4,1.6);
\draw[thick,->] (1.4,5.6)--(0.78,5.3);
\node[right] at (0.9,5.85) {\small $d_1=\mathsf{d}_1(d_2)$};
\draw[thick,->] (5.6,1.8)--(5.3,1.18);
\node[right] at (5.3,2.1) {\small $d_2=\mathsf{d}_2(d_1)$};
\draw[thick, black] (3.4,3)--(5.6,5.65);
\node[above] at (5.2,5.3) {\small $\hat\ell$};
\draw[thick, black] (3.4,2.2)--(5.7,5.1);
\node[right] at (5.4,4.6) {\small $\tilde\ell$};
\node[below] at (0.3,1.1) {\small $0$};
\node[below] at (3.4,1) {\small $\bar d_1$};
\node[left] at (0.65,3.2) {\small $\bar d_2$};
\end{tikzpicture}\\
{\small \hspace{-3mm} (a)\; Theorem \ref{th1.4}\,(1)} & \hspace{-5mm}{\small (b)\; Theorem \ref{th1.4}\,(2a)}\vspace{2mm}
\end{tabular}
\begin{tabular}{l}\begin{tikzpicture}[scale=0.65][thick,fill opacity=0.1]
\fill[RoyalBlue] (0.55,3.3) parabola (0.1,5.2) --(0.5,5.2)--(0.5,3.9)-- cycle;
\fill[RoyalBlue] (0.5,3.9)--(2.8,3.9)--(2.8,5.2)--(4.6,5.2)--(2.8,2.9)-- (2.8,3.3)--(0.45,3.3)--cycle;
\fill[RoyalBlue](0.45,3.85)--(2.85,3.85) --(2.85,5.2)--(0.45,5.2)--cycle;
\fill[yellow] (0.55,3.3) parabola (0.1,5.2) --(0,5.2)--(0,3.3)--cycle;
\fill[red] (2.8,2.45)--(2.8,3)--(4.6,5.2)--(4.6,4.7)--(4.6,4.5)-- (2.8,2.4)--cycle;
\fill[black] (2.8,2.4)--(4.6,4.55)--(4.6,1.12) parabola (2.8,1.6)--cycle;
\fill[yellow] (2.8,1)--(4.6,1)--(4.6,1.12) parabola (2.8,1.6)--cycle;
\draw[thick, black][->] (0,1)--(5.3,1);\draw[thick, black][->](0,1)--(0,6);
\node[right] at (5.3,1) {\small $d_1$};
\node[above] at (-0.1,6) {\small $d_2$};
\draw[thick, black] (2.8,1)--(2.8,3.35);
\draw[thick, black] (0,3.3)--(2.8,3.3);
\draw[thick,black](0.54,3.3) parabola (0.1,5.4);
\draw[thick, black](4.8,1.12) parabola (2.8,1.6);
\draw[thick, ->] (0.8,5.6)--(0.18,5.3);
\node[right] at (0.3,5.85) {\small $d_1=\mathsf{d}_1(d_2)$};
\draw[thick,->] (5,1.8)--(4.7,1.18);
\node[right] at (4.7,2.1) {\small $d_2=\mathsf{d}_2(d_1)$};
\draw[thick, black] (2.8,3)--(5.0,5.7);
\node[above] at (4.5,5.3) {\small $\tilde\ell$};
\draw[thick, black] (2.8,2.4)--(5.1,5.15);
\node[right] at (4.7,4.55) {\small $\hat\ell$};
\node[below] at (-0.1,1.1) {\small $0$};
\node[below] at (2.8,1.0) {\small $\bar d_1$};
\node[left] at (0.05,3.18) {\small $\bar d_2$};
\node[left] at (5,-0.5){\small (c)\; Theorem  \ref{th1.4}\,(2b)};
 \fill[yellow] (8.4,5.8)--(9,5.8)--(9,5.5)--(8.4,5.5)--cycle;
\fill[RoyalBlue] (8.4,5)--(9,5)--(9,4.7)--(8.4,4.7)--cycle;
\fill[red] (8.4,4)--(9,4)--(9,3.7)--(8.4,3.7)--cycle;
\fill[black] (8.4,3)--(9,3)--(9,2.7)--(8.4,2.7)--cycle;
\node[right] at (9.2,5.7) {Two species UP};
\node[right] at (9.2,4.8) {$(\theta_{d_1,m_1},0)$ is LS, $(0,\theta_{d_2,m_2})$ is LU};
\node[right] at (9.2,3.8) {$(\theta_{d_1,m_1},0)$ and $(0,\theta_{d_2,m_2})$ are LS};
\node[right] at (9.2,2.8) {$(\theta_{d_1,m_1},0)$ is LU, $(0,\theta_{d_2,m_2})$ is LS};
\node[right] at (8.2,1.8) {$\hat\ell$: asymptote of the curve $d_2=\widehat{\mathsf{d}}_2(d_1)$};
\node[right] at (8.2,0.8) {$\tilde\ell$: asymptote of the curve $d_2=\widetilde{\mathsf{d}}_2(d_1)$};
\end{tikzpicture}
\end{tabular}
\caption{\small Graphical Illustration of Theorems \ref{th1.3} and \ref{th1.4}. Abbreviations: UP=uniform persistence, LS=linearly stable, LU=linearly unstable.}
\label{fig2}
\end{figure}\vspace{-5mm}

\begin{remark}\label{rem1.2} Theorem \ref{th1.4}, as illustrated in Figure \ref{fig2}, describe the local and global dynamics of system \eqref{1.4} when one diffusion rate is large. In \ref{secB}, explicit examples of $m_1$ and $m_2$ are constructed to verify that all hypotheses of Theorem \ref{th1.4} are satisfied.
\begin{enumerate}
\item[{\rm(i)}] When $m_1(x,t)-m_2(x,t)$ is spatially homogeneous, the linear and global stability criteria for semi-trivial periodic solutions established in Theorem \ref{th1.3}(1) and Theorem \ref{th1.4}(1) coincide with those in \cite[Theorems 2.2 and 2.3]{BHN23} for the case $m_1(x,t)\equiv m_2(x,t)$.

\item[{\rm(ii)}] Theorem \ref{th1.4}($2a_2$) extends the uniform persistence results of {\rm\cite[Theorem 1.6]{HN17}} to the spatially dependent setting where $m_1(x,t)-m_2(x,t)$ varies with $x$; the latter result addresses the time-independent case with $m_i=m_i(x)$, $m_1\not\equiv m_2$, and $\overline{m}_1=\overline{m}_2$.

\item[{\rm(iii)}] Under the conditions $\overline{m}_1(t)\not\equiv\overline{m}_2(t)$ on $[0,T]$ and $\Pi<0$, Theorem \ref{th1.4}($2b_5$) reveals a novel bistability regime $(d_1,d_2)\in(\bar d_1,\infty)\times(\widehat{\mathsf{d}}_2(d_1),\widetilde{\mathsf{d}}_2(d_1))$, in which both semi-trivial periodic solutions are linearly stable. To the best of our knowledge, such a phenomenon has not been previously reported for Lotka-Volterra competition-diffusion systems {\rm(cf. \cite{BHN23,HMP01,HN13a,HN13b,HN16a,HN16b,HN17,LDC23})}.
\end{enumerate}
\end{remark}

The relative positions of the curves $\widehat{\mathsf{d}}_2$ and $\widetilde{\mathsf{d}}_2$ in Theorem \ref{th1.4} determine whether coexistence or bistability prevails for large $d_1$ and $d_2$.  However, the energy-estimate method used in \cite[Proposition 3.2]{BHN23} is insufficient for deriving the precise asymptotic expansions required in this large-diffusion regime. To characterize these curves, we therefore combine asymptotic analysis with the implicit function theorem to establish sharp asymptotic expansions for $\widehat{\mathsf{d}}_2$ and $\widetilde{\mathsf{d}}_2$ when diffusion rates are sufficiently large (see Lemmas \ref{lm4.7} and \ref{lm4.8}).

{\bf Asymptotic profiles}: Finally, we describe the shape of positive periodic solutions when one diffusion rate is small and the other is large.

\begin{theorem}[Asymptotic profiles]\label{th1.5} Assume that $\widehat{m}_1(x)\not\equiv\mathrm{constant}$ on $\boo$, and that conditions {\bf(M1)}--{\bf(M2)} and \qq{1.8} hold. Let $\mathsf{d}_1(d_2)$ and $\bar d_2$ be given in Theorem \ref{th1.4}. Then for all $(d_1,d_2)\in(0,\mathsf{d}_1(d_2))\times(\bar d_2,\infty)$, system \eqref{1.5} admits a linearly stable positive solution $(U,V)$, and
\bee\begin{cases}
\dd\lim_{d_2\to\infty}\lim_{d_1\to 0^+}
 \kk\|\int_0^T\kk[U-\kk(m_1-\min_{x\in\overline\Omega}\widehat{m}_1\rr)\rr]\mathrm dt\rr\|_{C(\overline\Omega)}=0,\\[4mm]
 \dd\lim_{d_2\to\infty}\lim_{d_1\to 0^+}\kk\|\int_0^T\kk(V-\min_{x\in\overline\Omega}\widehat{m}_1\rr)\mathrm dt\rr\|_{C(\overline\Omega)}=0.\end{cases}\label{1.14}
 \eee
\end{theorem}

Owing to the symmetry between $u$ and $v$, swapping $(d_1, m_1,u)$ and $(d_2, m_2,v)$ in Theorems \ref{th1.2}, \ref{th1.3} and \ref{th1.5} yields identical statements with the roles of the species reversed.

The remainder of this paper is organized as follows. Section~\ref{sec2} collects preliminary results on periodic-parabolic eigenvalue problems, stability criteria, and asymptotic profiles of single species solutions. Section~\ref{sec3} contains the proofs of Theorems~\ref{th1.1} and~\ref{th1.2} (small diffusion rates). Section \ref{sec4} provides several lemmas concerning the large-diffusion regime, which are essential for proving Theorems \ref{th1.3}-\ref{th1.5}. Section~\ref{sec5} is devoted to the proofs of Theorems~\ref{th1.3}--\ref{th1.5} (large diffusion rates). In Appendix, we construct explicit examples of $m_1$ and $m_2$ satisfying the conditions of Theorems~\ref{th1.3} and \ref{th1.4}.

\section{Preliminary results on the single $T$-periodic parabolic equation}\label{sec2}
\setcounter{equation}{0}\setlength\arraycolsep{2pt}

In this section, we introduce some basic results for subsequent discussion.

\subsection{Properties of principle eigenvalue and the associated positive eigenfunction}

In this subsection, we show some results on the principal eigenvalue and the corresponding eigenfunction of a time-periodic parabolic eigenvalue problem, as well as its adjoint eigenvalue problem.

Let $\mu(d,m)$ be the principal eigenvalue of \qq{1.3}.
Then it is also the principal eigenvalue of the adjoint eigenvalue problem of \eqref{1.3} as follows:
  \bee\label{2.1}\begin{cases}
-\psi_t-d\Delta\psi-m(x,t)\psi=\mu\psi &\text{in}\;\;Q_T,\\
\partial_{\nu}\psi=0&\text{on}\;\;S_T,\\
\psi(x,0)=\psi(x,T)&\text{on}\;\;\ol\Omega,
\end{cases}
 \eee
which possesses a unique positive adjoint eigenfunction $\psi(\cdot;d,m)\in C^{2+\alpha,1+\alpha/2}_T(\overline Q_T)$ up to a multiplicative constant. The following proposition is due to \cite[Theorem 2.4]{HMP01}, \cite[Lemma 2.15]{CC03} and \cite[Lemma 15.3]{H91} (see also \cite{LL23}).

\begin{proposition}\label{lm2.1} The following hold for the principal eigenvalue $\mu(d,m)$ of \eqref{1.3}.
\begin{enumerate}\setlength{\itemsep}{4pt}
\item[\rm (1)] There hold: $-\max_{\overline\Omega}\widehat{m}(x)\leq \mu(d,m)\leq-\widehat{\ol m}$, and $\mu(d,m)=-\widehat{\ol m}$ if and only if $m$ is spatially independent;
\item[\rm (2)] $\mu(d,m)$ depends continuously on $d>0$ and $m$, and is decreasing in $m$, i.e., $m_1\geq m_2$ implies $\mu(d,m_1)\leq \mu(d,m_2)$, where strict inequality holds  when $m_1\geq m_2$ and $m_1\not\equiv m_2$;
\item[\rm (3)] $\dd\lim_{d\to 0}\mu(d,m)=-\max_{\overline\Omega}\widehat{m}(x)$ and $\dd\lim_{d\to\infty}\mu(d,m)=-\widehat{\ol m}$;
\item[\rm (4)] If $V(t)\in L^{\infty}((0,T))$ satisfies $V(0)=V(T)$, then  $\mu(d,m+V)=\mu(d,m)-\frac{1}{T}\int_0^TV(t)\mathrm dt$.
\end{enumerate}
\end{proposition}

In the following, we characterize the asymptotic properties of $\mu(d,m)$ and its corresponding eigenfunctions $\varphi$ and $\psi$. The following proposition plays a key role in studying the dynamics of \eqref{1.4} when the diffusion rate $d_2$ is sufficiently large.

\begin{proposition}[\!{\rm\cite[Proposition 3.2]{BHN23}}]\label{p2.2}
Let $m\in C_T^{\alpha,\alpha/2}(\overline Q_T)$ and $m_t\in L^2(Q_T)$ satisfy
  \[\|m\|_{L^{\infty}(Q_T)}+\|m_t\|_{L^2(Q_T)}\leq M\]
for some constant $M>0$. Let $\varphi(\cdot;d,m)$ and $\psi(\cdot;d,m)$ be the principal eigenfunctions of \eqref{1.3} and \eqref{2.1} corresponding to $\mu(d,m)$ respectively with normalization: $\|\varphi\|_{L^2(Q_T)}=\|\psi\|_{L^2Q_T)}=1$.
Then there exist constants $d_M,C_M>0$ such that for any $d>d_M$, $\varphi$ and $\psi$ can be rewritten as
  \begin{align*}
\varphi(x,t;d,m)&=\ol\varphi(t;d,m)+\ol\varphi(t;d,m)
\Gamma_m(x,t)d^{-1}+\varphi_*(x,t;d,m)d^{-2},\\
 \psi(x,t;d,m)&=\ol\psi(t;d,m)+\ol\psi(t;d,m)\Gamma_m(x,t)d^{-1}
+\psi_* (x,t;d,m)d^{-2},
\end{align*}
where $\ol\varphi(t;d,m)$ and $\ol\psi(t;d,m)$ satisfy
  \begin{align*}
\displaystyle\ol\varphi(t;d,m)&=\ol\varphi(0)\exp\kk(\int_0^t(\ol m(s)+\mu)\rr)+ O(d^{-1/2}),\\
\displaystyle\ol\psi(t;d,m)&=\ol\psi(0)\exp\kk(-\int_0^t(\ol m(s)+\mu)\rr)+O(d^{-1/2}),
  \end{align*}
and
\[1/C_M\leq \ol\varphi(t;d,m),\;\ol\psi(t;d,m)\leq C_M, \;  \forall\,t\in[0,T],\]
while $\varphi_*(x,t;d,m)$ and $\psi_*(x,t;d,m)$ satisfy $\|\varphi_*\|_{L^2((0,T),H^1(\Omega))},\; \|\psi_*\|_{L^2((0,T),H^1(\Omega))}\leq C_M$.

Moreover, for each constant $\kappa>0$, there exists a constant $d_{\kappa,M}>0$ depending only on $\kappa$ and $M$ such that when $m$ further satisfies $\|\nabla\Gamma_m\|_{L^2(Q_T)}\geq \kappa$, there holds:
  \[\frac{\partial\mu(d,m)}{\partial d}>0,\; \forall\,d>d_{\kappa,M}. \]
\end{proposition}

\subsection{Stabilities of nonnegative $T$-periodic solutions}\label{sec2.2}

In this subsection, we state some criteria concerning the linear and global stability of nonnegative solutions to \eqref{1.5}. Similar to \cite[Lemma 3.2]{HMP01}, we have the following conclusions.\vspace{-1mm}

\begin{proposition}\label{p2.3} If $\mu(d_2,m_2-\theta_{d_1,m_1})>0\; ($resp. $<0)$, then  $(\theta_{d_1,m_1},0)$ is linearly stable $($resp. unstable$)$. If $\mu(d_1,m_1-\theta_{d_2,m_2})>0\; ($resp. $<0)$, then $(0,\theta_{d_2,m_2})$ is linearly stable $($resp. unstable$)$.
\end{proposition}\vspace{-1mm}

Combined Proposition \ref{p2.3} with the theory of monotone semi-flows (cf. \cite{H91}), this  yields the following results on the global dynamics of system \eqref{1.4}.\vspace{-.1mm}

\begin{proposition}\label{p2.4} The following statements hold.
\begin{enumerate}
\item[\rm(1)] If $\mu(d_2,m_2-\theta_{d_1,m_1})>0$, $\mu(d_1,m_1-\theta_{d_2,m_2})<0$ and \qq{1.5} has no positive solution, then $(\theta_{d_1,m_1},0)$ is globally asymptotically stable {\rm(cf. \cite[Theorem 34.1]{H91})}.

\item[\rm(2)] If the inequalities in {\rm(1)} are reversed and \qq{1.5} has no positive solution, then $(0,\theta_{d_2,m_2})$ is globally asymptotically stable {\rm(cf. \cite[Theorem 34.1]{H91})}.

\item[\rm(3)] If both principal eigenvalues are negative, then \eqref{1.4} is uniformly persistent and \qq{1.5} possesses a linearly stable positive solution {\rm(cf. \cite[Remark 33.2 and Theorem 33.3]{H91})}.

\item[\rm(4)] If both principal eigenvalues are positive, then \qq{1.5} has a linearly unstable positive solution {\rm (cf. \cite[Theorem 35.1]{H91}).}
\end{enumerate}
\end{proposition}

\subsection{Asymptotic profiles of the periodic solution $\theta_{d,m}(x,t)$} \label{sec2.3}
{\setlength\arraycolsep{2pt}

In this subsection, we present several basic results on the asymptotic profiles of $\theta_{d,m}(x,t)$. To this end, we first introduce solutions to specific equations and analyze their key properties.

Let $d>0$ and $m\in C_T^{\alpha,\alpha/2}(\overline Q_T)$. We denote by $\Theta_{d,m}$ and $\Theta_m$ the maximal nonnegative solutions of \eqref{1.2} and \eqref{1.6}, respectively. Then either $\Theta_{d,m}\equiv 0$ or $\Theta_{d,m}=\theta_{d,m}>0$ on $\overline Q_T$.

\begin{lemma}\label{lm2.5} Let $d>0$ and $m\in C_T^{\alpha,\alpha/2}(\overline Q_T)$. Then we have
 \bee{\lim_{d\to 0}\|\Theta_{d,m}-\Theta_m\|_{C(\overline Q_T)}=0.}
 \label{2.2}\eee
\end{lemma}

\begin{proof}\; If $\max_{x\in\overline\Omega}\widehat m\le0$, then $\Theta_m\equiv 0$, and $\mu(d,m)\ge 0$ by Proposition \ref{lm2.1}(1); hence $\Theta_{d,m}\equiv 0$. The limit \eqref{2.2} holds.

In the following we consider the case where $\max_{x\in\overline\Omega}\widehat m>0$.

{\it Step 1}. Using Proposition \ref{lm2.1}(3), we have $\lim_{d\to 0}\mu(d,m)=-\max_{x\in\overline\Omega} \widehat m<0$. Then there exists $d^*>0$ such that $\mu(d,m)<0$ for all $0<d\le d^*$. Consequently, $\Theta_{d,m}=\theta_{d,m}>0$ on $\overline Q_T$ for all $0<d\le d^*$.

{\it Step 2}. For each $x\in\overline\Omega$, since $m\in C_T^{\alpha,\alpha/2}(\overline Q_T)$ and
 \[\int_0^T\bigg(m-2\Theta_m+\frac{1}{T}\int_0^T(2\Theta_m-m)\bigg)\dt=0,\]
problem
\begin{equation}\begin{cases}\label{2.3}
P_t=\left(m-2\Theta_m+\dfrac{1}{T}\int_0^T(2\Theta_m-m)\right)P,\quad t\in(0,T],\\[2mm]
\dd\min_{t\in[0,T]}P(x,t)=1\end{cases}
\end{equation}
admits a unique positive $T$-periodic solution $P\in C_T^{\alpha,\alpha/2}(\overline Q_T)$.

We show that for each fixed $0<\varepsilon<1$, there exists $d_*=d_*(\varepsilon)\in(0,d^*)$ such that for all $0<d<d_*$,
  \begin{equation}\label{2.4}
\theta_{d,m}(x,t)\le \Theta_m(x,t)+\ep P(x,t)+\frac{\ep^2}{4B_0},
\quad \forall\, (x,t)\in\overline Q_T,
  \end{equation}
where
 \[B_0:=3M_0+3,\;\;\;M_0:=\max\Bigl\{1,\,\max_{\overline Q_T}|m|,\,\max_{\overline Q_T}P\Bigr\}.\]
Once this is done, letting $d\to0$ firstly and $\ep\to0$ secondly, we obtain
  \begin{equation}
\limsup_{d\to0}\theta_{d,m}(x,t)\le \Theta_m(x,t)\;\;\;\text{uniformly on}\;\; \ol Q_T.
\label{2.5}
\end{equation}

Indeed, owing to the fact that $\Theta_m\in C_T^{\alpha,\alpha/2}(\overline Q_T)$, by \cite[Proposition 4.4.8]{L95}, the periodic parabolic problem
\begin{equation}\begin{cases}\label{2b.6}
(u_d)_t-d\Delta u_d+u_d=\Theta_m(m-\Theta_m)+\Theta_m
& \text{in }\ Q_T,\\
\partial_{\nu}u_d=0 & \text{on }\ S_T,\\
u_d(x,0)=u_d(x,T) & \text{on }\ \overline\Omega
\end{cases}\end{equation}
admits a unique strong solution $u_d$, and hence $u_d$ is a classical solution by parabolic regularity. Moreover, by \cite[Lemma 4.2]{BHN23},
\begin{equation}\label{2b.7}
\lim_{d\to0}\|u_d-\Theta_m\|_{C(\overline Q_T)}=0.
\end{equation}

Next, consider the periodic parabolic problem
\begin{equation}\begin{cases}\label{2b.8}
(P_d)_t-d\Delta P_d+P_d=P_t+P
& \text{in }\ Q_T,\\
\partial_{\nu}P_d=0 & \text{on }\ S_T,\\
P_d(x,0)=P_d(x,T) & \text{on }\ \overline\Omega.
\end{cases}\end{equation}
Again by \cite[Proposition 4.4.8]{L95}, \eqref{2b.8} admits a unique strong solution $P_d$,
which is also classical. Moreover, by \cite[Lemma 4.2]{BHN23},
\begin{equation}\label{2b.9}
\lim_{d\to0}\|P_d-P\|_{C(\overline Q_T)}=0.
\end{equation}

Fix $0<\varepsilon<1$. By \eqref{2b.7} and \eqref{2b.9}, shrinking $d_*>0$ if necessary,
we may assume that for all $0<d<d_*$,
\begin{equation}\label{2b.10}
\|u_d-\Theta_m\|_{C(\overline Q_T)}\le \min\left\{\frac{\ep^2}{16B_0},\,\frac{\ep}{32M_0}\right\},
\end{equation}
and
 \begin{equation}\label{2b.11}
\|P_d-P\|_{C(\overline Q_T)}\le \min\left\{\frac12,\,\frac{\ep}{16B_0}\right\}.
\end{equation}
Since $0\le \Theta_m\le M_0$ and $P\ge 1$ on $\overline Q_T$, the estimates \eqref{2b.10}--\eqref{2b.11} imply that
\begin{equation}\label{2b.12}
|u_d|\le M_0+1,\quad \frac12\le P_d\le M_0+1
\quad \text{on }\; \overline Q_T.
\end{equation}

We now construct an upper solution with the form 
 \[w=u_d+\ep P_d\;\;\;\;\text{for all}\;\; 0<d<d_*.\]
Notice that $\widehat\Theta_m(x)=\widehat m(x)$ if $\widehat m(x)>0$, and $\Theta_m(x,t)\equiv 0$ if $\widehat m(x)\leq 0$. It is clear that  
  \bee\int_0^T(2\Theta_m-m)\geq 0\;\;\;\;\text{for every}\;\;x\in\overline\Omega.
   \label{2x.13}\eee
In view of \eqref{2b.10} and \eqref{2b.12}, there holds that 
 \begin{align}
 |\Theta_m(m-\Theta_m)-u_d(m-u_d)|&\leq (|m|+|u_d|+\Theta_m)|u_d-\Theta_m| \notag \\
&\leq B_0\,|u_d-\Theta_m|\le \frac{\ep^2}{16}
\quad \text{on }\; \overline Q_T  \label{2b.13}
  \end{align}
for $0<d<d_*$. From \eqref{2.3}, \eqref{2b.6} and \eqref{2b.8} we have
 \begin{align*}
&w_t-d\Delta w-w(m-w)\\
=\;&(u_d)_t-d\Delta u_d-u_d(m-u_d)+\ep\bigl[(P_d)_t-d\Delta P_d-(m-2u_d)P_d\bigr]+\ep^2P_d^2\\
=\;&\Theta_m(m-\Theta_m)+\Theta_m-u_d-u_d(m-u_d)+\ep\big[(P_t+P)-(m-2u_d+1)P_d \big]+\ep^2P_d^2 \\
 =\;&\Theta_m(m-\Theta_m)-u_d(m-u_d)+\Theta_m-u_d+\ep(1+m-2u_d)(P-P_d)\\
 &+\ep\bigg(2(u_d-\Theta_m)+\frac{1}{T}\int_0^T(2\Theta_m-m)\bigg)P+\ep^2P_d^2 \quad \text{in }\;  Q_T
\end{align*}
for $0<d<d_*$. Combining this with \eqref{2b.10}--\eqref{2b.13} and the fact that $P\le M_0$, we obtain 
  \begin{align*}
w_t-d\Delta w-w(m-w)
\ge\;&-\frac{\ep^2}{16}-\ep B_0|P_d-P|-(2\ep M_0+1)|u_d-\Theta_m|+\ep^2P_d^2\\
\ge\;&-\frac{\ep^2}{16}-\frac{\ep^2}{16}-\frac{\ep^2}{16}+\frac{\ep^2}{4}
>0 \quad \text{in } \; Q_T.
  \end{align*}
Clearly, $\partial_\nu w=0$ on $S_T$, and $w(x,0)=w(x,T)$ on $\overline\Omega$. Thus $w$ is a strict upper solution of
\eqref{1.2}.

Furthermore, by \eqref{2b.10} and \eqref{2b.12},
 \[w=u_d+\ep P_d=\Theta_m+u_d-\Theta_m+\ep P_d\geq -|u_d-\Theta_m|+\frac{\ep}{2}\geq \frac{\ep}{2}-\frac{\ep}{32M_0} >0\quad \text{on }\;\overline Q_T.\]
Let $\varphi_{d,m}$ be the positive eigenfunction corresponding to $\mu(d,m)$. Taking
$\sigma>0$ sufficiently small, we can check that $\sigma\varphi_{d,m}$ is a lower solution
of \eqref{1.2} and $\sigma\varphi_{d,m}\leq w$ on $\overline Q_T$. Since $\sigma\varphi_{d,m}(x,0)\in C^{2+\alpha}(\boo)$, by the upper and lower solution
method (cf. \cite[Theorem 7.3]{W21}) and the uniqueness of maximal solutions to \eqref{1.2},
we conclude that $\theta_{d,m}\le w=u_d+\ep P_d$ on $\overline Q_T$. Finally, by \eqref{2b.10} and \eqref{2b.11}, it yields 
  \begin{align*}
\theta_{d,m}\le u_d+\ep P_d\le\Theta_m+\ep P+\frac{\ep^2}{16B_0}+\frac{\ep^2}{16B_0}
\le \Theta_m+\ep P+\frac{\ep^2}{4B_0}\quad \text{on }\; \overline Q_T.
\end{align*}
This proves \eqref{2.4}.

{\it Step 3}. We show that
  \bee
 \liminf_{d\to0}\theta_{d,m}(x,t)\ge \Theta_m(x,t)\;\;\;\text{uniformly on}\;\; \ol Q_T.
 \label{2.7}\eee
Assume for contradiction that there exist $\varepsilon_0>0$, $d_k\to 0$, and $(x_k,t_k)\in\overline{Q}_T$ such that
 \begin{equation}\label{2.8}
 \Theta_m(x_k,t_k)-\theta_{d_k,m}(x_k,t_k)\ge\varepsilon_0,\;\;\forall\, k.
 \end{equation}
We can assume $(x_k,t_k)\to(x_*,t_*)\in\overline{Q}_T$. Then, by \eqref{2.8}, $\Theta_m(x_*,t_*)\geq\varepsilon_0$. Thus $\widehat m(x_*)>0$.

If $x_*\in\oo$, then there exists a closed ball $\ol B_r(x_*)\subset\oo$ such that $\widehat m(x)>0$ on $\ol B_r(x_*)$. By \cite[Theorem 1.1]{SSJ24}, $\lim_{d\to 0}\theta_{d,m}=\Theta_m$ uniformly on $\ol B_r(x_*)\times[0,T]$. This contradicts \eqref{2.8}.

Now we consider the case $x_*\in\partial\Omega$. We may assume that $x_*=0$, and use the boundary flattening transformation to handle it. Write $x=(x',x^n)$,  $0=(0',0^{n})$, and let $B^{n-1}_r(0')$ be the ball in $\mathbb{R}^{n-1}$ centered at $0'$ with radius $r$. Because $\partial\Omega\in C^{2+\alpha}$, there exist $r>0$, a $C^{ 2+\alpha}$ function $f: B^{n-1}_r(0')\to \mathbb{R}$, and a suitable Cartesian coordinate system (which can be obtained by translating and rotating the coordinate system) such that, after this change of coordinates:\vspace{-1mm}
\begin{enumerate}[leftmargin=4.5mm]
\item[$\bullet$] $f(0')=0$ and $\nabla_{x'}f(0')=0$, and
  \begin{align*}
\partial\Omega\cap B_r(0)=\big\{x\in B_r(0): x^n=f(x')\big\},\;\;
\Omega\cap B_r(0)=\big\{x\in B_r(0):x^n>f(x')\big\}.
\end{align*}
\end{enumerate}
Then, for all sufficiently large $k$, the sequence $x_k$ satisfies $x^n_k\ge f(x'_k)$.

Define the flattening map
 \[F(x)=(x',x^n-f(x')),\;\;x\in\Omega\cap B_r(0).\]
Then the function $y=F(x)$ has an inverse function $x=H(y)$ defined in $B_{2\rho}^+=B_{2\rho}(0)\cap\{y^n>0\}$ for some $0<\rho\ll1$. Define 
 \[\tilde w_k(y,t)=\theta_{d_k,m}\big(H(y),t\big).\] 
 Then $\tilde w_k$ satisfies
  \[\begin{cases}
  \displaystyle \partial_t\tilde w_k=d_k\left(\sum_{i,j=1}^na_{ij}\frac{\partial^2\tilde w_k}{\partial y^i\partial y^j}\right.+\left.\dd\sum_{i=1}^n b_i\frac{\partial\tilde w_k}{\partial y^i}\right)\!+\!\tilde w_k\big( m(H(y),t)\!-\!\tilde w_k \big), \;\;
  y\in B_{2\rho}^+,\; t\in(0,T],\\[5mm]
 \displaystyle\sum_{j=1}^n a_{nj}(y)\frac{\partial\tilde w_k}{\partial y^j}=0, \hspace{6mm}y\in B_{2\rho}\cap\{y^n=0\},\;t\in(0,T],\\
\tilde w_k(y,0)=\tilde w_k(y,T), \hspace{5mm}y\in\overline{B_{2\rho}^+},
 \end{cases}\]
where 
\[a_{ij}(y)=\sum_{q=1}^n\frac{\partial F_i}{\partial x^{q}}\big(H(y)\big)\frac{\partial F_j}{\partial x^{q}}\big(H(y)\big),\;\;
  b_i(y)=\Delta_x F_i\big(H(y)\big),\;\; 1\leq i,j\leq n.\] 
Now set 
 \[y_k=F(x_k).\] 
Then $y_k^n=x_k^n-f(x_k')\geq 0$, and $y_k\in B_{\rho}(0)$ for all sufficiently large $k$ since $x_k\to 0$. This naturally leads to a dichotomy between two distinct scenarios.

{\it Case} (i): the sequence $\big\{y_k^n/\sqrt{d_k} \big\}_{k=1}^{\infty}$ is bounded. By extracting a subsequence if necessary, we may assume that $y_k^n/\sqrt{d_k}\to \chi\geq 0$ as $k\to\infty$. Define
 \[\check w_k(z,t)=\tilde w_k\big(y_k'+\sqrt{d_k}z',\sqrt{d_k}z^n,t\big),\;\;
(z,t)\in B_{\rho/\sqrt{d_k}}^+\times[0,T].\]
Then $\check w_k$ satisfies
  \begin{equation}\label{2.9}
  \begin{cases}
\displaystyle \partial_t\check w_k=\mathcal L_k\check w_k+\check w_k(\check m^k(z,t)-\check w_k), & z\in B_{\rho/\sqrt{d_k}}^+,\; t\in(0,T], \\
 \displaystyle\sum_{j=1}^n\check a^k_{nj}(z)\frac{\partial\check w_k}{\partial z^j}=0, & z\in B_{\rho/\sqrt{d_k}}\cap\{z^n=0\},\; t\in(0,T], \\
\check w_k(z,0)=\check w_k(z,T), & z\in\overline{B_{\rho/\sqrt{d_k}}^+},
\end{cases}
\end{equation}
where
  \begin{align*}
  \mathcal L_k&=\sum_{i,j=1}^n\check a^k_{ij}(z)\frac{\partial^2}{\partial z^i\partial z^j}+\sqrt{d_k}\sum_{i=1}^n\check b^k_i(z)\frac{\partial}{\partial z^i},\\
\check a^k_{ij}(z)&=a_{ij}\big(y_k'+\sqrt{d_k}z',\sqrt{d_k}z^n\big),\\[0.2mm]
\check b^k_i(z)&=b_i\big(y_k'+\sqrt{d_k}z',\sqrt{d_k}z^n \big), \\[0.2mm]
\check m^k(z,t)&=m(H(y_k'+\sqrt{d_k}z',\sqrt{d_k}z^n),t).
\end{align*}
Since $f(0')=0$ and $\nabla_{x'} f(0')=0$, it is straightforward to verify that
  \begin{align}
 \lim_{k\to\infty}\check a^k_{ij}(z)=\delta_{ij}\,\;\text{(the Kronecker delta function)\;\;locally uniformly in }\; \mathbb R^n_+,\label{2.10}\end{align}
and for any given $R>0$,
  \begin{align}
 \|\check a^k_{ij}\|_{C^\alpha(\ol B_R^+)},\;\|\check b^k_i\|_{C^\alpha(\ol B_R^+)},\;
  \|\check m^k\|_{C^{\alpha,\frac\alpha 2}(\ol B_R^+\times[0,T])}\;\; \text{are bounded in }\; k. \label{2.11}\end{align}

For any $R>0$, we define
  \[z^0=(0,\dots,0,2R)\in \mathbb R^n_+,\qquad \mathcal{O}=B_R(z^0).\]
It is clear that $\mathcal{O}\Subset \mathbb R^n_+$. Let $\lambda$ denote the principal eigenvalue of $-\Delta$ in $\mathcal{O}$ subject to the homogeneous Dirichlet boundary condition. In view of $\widehat{m}(0)>0$, we may fix a constant $\delta\in\bigl(0,\widehat{m}(0)/8\bigr)$ and choose a large $R$ such that $\lambda<\widehat{m}(0)-4\delta$. Since $B_{\rho/\sqrt{d_k}}^+\nearrow\mathbb R^n_+$, the inclusion $\mathcal{O}\subset B_{\rho/\sqrt{d_k}}^+$ holds for all large $k$.
For each such $k$, let $\lambda_k$ be the principal eigenvalue of the elliptic problem
  \[\begin{cases}
-\mathcal L_k\phi_k=\lambda_k\phi_k, & z\in \mathcal O,\\
\phi_k=0, & z\in \partial\mathcal O,
\end{cases}\]
and let $\phi_k\in C^{2+\alpha}(\overline{\mathcal O})$ be the corresponding positive eigenfunction with normalization $\phi_k(z^0)=1$. By the uniform convergence $\check{a}^k_{ij}\to\delta_{ij}$ and $\sqrt{d_k}\check{b}^k_i\to0$ on $\overline{\mathcal O}$, the operator $\mathcal{L}_k$ converges uniformly to $\Delta$ on $\overline{\mathcal O}$, which implies $\lambda_k\to \lambda$ as $k\to\infty$. Consequently,
  \[\lambda_k<\widehat{m}(0)-3\delta \quad \text{for all large } k.\]
Furthermore, by the Harnack inequality and global Schauder estimates, there exists a constant $M_\phi>0$ independent of $k$ such that $\|\phi_k\|_{C(\overline{\mathcal O})}\leq M_\phi$ for all large $k$. Let $\eta(t)$ be the unique positive $T$-periodic solution to the scalar ODE
  \[\eta_t=\big(m(0,t)-\widehat{m}(0)\big)\eta,\quad \min_{t\in[0,T]}\eta(t)=1,\]
and set $M_\eta=\max_{t\in[0,T]}\eta(t)$. We choose $\sigma>0$ small so that $\sigma M_\eta M_\phi\leq \delta$, and define 
 \[v_k(z,t):=\sigma\eta(t)\phi_k(z),\quad (z,t)\in \overline{\mathcal O}\times[0,T].\]
Since $\check{m}^k(z,t)\to m(0,t)$ uniformly on $\overline{\mathcal O}\times[0,T]$ as $k\to\infty$, it follows that for all large $k$,
  \[|\check m^k(z,t)-m(0,t)|\le\delta\quad \text{on }\;\overline{\cal O}\times[0,T].\]
A direct calculation yields
 \begin{align}
 (v_k)_t-\mathcal L_k v_k-v_k(\check m^k-v_k)
&=\sigma\eta\phi_k\bigl(\lambda_k+m(0,t)-\widehat m(0)-\check m^k+\sigma\eta\phi_k\bigr)   \notag\\
&\le \sigma\eta\phi_k\bigl(\lambda_k-\widehat m(0)+\delta+\sigma M_{\eta} M_{\phi}\bigr)  \notag\\
&\le -\delta\sigma\eta\phi_k<0
\quad \text{in }\; {\cal O}\times(0,T]. \label{2.12}
  \end{align}

Notice that $\check w_k, v_k\in C(\overline {\cal O}\times[0,T])$, with $\check w_k>0$ and $v_k\ge 0$. There exists $\gamma>0$ such that $\check w_k\ge \gamma v_k$ on $\overline {\cal O}\times[0,T]$. Define
  \[\gamma_k=\sup\bigl\{\gamma>0:\ \check w_k\ge \gamma v_k
 \;\;\; \text{on}\;\;\overline {\cal O}\times[0,T]\bigr\}.\]
Then $0<\gamma_k<\infty$, as $v_k>0$ in ${\cal O}\times(0,T]$. We assert that $\gamma_k>1$.  If not, then $\gamma_k\le 1$. Set
  \[q_k=\check w_k-\gamma_kv_k.\]
Then $q_k\geq 0$ on $\overline {\cal O}\times[0,T]$, and there exists $(\bar z_k,\bar t_k)\in\overline {\cal O}\times[0,T]$ such that $q_k(\bar z_k,\bar t_k)=0$. Since $v_k=0$ and $\check w_k>0$ on $\partial {\cal O}\times[0,T]$, we have $\bar z_k\notin\partial {\cal O}$. Viewing the time variable on the torus $\mathbb T=\mathbb R/T\mathbb Z$, we conclude that $(\bar z_k,\bar t_k)\in {\cal O}\times\mathbb T$ is an interior minimum point of $q_k$. Consequently 
$(q_k)_t(\bar z_k,\bar t_k)=0$, $\nabla_z q_k(\bar z_k,\bar t_k)=0$, and $D_z^2 q_k(\bar z_k,\bar t_k)\ge0$. This implies
  \begin{equation}\label{2.13}
(q_k)_t-\mathcal L_k q_k\le0\quad \text{at }\; (\bar z_k,\bar t_k).
\end{equation}

On the other hand, using $\check w_k(\bar z_k,\bar t_k)=\gamma_kv_k(\bar z_k,\bar t_k)$ together with the assumption $0<\gamma_k\le 1$, we deduce from \eqref{2.9} and \eqref{2.12} that, at $(\bar z_k,\bar t_k)$, 
 \begin{align*}
 (q_k)_t-\mathcal L_k q_k&=\check w_k(\check m^k-\check w_k)-\gamma_k((v_k)_t-\mathcal L_k
  v_k)\nonumber\\
  &> \check w_k(\check m^k-\check w_k)-\gamma_kv_k(\check m^k-v_k)\\
  &=\check w_k(1-\gamma_k)v_k\ge 0.\end{align*}
This contradicts \eqref{2.13}. Hence $\gamma_k>1$, and so $\check w_k>v_k$ on $\overline{\cal O}\times[0,T]$ for all sufficiently large $k$. Particularly,
  \begin{equation}\label{2.15}
  \check w_k(z^0,t)>\sigma \eta(t),\;\; \forall\; t\in[0,T],\, k\gg1.
  \end{equation}
  
For any given $R>0$, observing \eqref{2.10} and \eqref{2.11}, and applying the local $L^p$ estimates firstly and the local Schauder estimates secondly for $B_R^+\times[T/2,T]\subset B^+_{\rho/\sqrt{d_k}}\times[0,T]$, we deduce that $\|\check w_k\|_{C^{2+\alpha,1+\frac\alpha 2}(\ol B_R^+\times[0,T])}$ is uniformly bounded in $k$. There exist $\check w_R\in C^{2,1}(\ol B_R^+\times[0,T])$ and a subsequence of $\{\check w_k\}$, denoted by itself, such that $\check w_k\to\check w_R$ in $C^{2,1}(\ol B_R^+\times[0,T])$. Choose $R_l\to\infty$. Using the diagonal argument, we can find $\check w_{\infty}\in C^{2,1}(\ol{\mathbb R}^n_+\times[0,T])$  and a subsequence of $\{\check w_k\}$, denoted by itself, such that $\check w_k\to\check w_{\infty}$ in $C^{2,1}_{\rm loc}(\ol{\mathbb R}^n_+\times[0,T])$. Letting $k\to\infty$ in \eqref{2.9}, and noting that $d_k\to 0$ while $\check a_{ij}^k(z)\to\delta_{ij}$, we find that $\check w_{\infty}$ satisfies the limiting problem
 \[\begin{cases}
\partial_t\check w_{\infty}=\Delta\check w_{\infty}+\check w_{\infty}(m(0,t)-\check w_{\infty}), & z^n>0,\; t\in(0,T],\\
\partial_{z^n}\check w_{\infty}=0, &z^n=0,\; t\in(0,T],\\
\check w_{\infty}(z,0)=\check w_{\infty}(z,T), & z^n\ge 0.
  \end{cases}\]
Letting $k\to\infty$ in \eqref{2.15}, we obtain $\check w_\infty(z^0,t)\ge \sigma\eta(t)$ for all $t\in[0,T]$. Hence, $\inf_{t\in\mathbb R}\check w_\infty(z^0,t)>0$ by the time-periodicity of $\check w_\infty$. Finally, extending $\check w_{\infty}$ to the whole space $\mathbb R^n$ by even reflection across the hyperplane $\{z^n=0\}$, we assert  that $\check w_{\infty}$ is a bounded nonnegative solution of
  \begin{equation}\label{2.16}
  \begin{cases}
\partial_t\check w_{\infty}=\Delta \check w_{\infty}+\check w_{\infty}\big(m(0,t)-\check w_{\infty} \big) & \text{in } \mathbb R^n\times(0,T],\\
\check w_{\infty}(z,0)=\check w_{\infty}(z,T) & \text{in }\mathbb R^n,
 \end{cases}  \end{equation}
As $\widehat m(0)>0$, the kinetic equation \eqref{1.6} with $x=0$ admits a unique positive $T$-periodic solution $\Theta_m(0,t)$. By \cite[Theorem 1.5]{N09}, the condition $\inf_{t\in\mathbb R}\check w_\infty(z^0,t)>0$ implies $\inf_{(z,t)\in\mathbb R^n\times\mathbb R}\check w_\infty(z,t)>0$. Now $\check w_\infty$ is a bounded entire solution of \eqref{2.16}, and $\Theta_m(0,t)$
is also a bounded entire solution of \eqref{2.16} independent of $z$. By \cite[Theorem 1.3]{N09},
  \[\check w_\infty(z,t)\equiv \Theta_m(0,t)\quad \text{in }\; \mathbb R^n\times[0,T].\]
In the $z$-variables, the original point $x_k$ corresponds to $z_k^*:=\left(0',{y_k^n}/{\sqrt{d_k}}\right)\to (0',\chi)$. Hence,
  \[\theta_{d_k,m}(x_k,t_k)=\check w_k(z_k^*,t_k)\to \check w_\infty((0',\chi),t_*)=\Theta_m(0,t_*)\;\;\text{as}\;\;k\to\infty.\]
From the convergence $\Theta_m(x_k,t_k)\to \Theta_m(0,t_*)$, it follows that  $\Theta_m(x_k,t_k)-\theta_{d_k,m}(x_k,t_k)\to 0$, which contradicts \eqref{2.8}.

{\it Case }(ii): The sequence $\big\{y_k^{n}/\sqrt{d_k} \big\}^{\infty}_{k=1}$ is unbounded. Passing to a subsequence if necessary, we may assume that $y_k^{n}/\sqrt{d_k}\to \infty$ as $k\to\infty$. In this regime, we define
 \[\check w_k(z,t)=\tilde w_k\big(y_k+\sqrt{d_k}z,t\big),\;\;\text{with}\;\;z\in B_{\rho/\sqrt{d_k}}\cap \big\{(z',z^n)\!:z^n>-y^n_k/\sqrt{d_k}\big\}=:D_k,\]
and set
 \begin{align*}
   \mathcal L_k\psi&=\sum_{i,j=1}^n\check a^k_{ij}(z)\frac{\partial^2\psi}{\partial z^i\partial z^j}+\sqrt{d_k}\sum_{i=1}^n\check b^k_i(z)\frac{\partial\psi}{\partial z^i},\\[.1mm]
\check a^k_{ij}(z)&=a_{ij}(y_k+\sqrt{d_k}z), \;\;\check b^k_i(z)=b_i(y_k+\sqrt{d_k}z)\\[.1mm]
\check m^k(z,t)&=m(H(y_k+\sqrt{d_k}z),t),\\[.1mm]
S_k&=B_{\rho/\sqrt{d_k}}\cap \bigl\{(z',z^n): z^n=-y_k^n/\sqrt{d_k}\bigr\}.
  \end{align*}
A direct calculation shows that $\check w_k$ satisfies
 \[\begin{cases}
\displaystyle \partial_t\check w_k=\mathcal L_k\check w_k+\check w_k(\check m^k(z,t)-\check w_k), & z\in D_k,\; t\in(0,T],\\
 \displaystyle\sum_{j=1}^n\check a^k_{nj}(z)\frac{\partial\check w_k}{\partial z^j}=0, & z\in S_k,\; t\in(0,T],\\
\check w_k(z,0)=\check w_k(z,T), & z\in\overline{D_k}.
\end{cases}\]

Note that $y_k^n/\sqrt{d_k}\to\infty$. For any fixed $\chi>0$, there exists $k_\chi$ such that $B_\chi\subset D_k$ for all $k\geq k_\chi$. Consequently, the domains $D_k$ exhaust the entire space $\mathbb R^n$ as $k\to\infty$. By the same compactness argument employed in Case (i)--utilizing the local $L^p$ and Schauder estimates and the compactness argument--we may extract a further subsequence such that $\check w_k$ converges to a function $\check w_{\infty}\in C^{2,1}(\mathbb R^n\times[0,T])$ on any compact subset of $\mathbb R^n\times[0,T]$, and $\check w_{\infty}$ solves the whole-space problem \eqref{2.16}. As in Case (i), we can prove $\check w_\infty(z,t)\equiv \Theta_m(0,t)$ in $\mathbb R^n\times[0,T]$ and $\lim_{k\to\infty}|\Theta_m(x_k,t_k)-\theta_{d_k,m}(x_k,t_k)|=0$, and thereby contradicting \eqref{2.8}. Thus, \eqref{2.7} holds.

Combining \eqref{2.5} and \eqref{2.7}, we conclude $\lim_{d\to0}\|\theta_{d,m}-\Theta_m\|_{C(\overline Q_T)}=0$.
\end{proof}

By use of Lemma \ref{lm2.5} and the upper and lower solution method, we can prove the following lemma. Since its proof is analogous to that of \cite[Theorem 3.4]{FL17}, we omit the details here.

\begin{lemma}\label{lm2.6} Let $f, f_{(d_1,d_2)}\in C_T^{\alpha,\alpha/2}(\overline Q_T)$. If $\lim_{(d_1,d_2)\to(0,0)}\|f_{(d_1,d_2)}-f\|_{C(\overline Q_T)}=0$, then
  \[\lim_{(d_1,d_2)\to(0,0)}\|\Theta_{d_1,f_{(d_1,d_2)}}
-\Theta_{f}\|_{C(\overline Q_T)}=0.\zzz\]
\end{lemma}

Next, we establish estimates for $\theta_{d,m}$ and positive solutions to \eqref{1.5}. For brevity, we set $\|\cdot\|_p=\|\cdot\|_{L^p(Q_T)}$ for $1\le p\le\infty$. 
 
\begin{lemma}\label{lm2.7} Let $m\in C_T^{\alpha,\alpha/2}(\overline Q_T)$ with $m>0$ on $\overline Q_T$. Then the following assertions hold.
\begin{enumerate}
\item[\rm(1)] The estimates below hold {\rm(\!\!\cite[Lemma 4.4]{BHN23})}:
  \[  \|\theta_{d,m}\|_\infty<\|m\|_\infty,\;\; \|\partial_t\theta_{d,m}\|_\infty\leq C\|m\|_\infty^2,\;\; \min_{\overline Q_T}\theta_{d,m}>\min_{\overline Q_T}m.\zzz\]
\item[\rm(2)] If $(U,V)$ is a positive solution of \eqref{1.5}, then
 \begin{align}
\|U\|_\infty &<\|\theta_{d_1,m_1}\|_\infty<\|m_1\|_\infty, \;\; \|V\|_\infty<\|\theta_{d_2,m_2}\|_\infty<\|m_2\|_\infty,\label{2.17}\\
\|U_t\|_{L^2(Q_T)} &\leq C\|m_1\|_{\infty}\bigl(\|m_1\|_{\infty}+\|m_2\|_{\infty}\bigr), \label{2.18}\\
\|V_t\|_{L^2(Q_T)} &\leq C\|m_2\|_{\infty}\bigl(\|m_1\|_{\infty}+\|m_2\|_{\infty}\bigr).\label{2.19}
 \end{align}
\end{enumerate}
Here, the positive constant $C$ depends only on $T$ and $\Omega$.\www
\end{lemma}

\begin{proof}\; We only prove part (2). Let $(U,V)$ be a positive solution to \eqref{1.5}. Estimates \eqref{2.17} follow from the maximum principle for time‑periodic parabolic equations (cf. \cite[Theorem 7.1]{W21}). To derive \eqref{2.18}, multiply the equation for $U$ by $U_t$ and integrate over $Q_T$. This yields
   \begin{align*}
\|U_t\|_{L^2(Q_T)}^2
&=d_1\int_{Q_T}U_t\Delta U+\int_{Q_T}U_tU(m_1-U-V)\\
&\leq -\frac{d_1}{2}\int_{Q_T}\bigl(|\nabla U|^2\bigr)_t-\frac{1}{3}\int_{Q_T}\bigl(U^3\bigr)_t
+\bigl(\|m_1\|_{\infty}+\|m_2\|_{\infty}\bigr)\|U\|_2\|U_t\|_2\\[1mm]
&=\bigl(\|m_1\|_{\infty}+\|m_2\|_{\infty}\bigr)\|U\|_2\|U_t\|_2\\[1mm]
&\leq C\|m_1\|_{\infty}\bigl(\|m_1\|_{\infty}+\|m_2\|_{\infty}\bigr)\|U_t\|_2.
   \end{align*}
This implies \eqref{2.18}. The proof of \eqref{2.19} is similar to the above argument.
\zzz\end{proof}

We now characterize the asymptotic behavior of $\theta_{d,m}$ as $d\to 0$. Suppose that
 \begin{equation}\label{2.27a}
m\in C_T^{2,1}(\overline\Omega\times\mathbb R), \quad m>0 \text{ on } \overline\Omega\times\mathbb R, \quad \text{and}\;\;\, \partial_\nu m=0 \text{ on } \partial\Omega\times\mathbb R.
 \end{equation}
For each $x\in\overline\Omega$, periodic problem \eqref{1.6} has a unique positive solution $\Phi_{m}\in C^{2,1}(\overline Q_T)$, and 
 \[\int_0^T(m-2\Phi_m)=-\int_0^T\Phi_m<0.\] 
Thus, for each $x\in\overline\Omega$, the periodic ODE problem
   \bee\begin{cases}\label{2.20}
\Psi_t=\Psi(m-2\Phi_m)+\Delta \Phi_m,\;\; t\in[0,T],\\
  \Psi(x,0)=\Psi(x,T) \end{cases}\eee
has a unique solution by the variation-of-constants formula. The proposition below is a slight modification of \cite[Lemma 4.6]{BHN23}; its proof is omitted. We note that the assumptions $\nabla m\not\equiv 0$ and $m_t\not\equiv 0$ on $\overline Q_T$ imposed in \cite[Lemma 4.6]{BHN23} are not used in its proof and may hence be omitted.\vspace{-1mm}

\begin{proposition}\label{lm2.8}
Suppose that the function $m$ satisfies condition \eqref{2.27a}. Let $\Phi_m$ denote the unique positive solution to \eqref{1.6}, and let $\Psi_m$ denote the unique solution to \eqref{2.20}. Then the following convergence results hold:
  \begin{align*}
\lim_{d\to 0}\frac{1}{d}\big\| \theta_{d,m}-\Phi_m-d\Psi_m \big\|_{C^{0,1}(\overline Q_T)}=0,\;\;\
\lim_{d\to 0}\big\| \Delta\theta_{d,m}-\Delta \Phi_m \big\|_{C(\overline Q_T)}=0,
\end{align*}
and, for every $q>1$,
  \[\lim_{d\to 0}\big\| \theta_{d,m}(\cdot,t)-\Phi_m(\cdot,t) \big\|_{W^{2,q}(\Omega)}=0
\quad \text{uniformly for } t\in[0,T].\]
In particular, $\lim_{d\to 0}\|\theta_{d,m}-\Phi_m\|_{C([0,T],C^1(\overline\Omega))}=0$.
\vspace{-2mm}\end{proposition}

To conclude this section, we give some results on the asymptotic profile of $\theta_{d,m}$ as $d\to\infty$. Let $\Gamma_m$ be the unique solution to \eqref{1.10} and $l_m(t)$ be the unique positive solution to \eqref{1.11}. Then $\widehat l_m=\widehat{\ol m}>0$, and the periodic ODE problem 
  \bee\begin{cases}\label{2.21}
h_t+ l_mh=\dd\frac{1}{|\Omega|}\int_{\Omega}|\nabla\Gamma_m|^2,\;\;0<t\le T,\\[2mm]
h(0)=h(T)
  \end{cases}\eee\zzz
has unique solution $h_m$ by the variation-of-constants formula. For each $t\in[0,T]$, in view of the Fredholm alternative theorem (cf. \cite[Theorem 6.2.4]{E10}), the elliptic problem 
  \bee\label{2.22}\begin{cases}
-\Delta\Lambda=(\Gamma_m+h_m)(m-\ol m)-\partial_t\Gamma_m- l_m\Gamma_m-\dd\frac{1}{|\Omega|}
\int_{\Omega}|\nabla\Gamma_m|^2, & x\in\Omega,\\[.5mm]
\partial_{\nu}\Lambda=0, & x\in\partial\Omega,\\[.2mm]
\ol\Lambda(t)=0
\end{cases}
 \eee
has a unique solution $\Lambda_m(x,t)$, and $\Lambda_m(x,t)$ is $T$-periodic in time $t$. Let $k_m$ denote the unique solution to
  \bee\label{2.23}
  \begin{cases}
k_t+ l_m k=\dd\frac{1}{|\Omega|}\int_{\Omega}\Lambda_m(m-2l_m)
-\frac{l_m}{|\Omega|}\int_{\Omega}\Gamma_m^2-l_mh_m^2,\;\;0<t\le T,\\[2mm]
k(0)=k(T).
 \end{cases} \eee

Based on standard regularity theory for parabolic equations, we have the following lemma; the detailed proof is omitted for brevity.

\begin{lemma}\label{lm2.9} Assume $m\in C_T^{0,1}(\overline Q_T)$. Then
  \[h_m\in C^1([0,T]), \;\;\;\Gamma_m\in C^1([0,T], W_q^2(\Omega)),\;\;\forall\; q>n,\]
and 
 \[\|h_m\|_{C^1([0,T])}\le C(m),\quad\|\Gamma_m\|_{C^1([0,T],W_q^2(\Omega))}\leq C_q(m).\]
Moreover, if $m\in H^2((0,T),L^p(\Omega))$ for some $p>n$, then 
  \[\Lambda_m\in H^1((0,T),W_p^2(\Omega)),\;\;\;k_m\in H^1((0,T)),\] 
and
  \[\|\Lambda_m\|_{H^1((0,T),W_p^2(\Omega))}\leq C_p(m),\quad \|k_m\|_{H^1((0,T))}\leq C_p(m).\]
Here, $C_k(m)$ depends only on $m$ and $k$ for $k=p,q$, and $C(m)$ depends only on $m$.
\end{lemma}\zzz

We now show the asymptotic profiles of $\theta_{d,m}$ as $d\to \yy$.

\begin{lemma}\label{lm2.10} Assume that $m\in C_T^{\alpha,1}(\overline Q_T)$ and $m>0$ on $\overline Q_T$. Then there exists a constant $D_m>0$ depending only on $m$ such that
   \bee\label{2.24}
\theta_{d,m}= l_m+l_m(\Gamma_m+h_m)d^{-1} +o(d^{-1}),\;\;\forall\, d>D_m.
 \eee
Moreover, if $m\in H^2((0,T),L^p(\Omega))$ for some $p>n$, then
   \bee\label{2.25}
\theta_{d,m}= l_m+l_m(\Gamma_m+h_m)d^{-1}+l_m(\Lambda_m+k_m)d^{-2}+o(d^{-2}),\;\;\forall\, d>D_m.
 \eee
Here, the higher-order term $o(\cdot)$ is understood in the sense of the $C(\ol Q_T)$-norm.
\end{lemma}\zzz

\begin{proof}\; When $m=m(x)$ is independent of $t$, He and Ni \cite[Proposition 3.1]{HN16b} derived asymptotic expansions \qq{2.24} and \qq{2.25} via the upper and lower solutions method. If $m$ depends on $t$, the same approach requires the extra regularity  $m_{tt}\in C(\ol Q_T)$  for \qq{2.24} and $m_{ttt}\in C(\ol Q_T)$  for \qq{2.25}. By contrast, we establish these expansions under the weaker assumption $m\in H^2((0,T),L^p(\Omega))$ by means of the implicit function theorem.

{\it Step 1}. We first prove \eqref{2.24}. Decompose
\[  \theta_{d,m}= l_m+{\theta_1}d^{-1}.\]
Then $\theta_1\in C_T^{2,1}(\overline Q_T)$ satisfies
 \bee\label{2.26} \begin{cases}
\theta_t-d\Delta \theta=d l_m(m-\ol m)+(m-2 l_m)\theta-\theta^2d^{-1} &\text{in}\;\;Q_T,\\
\partial_{\nu}\theta=0 &\text{on}\;\;S_T,\\
\theta(x,0)=\theta(x,T) &\text{on}\;\;\boo.
\end{cases}
 \eee
We will use the implicit function theorem to show that $\theta_1\approx l_m(\Gamma_m+h_m)$ when $d\gg1$. 

Let $q>n$, and define Banach spaces
  \bee\label{2.27}\begin{cases}
 \mathbb X_q:=\left\{\phi\in H^1((0,T),W_q^2(\Omega)):\partial_{\nu}\phi|_{\partial\Omega}=0,\,
 \phi(x,0)=\phi(x,T),\,\int_{\Omega}\psi=0\right\},\\[1mm]
  \mathbb Y_q:=\left\{\phi\in L^2((0,T),L^q(\Omega)):\phi(x,0)=\phi(x,T),
  \,\int_{\Omega}\phi=0\right\},\\[1mm]
E:=\{b\in H^1((0,T)):b(0)=b(T)\},\\[1mm]
\widetilde E:=\{b\in L^2((0,T)):b(0)=b(T)\}.
  \end{cases} \eee
In the meantime, we introduce an operator $\mathcal F:\,\mathbb X_q\times E\times[0,\infty)
\to\mathbb Y_q\times\widetilde E$ by
 \[\mathcal F(\sigma,r,s)=\big(f_1(\sigma,r,s),\;
    f_2(\sigma,r,s)\big)^{\mathrm{T}},\]
with
 \begin{align*}
 f_1(\sigma,r,s)&=\Delta \sigma+ l_m(m-\ol m)-s\kk[(\sigma+r)_t-(m-2 l_m)(\sigma+r)
+s(\sigma+r)^2\rr],\\
 f_2(\sigma,r,s)&=\int_{\Omega}(\sigma+r)_t-\int_{\Omega}\kk[(m-2 l_m)(\sigma+r)
 -s(\sigma+r)^2\rr].
  \end{align*}
Clearly, for $s=1/d$ and $(\sigma,r)\in \mathbb X_q\times E$, $\theta=\sigma+r$ solves \eqref{2.26} if and only if $\mathcal F(\sigma,r,s)=0$.

By Lemma \ref{lm2.9}, $\Gamma_m\in\mathbb X_q$ and $h_m\in E$. Thus  $\sigma_0=l_m\Gamma_m\in \mathbb X_q$ and $r_0=l_mh_m\in E$. The fact $\ol\Gamma_m(t)=0$ implies $\ol{(\Gamma_m)_t}(t)=0$ for all $0\le t\le T$. Consequently, 
  \begin{align*}
  f_1(\sigma_0, r_0,0)&=\Delta( l_m\Gamma_m)+ l_m(m-\ol m)=0,\\
f_2(\sigma_0, r_0,0)&=\int_{\Omega}(l_m\Gamma_m+l_mh_m)_t
-\int_{\Omega}(m-2 l_m)(l_m\Gamma_m+ l_mh_m)\\
&=|\oo|\kk[(l_m)_th_m+l_m(h_m)_t-(\ol m-2l_m)l_mh_m\rr]-\int_{\Omega}(m-\ol m)\Gamma_m\\
&=|\oo|\kk(l_m(h_m)_t+ l_m^2h_m\rr)-\int_{\Omega}|\nabla\Gamma_m|^2=0,
  \end{align*}
i.e., $\mathcal F(\sigma_0, r_0,0)=0$. The Fr\'echet derivative of $\mathcal F$ with respect to $(\sigma,r)$ at $(\sigma_0, r_0,0)$ is given by
  \[D_{\sigma r}\mathcal F(\sigma_0, r_0,0)\begin{pmatrix}
    \phi\\ b
\end{pmatrix}
  =\begin{pmatrix}
  \Delta \phi\\[1mm]
\int_{\Omega}(\phi+b)_t-\int_{\Omega}(m-2 l_m)(\phi+b)
  \end{pmatrix},\;\;(\phi,b)\in\mathbb X_q\times E.\]
  
We show that $D_{\sigma r}\mathcal F(\sigma_0, r_0,0)$ is non-degenerate. Let $(\phi,b)\in\mathbb X_q\times E$ satisfy 
 \[D_{\sigma r}\mathcal F(\sigma_0, r_0,0)(\phi,\,b)^{\mathrm{T}}=0.\] 
Then $\Delta \phi=0$ and $\int_{\Omega}\phi=0$, so $\phi=0$. Note that $b\in E$. It follows that $b$ satisfies
 \[b_t=(\ol m-2 l_m)b.\]
Thus, $b=0$ as $ l_m$ is the unique positive solution to \qq{1.11}. This implies that the linearized operator $D_{\sigma r}\mathcal F(\sigma_0, r_0,0)$ is non-degenerate.
Applying the implicit function theorem to the equation $\mathcal F(\sigma,r,s)=0$, we obtain a sufficiently small constant $s_0>0$ and a mapping
  \[s\mapsto (\sigma(s), r(s))\in C\big([0,s_0),\mathbb X_q\times E\big)\]
such that $(\sigma(0), r(0))=(\sigma_0, r_0)$ and $(\sigma(s), r(s))$ uniquely solves $F(\sigma,r,s)=0$ for all $s\in[0,s_0)$.
Therefore, $\sigma(s)+r(s)$ is the unique solution to \eqref{2.26} with  $s=d^{-1}\in(0,s_0)$, which together with the fact that $\theta_1$ is a solution to \eqref{2.26} yields $\theta_1=\sigma(s)+r(s)$. 

In view of $q>n$, we obtain the continuous embedding
  \[H^1((0,T),W_q^2(\Omega))\hookrightarrow C(\ol Q_T).\]
Note that $\lim_{s\to 0}\sigma(s)=\sigma_0=l_m\Gamma_m$ in $\mathbb X_q$ and $\lim_{s\to 0}r(s)=r_0=l_mh_m$ in $E$. Finally, we obtain
  \[\lim_{d\to\infty}\theta_1(x,t)=\lim_{s\to 0}\bigl(\sigma(s)+ r(s)\bigr)= l_m(h_m+\Gamma_m)\quad \text{in } C(\ol Q_T),\]
and \eqref{2.24} holds.

{\it Step 2}. We now prove \eqref{2.25}. From Lemma \ref{lm2.9}, one has $\Lambda_m\in \mathbb X_p$ and $k_m\in E$, and thereby $\sigma_0=l_m\Lambda_m\in\mathbb X_q, r_0=l_m k_m\in E$. Let
  \begin{align*}\theta_{d,m}=l_m+{l_m(\Gamma_m+h_m)}d^{-1}+\theta_2d^{-2}.\end{align*}
Then $\theta_2\in C^1_T([0,T],W_p^2(\Omega))$ satisfies, with $s=d^{-1}$,
 \bee\label{2.28}\begin{cases}
s\theta_t-\Delta\theta=Q(\theta,s) & \text{in }\; Q_T,\\[.1mm]
\partial_{\nu}\theta=0 & \text{on }\;S_T,\\[.1mm]
\theta(x,0)=\theta(x,T) & \text{on }\;\overline\Omega,
  \end{cases} \eee
where
 \begin{align*}
 Q(\theta,s)=\;&s\kk[(m-2 l_m)\theta- l_m^2(\Gamma_m+h_m)^2-2s l_m(\Gamma_m+h_m)\theta-s^2\theta^2\rr]\\
&+l_m\kk((\Gamma_m+h_m)(m-\ol m)-\dd\partial_t\Gamma_m- l_m\Gamma_m-\frac{1}{|\Omega|}\int_{\Omega}|\nabla\Gamma_m|^2 \rr).
 \end{align*}
We will use the implicit function theorem to show that $\theta_2\approx l_m(\Lambda_m+k_m)$ when $d\gg1$. Define $\mathbb X_p, \mathbb Y_p,E, \widetilde E$ as in \qq{2.27}, and
   \begin{align*}g_1(\sigma,r,s)=\;&\Delta \sigma-s(\sigma+r)_t+Q(\sigma+r,s),\\[.1mm]
g_2(\sigma,r,s)=&\int_{\Omega}(\sigma+r)_t
-\int_{\Omega}Q(\sigma+r,s), \end{align*}
and define an operator $\mathcal G:\mathbb X_p\times E \times[0,\infty)\to\mathbb Y_p\times\widetilde E$ by
   \[\mathcal G(\sigma,r,s)=\big(g_1(\sigma,r,s),\,
    g_2(\sigma,r,s)\big)^{\mathrm{T}}.\]
Clearly, $\mathcal G(\sigma_0, r_0,0)=0$, and for  $s>0$ and $(\sigma,r)\in E\times\mathbb X_p$, $\theta_2=\sigma+r$ solves \eqref{2.28} if and only if $\mathcal G(\sigma,r,s)=0$. The Fr\'echet derivative of $\mathcal G$ with respect to $(\sigma,r)$ at $(\sigma_0, r_0,0)$ is given by
  \[D_{\sigma r}\mathcal G(\sigma_0, r_0,0)\begin{pmatrix}
    \phi\\ b
\end{pmatrix}=\begin{pmatrix}
 \Delta\phi\\[.5mm]
 \int_{\Omega}(\phi+b)_t-\int_{\Omega}(m-2 l_m)(\phi+b)
  \end{pmatrix}.\]
Similar to Step 1, we can verify that $D_{\sigma r}\mathcal G(\sigma_0, r_0,0)$ is non-degenerate, and use the implicit function theorem to deduce \eqref{2.25}.
\end{proof}\www

\section{Small diffusion rates: proofs of Theorems \ref{th1.1} and \ref{th1.2}}\label{sec3}\setcounter{equation}{0}

We first state a known result. Let $r_1,r_2\in C_T([0,T])$ be positive functions, and consider the following Lotka-Volterra ODE system with time-periodic coefficients:
  \bee\label{3.1}\begin{cases}
u_t=u(r_1(t)-u-v), \;\;\; t>0,\\
v_t=v(r_2(t)-u-v), \;\;\; t>0,\\
u(0)=u_0>0,\;\,v(0)=v_0>0.
  \end{cases} \eee
Then $(\Phi_{r_1}(t),0)$ and $(0,\Phi_{r_2}(t))$ are two semi-trivial $T$-periodic solutions to \eqref{3.1}.

\begin{proposition}\label{p3.1}{\rm(\!\cite[Proposition 36.2]{H91})}
If $\widehat r_1>\widehat r_2$, then \eqref{3.1} has no positive $T$-periodic solution, and $(\Phi_{r_1}(t),0)$ is globally asymptotically stable. If $\widehat r_1<\widehat r_2$, then \eqref{3.1} has no positive $T$-periodic solution, and $(0,\Phi_{r_2}(t))$ is globally asymptotically stable.
\end{proposition} 

Henceforth, to save space and keep the page clean, we abbreviate for $i=1,2$,
 \[ \mathcal{E}(m_i)=\mathcal{E}_i,\;\Phi_{m_i}=\Phi_i, \;\Gamma_{m_i}=\Gamma_i,\;\Psi_{m_i}=\Psi_i,
 \;\Lambda_{m_i}=\Lambda_i,\; l_{m_i}= l_i,\;h_{m_i}=h_i,
 \;k_{m_i}=k_i,\]
where $\mathcal{E}(m_i)$ is defined by \qq{1.12}, and $\Phi_{m_i}, \Gamma_{m_i},\Psi_{m_i}, \Lambda_{m_i}, l_{m_i}, h_{m_i}, k_{m_i}$ are the unique solutions to \qq{1.6}, \qq{1.10}, \qq{2.20}, \qq{2.22}, \qq{1.11}, \qq{2.21}  and \qq{2.23} with $m=m_i$, respectively, for $i=1,2$.\vskip 4pt

\begin{proof}[Proof of Theorem $\ref{th1.1}$] By Proposition \ref{lm2.1}(3), Lemma \ref{lm2.5} and the fact \qq{1.7} we have
 \begin{align*}
  &\lim_{d_2\to 0}\lim_{d_1\to 0}\mu(d_2,m_2-\theta_{d_1,m_1})
  =-\max_{\overline\Omega}[\widehat{m}_2(x)-\widehat \Phi_1(x)]
=-\max_{\overline\Omega}[\widehat{m}_2(x)-\widehat{m}_1(x)]<0,\\
&\lim_{d_1\to 0}\lim_{d_2\to 0}\mu(d_1,m_1-\theta_{d_2,m_2})=-\max_{\overline\Omega}
[\widehat{m}_1(x)-\widehat \Phi_2(x)]=-\max_{\overline\Omega}[\widehat{m}_1(x)-\widehat{m}_2(x)]<0.
  \end{align*}
Together with Proposition \ref{p2.4},  this forces both $(\theta_{d_1,m_1},0)$ and $(0,\theta_{d_2,m_2})$ to be linearly unstable; hence system \eqref{1.4} is uniformly persistent and \qq{1.5} admits a linearly stable positive solution when both $d_1$ and $d_2$ are small.

Let $(U,V)$ denote a positive solution of \eqref{1.5}. In the following, we show the asymptotic profile of $(U,V)$ as $d_1,d_2\to 0$. Recall that $\Theta_{d,m}$ denotes the maximal nonnegative solution to \eqref{1.2}, which belongs to $C_T^{2+\alpha,1+\alpha/2}(\overline Q_T)$. We define the sequences consisting of $C_T^{2+\alpha,1+\alpha/2}(\overline Q_T)$-functions by
  \bee\label{3.2} \begin{cases}
  \tilde u_1=\theta_{d_1,m_1},\;\;\undt v_k=\Theta_{d_2,m_2-\tilde u_k},\;\;
  \tilde u_{k+1}=\Theta_{d_1,m_1-\undt v_k},\;\;k\geq 1,\\[.1mm]
\tilde v_1=\theta_{d_2,m_2},\;\; \undt u_k=\Theta_{d_1,m_1-\tilde v_k},\;\; \tilde v_{k+1}=\Theta_{d_2,m_2-\undt u_k},\;\;k\geq 1.
  \end{cases} \eee
By the maximum principle for time-periodic parabolic equations (cf. \cite[Theorem 7.1]{W21},
  \bee\label{3.3}
 \undt u_k\leq \undt u_{k+1}\leq U\leq \tilde u_{k+1}\leq \tilde u_k,\;\;\;
\undt v_k\leq \undt v_{k+1}\leq V\leq \tilde v_{k+1}\leq \tilde v_k,\;\;\forall\, k\geq 1.
 \eee

{\it Step 1}. We show that there exist $C_T^{\alpha,1}(\overline Q_T)$-functions:  $\widetilde U_k$, $\undl U_k$, $\widetilde V_k$ and $\undl V_k$, such that
  \bee\label{3.4}
 \tilde u_k\to\widetilde U_k,\;\; \undt u_k\to\undl U_k,\;\;\;
\tilde v_k\to\widetilde V_k,\;\; \undt v_k\to\undl V_k\;\;\;\text{in}\;\;C(\tilde Q_T)
   \eee
as $(d_1,d_2)\to(0,0)$ for all $k$. In addition, these functions are recursively given by
  \bee\label{3.5}\begin{cases}
\widetilde U_1=\Phi_1,\;\;\undl V_1=\Theta_{m_2-\widetilde U_1},\;\;
\widetilde U_{k+1}=\Theta_{m_1-\undl V_k}, \;\;
 \undl V_{k+1}=\Theta_{m_2-\widetilde U_{k+1}},\\[1mm]
\widetilde V_1=\Phi_2, \;\;\undl U_1=\Theta_{m_1-\widetilde V_1},\;\;
\widetilde V_{k+1}=\Theta_{m_2-\undl U_k},\;\;
 \undl U_{k+1}=\Theta_{m_1-\widetilde V_{k+1}},
 \end{cases}
  \eee
where $\Phi_i$ ($i=1,2$) is the unique positive $T$-periodic solution of \eqref{1.6} with $m=m_i$ for each $x\in\overline\Omega$.

Firstly, it follows from Lemma \ref{lm2.6} that, in $C(\ol Q_T)$,
  \begin{align*}\lim_{(d_1,d_2)\to(0,0)}\tilde u_1=\Phi_1=:\widetilde U_1,\;\;\;
  \lim_{(d_1,d_2)\to(0,0)}\tilde v_1=\Phi_2=:\widetilde V_1.
  \end{align*}
Thus, by use of Lemma \ref{lm2.6} again, in $C(\ol Q_T)$,
 \begin{align*}
  \lim_{(d_1,d_2)\to(0,0)}\undt v_1=\lim_{(d_1,d_2)\to(0,0)}\Theta_{d_2,m_2-\tilde u_1}=\Theta_{m_2-\widetilde U_1}=:\undl V_1,\\
 \lim_{(d_1,d_2)\to(0,0)}\undt u_1=\lim_{(d_1,d_2)\to(0,0)}\Theta_{d_1,m_1-\tilde v_1}=\Theta_{m_1-\widetilde V_1}=:\undl U_1.
 \end{align*}
By applying Lemma \ref{lm2.6} repeatedly, we can thereby obtain, in $C(\ol Q_T)$,
  \begin{align*}
  \lim_{(d_1,d_2)\to(0,0)}\tilde u_2=\lim_{(d_1,d_2)\to(0,0)}\Theta_{d_1,m_1-\undt v_1}=\Theta_{m_1-\undl V_1}=:\widetilde U_2,\\
  \lim_{(d_1,d_2)\to(0,0)}\undt v_2=\lim_{(d_1,d_2)\to(0,0)}\Theta_{d_2,m_2-\tilde u_2}=\Theta_{m_2-\widetilde U_2}=:\undl V_2,\\
  \lim_{(d_1,d_2)\to(0,0)}\tilde v_2=\lim_{(d_1,d_2)\to(0,0)}\Theta_{d_2,m_2-\undt u_1}=\Theta_{m_2-\undl U_1}=:\widetilde V_2,\\
  \lim_{(d_1,d_2)\to(0,0)}\undt u_2=\lim_{(d_1,d_2)\to(0,0)}\Theta_{d_1,m_1-\tilde v_2}=\Theta_{m_1-\widetilde V_2}=:\undl U_2.
  \end{align*}
By induction, \qq{3.4} and \qq{3.5} follow immediately.

{\it Step 2}. From \eqref{3.3}, \eqref{3.4} and Step 1 we obtain that
 \bee\label{3.6}
 \begin{cases}
\Theta_{m_1-\Phi_2}\leq\dd\undl U_k\leq \undl U_{k+1}\leq \liminf_{(d_1,d_2)\to(0,0)}U \leq \limsup_{(d_1,d_2)\to(0,0)}U\leq \widetilde U_{k+1}\leq \widetilde U_k\leq \Phi_1,\\
\Theta_{m_2-\Phi_1}\leq\dd\undl V_k\leq \undl V_{k+1}\leq \liminf_{(d_1,d_2)\to(0,0)}V \leq \limsup_{(d_1,d_2)\to(0,0)}V\leq \widetilde V_{k+1}\leq \widetilde V_k\leq \Phi_2,
  \end{cases} \eee
uniformly in $\ol Q_T$, which yields four bounded and time $T$-periodic functions $U^*,U_*,V^*$ and $V_*$ such that for each $(x,t)\in\ol Q_T$,
   \[ \lim_{k\to\infty}\widetilde U_k=U^*,\ \lim_{k\to\infty}\undl U_k=U_*,\ \lim_{k\to\infty}\widetilde V_k=V^*\ \mbox{and}\ \lim_{k\to\infty}\undl V_k=V_*.\]
Moreover, it is inferred from \eqref{3.5} that the functions $\undl V_k,\widetilde U_{k+1}, \undl U_k,\widetilde V_{k+1}$ satisfy
  \bee \begin{cases}
  \partial_t\undl V_k=\undl V_k(m_2(x,t)-\widetilde U_k-\undl V_k) & \text{on} \;\;\ol Q_T,\\[1mm]
 \partial_t \widetilde U_{k+1}=\widetilde U_{k+1}(m_1(x,t)-\undl V_k-\widetilde U_{k+1})& \text{on} \;\;\ol Q_T,\\[1mm]
 \partial_t\undl U_k=\undl U_k(m_1(x,t)-\widetilde V_k-\undl U_k)& \text{on} \;\;\ol Q_T,\\[1mm]
 \partial_t\widetilde V_{k+1}=\widetilde V_{k+1}(m_2(x,t)-\undl U_k-\widetilde V_{k+1})& \text{on} \;\;\ol Q_T.
 \end{cases}\label{3.7}\eee
For each $x\in\boo$, since the functions on the right-hand side of \qq{3.7} are uniformly bounded on $[0,T]$, we first deduce the equicontinuity of $\big\{(\widetilde U_k, \undl V_k, \undl U_k, \widetilde V_k)(x,\cdot)\big\}$ by the mean value theorem. Then, by the Arzel\`{a}-Ascoli theorem, $\big\{(\widetilde U_k, \undl V_k, \undl U_k, \widetilde V_k)(x,\cdot)\big\}$ converges uniformly to $\big\{(U^*, V_*, U_*, V^*)(x,\cdot)\big\}$ on $[0,T]$. Finally, applying the theorem on the interchange of limit and differentiation, we conclude that $U^*(x,\cdot),V_*(x,\cdot),U_*(x,\cdot),V^*(x,\cdot)\in C^1([0,T])$, and $(U^*,V_*)$ and $(U_*,V^*)$ satisfy
 \bee\label{3.8}\begin{cases}
  U_t=U(m_1(x,t)-U-V)& \text{on} \;\;\ol Q_T,\\[.1mm]
  V_t=V(m_2(x,t)-U-V)& \text{on} \;\;\ol Q_T,\\[.1mm]
U(x,0)=U(x,T),\ V(x,0)=V(x,T)& \text{on} \;\;\boo.
\end{cases}
 \eee
Since $m_i\in C^{2,1}(\overline Q_T)$ (cf. the condition {\bf(M)}), the dependence of solutions to ODEs on the parameters yields $U^*,V_*,U_*,V^*\in C_T^{2,1}(\overline Q_T)$.

{\it Step 3}. By Dini's Theorem, there holds that
\[ \lim_{k\to\infty}\widetilde U_k=U^*,\ \lim_{k\to\infty}\undl U_k=U_*,\ \lim_{k\to\infty}\widetilde V_k=V^*\ \mbox{and}\ \lim_{k\to\infty}\undl V_k=V_*\]
in $C(\ol Q_T)$. Moreover,
  \[  \Theta_{m_1-\Phi_2}\leq U_* \leq U^* \leq \Phi_1\quad\mbox{and} \quad \Theta_{m_2-\Phi_1}\leq V_* \leq V^* \leq \Phi_2.\]
By Proposition \ref{p3.1}, for each $x\in\Omega_1\cup\Omega_2$, system \eqref{3.8} has no positive $T$-periodic solution. Using
 \[\widehat{(m_1-\Phi_2)}(x)=\widehat{m}_1(x)-\widehat{m}_2(x)>0,\;\;\forall\, x\in\Omega_1,\]
one concludes that 
 \[\widehat U^*(x)\ge\widehat U_*(x)\ge \widehat\Theta_{m_1-\Phi_2}(x)>0,\;\;\forall\, x\in\Omega_1. \]
Thus, $V_*=V^*=0$, so that $(U^*,V_*)=(U_*,V^*)=(\Phi_1,0)$ for all $x\in\Omega_1$. Similarly, $(U^*,V_*)=(U_*,V^*)=(0,\Phi_2)$ for all $x\in\Omega_2$.
It then follows from \eqref{3.6} that the limit \eqref{a.9} holds. Theorem \ref{th1.1} is proved.
\end{proof}

\begin{proof}[Proof of Theorem $\ref{th1.2}$] By the proof of \cite[Theorem 2.1$'$]{BHN23}, there exists some $\ep_0>0$ such that the following facts hold:
\begin{enumerate}
\item[(a)] If $I(\Phi_i)>0$ and $\min_{x\in\overline\Omega}\int_0^T\frac{\Delta \Phi_i}{\Phi_i}(x,t)>0$ for $i=1,2$, then $\mu(d_2,m_1-\theta_{d_1,m_1})>0$ and $\mu(d_1,m_2-\theta_{d_2,m_2})<0$ for all $(d_1,d_2)\in(0,\ep_0)\times(0,d_1)$.
    \vskip 6pt
\item[(b)] If $I(\Phi_i)>0$ and $\min_{x\in\overline\Omega}\int_0^T\frac{\Delta \Phi_i}{\Phi_i}(x,t)<0$ for $i=1,2$, then $\mu(d_1,m_2-\theta_{d_2,m_2})<0$ and there exists a $C^1$ function $\ol{\mathsf{d}}_2:(0,\ep_0)\to(0,\ep_0)$ with $0<\ol{\mathsf{d}}_2(d_1)<d_1$ such that
    \begin{align*}
    \mu(d_2,m_1-\theta_{d_1,m_1})<\;&0,\;\;\forall\, (d_1,d_2)\in(0,\ep_0)\times(0,\ol{\mathsf{d}}_2(d_1)),\\
    \mu(d_2,m_1-\theta_{d_1,m_1})>\;&0,
    \;\;\forall\, (d_1,d_2)\in(0,\ep_0)\times(\ol{\mathsf{d}}_2(d_1),d_1). \end{align*}
Moreover,
   \[\lim_{d_1\to 0^+}\overline{\mathsf{d}}_2(d_1)=0,\quad\lim_{d_1\to 0^+}\overline{\mathsf{d}}_2'(d_1)=\overline{\mathsf{d}}_2'(0)\in(0,1).\]
\item[(c)] If $I(\Phi_i)<0$ for $i=1,2$, and the condition {\bf(M2$'$)} holds, then $\mu(d_2,m_1-\theta_{d_1,m_1})<0$ and there exists a $C^1$ function $\ud{\mathsf{d}}_2:(0,\ep_0)\to(0,\ep_0)$ with $0<\ud{\mathsf{d}}_2(d_1)<d_1$ such that
     \begin{align*}
    \mu(d_1,m_2-\theta_{d_2,m_2})<\;&0,\;\;\forall\, (d_1,d_2)\in(0,\ep_0)\times(0,\ud{\mathsf{d}}_2(d_1)),\\
    \mu(d_1,m_2-\theta_{d_2,m_2})>\;&0,
    \;\;\forall\, (d_1,d_2)\in(0,\ep_0)\times(\ud{\mathsf{d}}_2(d_1),d_1). \end{align*}
Moreover,
  \[ \lim_{d_1\to 0^+}\ud{\mathsf{d}}_2(d_1)=0,\quad\lim_{d_1\to 0^+}\ud{\mathsf{d}}_2'(d_1)=\ud{\mathsf{d}}_2'(0)\in(0,1).\]
\end{enumerate}

Note that $m_1-m_2$ is spatially homogeneous and $\int_{Q_T}(m_1-m_2)=0$, By Proposition \ref{lm2.1}(4), we obtain 
 \begin{align*} 
 \mu(d_2,m_2-\theta_{d_1,m_1})=\mu(d_2,m_1-\theta_{d_1,m_1}),\;\;\;
 \mu(d_1,m_1-\theta_{d_2,m_2})=\mu(d_1,m_2-\theta_{d_2,m_2}).
 \end{align*}
Conclusion (1) then follows immediately from Fact (a) and Proposition \ref{p2.3}.
Conclusions (2) and (3) are obtained from Facts (b) and (c), respectively, together with Propositions \ref{p2.3} and \ref{p2.4}(3).
\end{proof}

\section{Lemmas for the large-diffusion regime\label{sec4}}\setcounter{equation}{0}
{\setlength\arraycolsep{2pt}

In this section, we establish several lemmas concerning the large-diffusion regime, which are essential for proving Theorems \ref{th1.3}-\ref{th1.5}. The first lemma establishes that, as $d_2\to\infty$, every positive solution to system \eqref{1.5} converges to a nonnegative solution to its associated shadow system.

\begin{lemma}\label{lm4.1} Suppose $d_1^k\to d_1>0$, $d_2^k\to\yy$ as $k\to\infty$, and let $(U_k,V_k)$ be positive solution to \qq{1.5} with $(d_1,d_2)=(d_1^k, d_2^k)$.  Then $U_k\to U$ in $C^{1+\alpha,\frac{1+\alpha}2}(\ol Q_T)$ and $V_k\to r$ in $C(\ol Q_T)$, where $U\in C^{2+\alpha,1+\frac\alpha2}(\ol Q_T)$ and $r\in C^1([0,T])$ are nonnegative solution to the shadow system:
  \bee\begin{cases}
  U_t=d_1\Delta U+U(m_1(x,t)-U-r) &\text{in}\;\;Q_T,\\
   \partial_{\nu}U=0 &\text{on}\;\;S_T,\\
  r'(t)=r(\ol m_2(t)-\ol U(t)-r)&\text{in}\;\;[0,T],\\
  U(x,0)=U(x,T),\;  &\text{on}\;\;\boo,\\
  r(0)=r(T).
\end{cases}\label{4.1}\eee
 \end{lemma}

\begin{proof}\; {\it Step 1}. By Lemma \ref{lm2.7}, $(U_k, V_k)$ is uniformly bounded in $Q_T$. Let
 \[f_k=U_k(m_1-U_k-V_k),\;\;\;g_k=V_k(m_2-U_k-V_k).\]
Then $f_k$ and $g_k$ are uniformly bounded in $Q_T$. Using the interior estimate in $\oo\times[T/2,2T]$ and the periodicity of $U_k$ we obtain
 \begin{align*}
 \|U_k\|_{W^{2,1}_p(Q_T)}&\le C(p),\;\;\forall\, p>1,\\[1mm]
 \|U_k\|_{C^{1+\gamma,\frac{1+\gamma}2}(\ol Q_T)}&\le C(\gamma),\;\;\forall\,\gamma\in(0,1).
 \end{align*}
Thus, there exists $U\in C^{1+\alpha,\frac{1+\alpha}2}(\ol Q_T)$ such that
 \bee
 U_k\to U \;\;\;\text{in}\;\;C^{1+\alpha,\frac{1+\alpha}2}(\ol Q_T).
 \label{4.2}\eee

Set $W_k=V_k-\ol V_k(t)$. Then $W_k$ satisfies 
  \bee\begin{cases}
 \partial_t W_k=d_2^k\Delta W_k+g_k-\ol g_k(t)&\text{in}\;\; Q_T,\\
 \partial_\nu W_k=0&\text{on}\;\; S_T,\\
 W_k(x,0)=W_k(x,T)&\text{on}\;\; \boo,\\
 \ol W_k(t)=0&\text{on}\;\;[0,T],\end{cases}\label{4.3}\eee
and $\ol V_k$ satisfies
 \bee
 \ol V_k'(t)=\ol g_k(t)\;\;\;\text{in}\;\;[0,T];
 \;\;\;\ol V_k(0)=\ol V_k(T).\label{4.4}\eee

{\it Step 2}. Without loss of generality, we may assume that $d_2^k>1$ for all $k$.
Take $\alpha>1/2$ and $p>\max\{2,\frac{n}{2\alpha}\}$, and define
   \begin{align*}
  X_\perp=\;&\{w\in L^p(\oo): \bar w=0\},\;\;\;\text{with}\;\;\|\cdot\|_{X_\perp}=\|\cdot\|_{L^p(\oo)},\\[.5mm]
  Y_\perp=\;&\kk\{w\in W^2_p(\oo):\,\partial_\nu w|_{\partial\oo}=0,\;\ol w(t)=0\rr\}.
   \end{align*}
Then $X_\perp$ and $Y_\perp$ are Banach spaces. Let $0<\theta<\pi/2$ be fixed. By the Agmon-Douglis-Nirenberg theorem, there exists $C(\oo, p, \theta)>0$ such that for any $\lambda\in\mathbb{C}$ with $|\lambda|\ge 1$ and $|\arg(\lambda)|>\theta$, the following holds:
  \begin{align*}\|u\|_{X_\perp}\le\frac{C(\oo, p, \theta)}{|\lambda|}\|\lambda u
 +d\Delta u\|_{X_\perp},\;\;\forall\, u\in Y_\perp.\end{align*}
This shows that $-d_2^k\Delta:Y_\perp\to X_\perp$ is a sectorial operator  and ${\rm e}^{d_2^k\Delta t}$ is an analytic semigroup in $X_\perp$. Let $\mu>0$ be the second eigenvalue of the operator $-\Delta$ in $\Omega$ under homogeneous Neumann boundary condition. Then $\mathrm{Re}\lambda>d_2^k\mu$ for all $\lambda\in\sigma(-d_2^k\Delta)$, and thereby
 \bee
 \|(-d_2^k\Delta)^\alpha {\rm e}^{d_2^k\Delta t}\|\leq C(\oo, p, \theta,\alpha)t^{-\alpha}{\rm e}^{-d_2^k\mu t},\;\; \forall \,t>0,\label{4.5}\eee
see \cite[Theorem 9.29]{W21} (The detailed proof of this result can be found in \cite[Theorem 7.1.9]{Wangc}). Apply the estimate \qq{4.5} to the solution $W_k$ of \qq{4.3} to derive
  \begin{align*}
 \|W_k(\cdot, t)\|_{X_\perp^\alpha}\le\;&\|(-d_2^k\Delta)^\alpha {\rm e}^{d_2^k\Delta t}W_k(\cdot, 0)\|_{X_\perp}
 +\int_0^t\|(-d_2^k\Delta)^\alpha {\rm e}^{d_2^k\Delta(t-s)}[g_k(\cdot, s)-\ol g_k(s)]\|_{X_\perp}\mathrm{d}s\\
\le\;& Ct^{-\alpha}{\rm e}^{-d_2^k\mu t}\|W_k(\cdot, 0)\|_{X_\perp}
 +C\int_0^t (t-s)^{-\alpha}{\rm e}^{-d_2^k\mu(t-s)}\|g_k(\cdot, s)-\ol g_k(s)\|_{X_\perp}\mathrm{d}s.
  \end{align*}
Thus, for $t\in[T,2T]$, we have
   \begin{align*}
 \|W_k(\cdot, t)\|_{X_\perp^\alpha}
 \le\;& CT^{-\alpha}{\rm e}^{-d_2^k\mu T}\|W_k(\cdot, 0)\|_{X_\perp}
 +CM\int_0^t (t-s)^{-\alpha}{\rm e}^{-d_2^k\mu(t-s)}\mathrm{d}s\\
 \le\;& CT^{-\alpha}{\rm e}^{-d_2^k\mu T}\|W_k(\cdot, 0)\|_{X_\perp}+CM \Gamma(1-\alpha)(d_2^k\mu)^{\alpha-1},
 \end{align*}
where $M=\max_{[0,2T]}\|g_k(\cdot, s)-\ol g_k(s)\|_{X_\perp}$. Since   $X_\perp^\alpha\hookrightarrow C^\alpha(\boo)$, one has 
  \[\max_{[T,2T]}\|W_k(\cdot, t)\|_{C^\alpha(\boo)}\le C\max_{[T,2T]}\|W_k(\cdot, t)\|_{X_\perp^\alpha}\to 0\]
as $k\to\yy$. By periodicity, we also obtain 
 \bee
  \max_{[0,T]}\|W_k(\cdot, t)\|_{C^\alpha(\boo)}=\max_{[T,2T]}\|W_k(\cdot, t)\|_{C^\alpha(\boo)}\to 0\;\;\;\text{as}\;\;k\to\yy.
  \label{4.6}\eee\www

{\it Step 3}. Notice that $\ol g_k(t)$ and $\overline V_k$ are uniformly bounded on $[0,T]$. In view of \qq{4.4}, we derive that $\overline V_k$ is equicontinuous on $[0,T]$. By the Arzel\`a--Ascoli theorem, there exist a subsequence, still denoted by $\{\overline V_k\}$, such that $\ol V_k(t)\to r(t)$ in $C([0,T])$ for some $r(t)\in C_T([0,T])$ as $k\to\infty$.  This combined with \qq{4.6} yields $V_k\to r(t)$ uniformly on $\ol Q_T$ as $k\to\yy$, which together with \qq{4.2} gives
   \begin{align*}
 \frac 1{|\oo|}\int_\oo V_k(m_2-U_k-V_k)=\;&\frac 1{|\oo|}\int_\oo[V_k-\ol V_k(t)](m_2-U_k)+\ol V_k(t)[\ol m_2(t)-\ol U(t)]\\
 &+\ol V_k(t)[\ol U(t)-\ol U_k(t)]-\ol V_k^2(t)-\frac 1{|\oo|}\int_\oo V_k[V_k-\ol V_k(t)]\\
 \to\;& r[\ol m_2(t)-\ol U(t)]-r^2\;\;\;\text{as}\;\;k\to\yy
 \end{align*}
by \qq{4.2}. This combined with \qq{4.4} implies that $r(t)$ satisfies the third equation of \qq{4.1}. Apply the regularity theory to \qq{4.1} to derive that $U\in C^{2+\alpha,1+\frac\alpha2}(\ol Q_T)$ and $r\in C^1([0,T])$.\vspace{-1mm}
\end{proof}

The following Lemmas \ref{lm4.2}-\ref{lm4.4} are devoted to the stability of $(\theta_{d_1,m_1},0)$ and $(0,\theta_{d_2,m_2})$.\vspace{-1mm}

\begin{lemma}\label{lm4.2} The following assertions hold:
\begin{enumerate}
\item[\rm(1)] For any given $b>a>0$, there exists $d_2^{a,b}>0$ such that $(\theta_{d_1,m_1},0)$ is globally asymptotically stable for all $(d_1,d_2)\in[a,b]\times(d_2^{a,b},\infty)$.
\item[\rm(2)] If $m_1(x,t)-m_2(x,t)$ is spatially homogeneous, then there exists $\bar d\gg 1$, depending only on $m_1$ and $m_2$, such that $(\theta_{d_1,m_1},0)$ is globally asymptotically stable for all $d_2>d_1>\bar d$.\end{enumerate}
\end{lemma}\www

\begin{proof}\; (1) Condition {\bf(M)} and the equation of $\theta_{d_1,m_1}$ imply that
  \[\int_{Q_T}(m_2-\theta_{d_1,m_1})=\int_{Q_T}(m_1-\theta_{d_1,m_1})
  =-\int_{Q_T}\kk|\frac{\nabla\theta_{d_1,m_1}}{\theta_{d_1,m_1}}\rr|^2<0.\]
Then by Proposition \ref{lm2.1}(3),
  \[\lim_{d_2\to\infty}\mu(d_2,m_2-\theta_{d_1,m_1})=-\frac{1}{T|\Omega|}
 \int_{Q_T}(m_2-\theta_{d_1,m_1})>0.\]
Moreover, Lemma \ref{lm2.10} shows that $\theta_{d_2,m_2}\to l_2$ in $C(\ol Q_T)$ as $d_2\to\infty$, where $ l_2$ is the unique positive $T$-periodic solution to \eqref{1.11} with $m=m_2$. By condition {\bf(M)}, we have $\int_{Q_T}(m_1- l_2)=0$. Since $m_1- l_2$ is spatially heterogeneous, it follows from Proposition \ref{lm2.1}(1) and (2) that
   \[\lim_{d_2\to\infty}\mu(d_1,m_1-\theta_{d_2,m_2})
   =\mu(d_1,m_1- l_2)<-\frac{1}{T|\Omega|}\int_{Q_T}(m_1- l_2)=0.\]
By virtue of continuity, there exists a constant $d_2^{a,b}>0$ such that $\mu(d_2,m_2-\theta_{d_1,m_1})>0$ and $\mu(d_1,m_1-\theta_{d_2,m_2})<0$ for all $(d_1,d_2)\in[a,b]\times(d_2^{a,b},\infty)$. Using Lemma \ref{lm4.1} and following the argument in the proof of \cite[Theorem 5.3(a)]{HMP01}, it can be shown that \eqref{1.5} admits no positive solution for sufficiently large $d_2$. By further increasing $d_2^{a,b}$, Proposition \ref{p2.4}(1) implies that $(\theta_{d_1,m_1},0)$ is globally asymptotically stable for all $(d_1,d_2)\in[a,b]\times(d_2^{a,b},\infty)$.

(2) Assume that $m_1-m_2$ is spatially homogeneous. Then $\widehat{m_1-m_2}\equiv 0$. Following the proof of \cite[Claim 5.1(ii)]{BHN23}, we have that there exists $\bar d\gg 1$, depending only on $m_i$, such that
  \bee \frac{\partial \mu(d_1,m_2-\theta_{d_2,m_2})}{\partial d_1}>0,\;\;\; \frac{\partial\mu(d_2,m_1-\theta_{d_1,m_1})}{\partial d_2}>0,\;\; \forall\, d_1, d_2>\bar d.\label{4.7} \eee
Notice that that $\mu(d_i,m_i-\theta_{d_i,m_i})=0$ for $i=1,2$. Applying Proposition \ref{lm2.1}(4) and \qq{4.7} in succession, we obtain that, for all $d_2>d_1>\bar d$,
 \bee\begin{cases}
  \mu(d_2,m_2-\theta_{d_1,m_1})=\mu(d_2,m_1-\theta_{d_1,m_1})
 >\mu(d_1,m_1-\theta_{d_1,m_1})=0,\\
  \mu(d_1,m_1-\theta_{d_2,m_2})=\mu(d_1,m_2-\theta_{d_2,m_2})
  <\mu(d_2,m_2-\theta_{d_2,m_2})=0.\www\end{cases}\label{4.8}\eee

Next, we show the nonexistence of positive solution to \eqref{1.5} for all $d_2>d_1>\bar d$. Suppose for contradiction that \eqref{1.5} admits a positive solution $(U,V)$ for some $d_2>d_1>\bar d$. Then
  \[\mu(d_1,m_1-U-V)=\mu(d_2,m_2-U-V)=0,\]
and hence
  \[\mu(d_1,m_2-U-V)=\mu(d_2,m_1-U-V)=0\]
by Proposition \ref{lm2.1}(4). Following the proof of \cite[Claim 5.1(i)]{BHN23}, we derive that (by choosing $\bar d$ even larger if necessary)
  \[\frac{\partial \mu(d,m_2-U-V)}{\partial d}>0,\;\; \forall\, d>\bar d,\]
which leads to 
  \[\mu(d_2,m_2-U-V)>\mu(d_1,m_2-U-V)=0\] 
since $d_2>d_1>\bar d$. We get a contradiction, and by Proposition \ref{p2.4}(1), $(\theta_{d_1,m_1},0)$ is globally asymptotically stable for all $d_2>d_1>\bar d$.\zzz
\end{proof}

The remainder of this section focuses on the signs of $\mu(d_2,m_2-\theta_{d_1,m_1})$ and $\mu(d_1,m_1-\theta_{d_2,m_2})$, which determine the local stability of $(\theta_{d_1,m_1},0)$ and $(0,\theta_{d_2,m_2})$, respectively.\www

\begin{lemma}\label{lm4.3} The following statements are valid:
\begin{enumerate}
\item[\rm(1)] If $\widehat{m}_2(x)\not\equiv\mathrm{constant}$ on $\overline\Omega$,  then there exist $\bar d_1^1\gg 1$ and $0<\varepsilon_1\ll 1$ such that    
     \[\mu(d_2,m_2-\theta_{d_1,m_1})<0,\;\;\forall\;(d_1,d_2)\in(\bar d_1^1,\infty)\times(0,\varepsilon_1).\]
\item[\rm(2)] If $\widehat{m}_1(x)\not\equiv\mathrm{constant}$ on $\overline\Omega$,  then there exist $\bar d_2^1\gg 1$ and $0<\ep\ll 1$ such that \[\mu(d_1,m_1-\theta_{d_2,m_2})<0,\;\;\forall\;(d_1,d_2)\in(0,\ep)\times(\bar d_2^1,\infty).\]
\end{enumerate}
\end{lemma}\zzz

\begin{proof}\; We only prove part (1) as part (2) can be proved analogously.

For part (1), there exists $x_0\in\Omega$ satisfying
\[\widehat{m}_2(x_0)-\ol{\widehat{m}}_1=\widehat{m}_2(x_0)-\ol{\widehat{m}}_2:=2\delta_0>0.\]
By Lemma \ref{lm2.10}, we have $\lim_{d_1\to\infty}\|\theta_{d_1,m_1}- l_1\|_{C(\ol Q_T)}=0$. Hence there exists $\bar d_1^1\gg1$ such that for all $d_1>\bar d_1^1$, 
  \[m_2-\theta_{d_1,m_1}>m_2- l_1-\delta_0\;\;\;\text{on}\;\;\overline Q_T,\]
and hence 
  \[\mu(d_2,m_2-\theta_{d_1,m_1})<\mu(d_2,m_2- l_1-\delta_0)\]
(cf. Proposition \ref{lm2.1}(2)). Moreover, since $\widehat{l}_1=\widehat{\ol m}_1$,  Proposition \ref{lm2.1}(3) gives 
  \[\lim_{d_2\to 0}\mu(d_2,m_2-l_1-\delta_0)
  =-\max_{x\in\overline\Omega}(\widehat{m}_2-\widehat{\ol m}_1-\delta_0)\leq -(\widehat{m}_2(x_0)-\widehat{\ol m}_2)+\delta_0<0.\]
There exists $0<\varepsilon_1\ll 1$ such that  $\mu(d_2,m_2-\theta_{d_1,m_1})<0$ for all $(d_1,d_2)\in(\bar d_1^1,\infty)\times(0,\varepsilon_1)$.
\end{proof}

\begin{lemma}\label{lm4.4} For any fixed $\bar d_1>0$, there exists $d_{2,\bar d_1}>0$ such that the following statements hold.
\begin{enumerate}
\item[\rm(1)] If conditions {\bf(M1)}--{\bf(M2)} hold, then there exists a strictly decreasing $C^1$ function:
\[\mathsf{d}_1:(d_{2,\bar d_1},\infty)\to (0,\infty)\quad \mbox{with}\quad \mathsf{d}_1(d_2)=O\kk(d_2^{-1}\rr)\;\; \mbox{as } d_2\to\infty\]
such that for every $d_2>d_{2,\bar d_1}$,
 \begin{align*}
 \mu(d_2,m_2-\theta_{d_1,m_1})<&0\;\;\;\text{when}\;\;d_1\in(0,\mathsf{d}_1(d_2)),\\ \mu(d_2,m_2-\theta_{d_1,m_1})>&0\;\;\;\text{when}\;\;d_1\in(\mathsf{d}_1(d_2),\bar d_1].\vspace{-2mm}\end{align*}
\item[\rm(2)] If $\widehat{m}_1(x)\not\equiv\mathrm{constant}$ on $\overline\Omega$, then
     \[\mu(d_1,m_1-\theta_{d_2,m_2})<0,\;\;\forall\;(d_1,d_2)\in(0,\bar d_1]\times(d_{2,\bar d_1},\infty).\]
\end{enumerate}
\end{lemma}

\begin{proof}\; 
(1) Let $\varphi$ and $\psi$ be the normalized positive principal eigenfunction and the normalized positive adjoint principal eigenfunction corresponding to $\mu=\mu(d_2,m_2-\theta_{d_1,m_1})$, respectively. Then 
  \begin{equation}\label{4.9} \begin{cases}
 \varphi_t-d_2\Delta \varphi-(m_2-\theta_{d_1,m_1})\varphi=\mu\varphi & \text{in}\;\;Q_T,\\
\partial_\nu\varphi=0 & \text{on}\;\;S_T,\\
\varphi(x,0)=\varphi(x,T) & \text{on\ \ }\ol\Omega,
  \end{cases}\end{equation}
and
  \[\begin{cases}
-\psi_t-d_2\Delta \psi-(m_2-\theta_{d_1,m_1})\psi=\mu\psi & \text{in\ \ }Q_T, \\
\partial_\nu\psi=0 & \text{on\ \ }S_T, \\
\psi(x,0)=\psi(x,T) & \text{on\ \ }\ol\Omega.
\end{cases} \]

{\it Step 1}. Differentiating the equation of $\varphi$ in \eqref{4.9} with respect to $d_1$ gives that
  \[ \begin{cases}
\varphi'_t-d_2\Delta \varphi'-(m_2-\theta_{d_1,m_1})\varphi'
+\varphi\theta_{d_1,m_1}'=\mu\varphi'+\mu'\varphi & \text{in\ \ }Q_T,\\
\partial_\nu\varphi'=0 & \text{on\ \ }S_T \\
\varphi'(x,0)=\varphi'(x,T) & \text{on\ \ }\boo,
  \end{cases} \]
where $'$ denotes the derivative with respect to $d_1$. Multiplying the equation of $\varphi'$ by $\psi$ and the equation of $\psi$ by $\varphi'$, integrating them over $Q_T$ and subtracting the results, we get
  \begin{equation}\label{4.10}
\mu'\int_{Q_T}\varphi\psi=\int_{Q_T}\varphi\psi\theta_{d_1,m_1}'.
  \end{equation}
Differentiating the equation of $\theta_{d_1,m_1}$ with respect to $d_1$ yields that
  \[\begin{cases}
\partial_t\theta_{d_1,m_1}'-d_1\Delta \theta_{d_1,m_1}'=\Delta \theta_{d_1,m_1}+\theta_{d_1,m_1}'(m_1-2\theta_{d_1,m_1}) & \text{in\ \ }Q_T,\\[.1mm]
\partial_\nu\theta_{d_1,m_1}'=0 & \text{on\ \ }S_T, \\[.1mm]
\theta_{d_1,m_1}'(x,0)=\theta_{d_1,m_1}'(x,T) & \text{on\ \ }\boo.
  \end{cases}\]
In view of the condition {\bf(M1)}, Proposition \ref{lm2.8} and \cite[Lemma 4.2]{BHN23}, it follows that
  \begin{equation}\label{4.11}
\lim_{d_1\to 0}\|\theta_{d_1,m_1}'-\Psi_1\|_{\infty}=0.
  \end{equation}
Dividing the equation of $\Psi_1$ by $\Phi_1$, and integrating the result over $(0,T)$ and using the equation of $\Phi_1$, we have
  \[\int_0^T\frac{\Psi_1(m_1-2\Phi_1)}{\Phi_1}+\int_0^T\frac{\Delta \Phi_1}{\Phi_1}=\int_0^T\frac{\partial_t\Psi_1}{\Phi_1}
  =\int_0^T\frac{\Psi_1\partial_t\Phi_1}{\Phi_1^2}
  =\int_0^T\frac{\Psi_1(m_1-\Phi_1)}{\Phi_1}.\]
It follows that
  \begin{equation}\label{4.12}
\int_{Q_T}\Psi_1=\int_{Q_T}\frac{\Delta \Phi_1}{\Phi_1}=\int_{Q_T}\frac{|\nabla \Phi_1|^2}{\Phi_1^2}>0.\vspace{-2mm}
  \end{equation}

{\it Step 2}. By Lemma \ref{lm2.7},
 \[\|m_2-\theta_{d_1,m_1}\|_{L^{\infty}(Q_T)}+\|\partial_tm_2-\partial_t
 \theta_{d_1,m_1}\|_{L^2(Q_T)}\le \|m_1\|_{C^{0,1}(\ol Q_T)}+
  \|m_2\|_{C^{0,1}(\ol Q_T)}=:M. \]
Thanks to Proposition \ref{p2.2}, there exist constants $C_M$ and $d_{2,M}$ such that for $d_2\ge d_{2,M}$,
  \bee\label{4.13}\begin{cases}
\dd\varphi(x,t)=\ol\varphi(t)+{\ol\varphi(t)\Gamma_{m_2-\theta_{d_1,m_1}}}d_2^{-1}
 +{\varphi_*(x,t)}d_2^{-2},\\[1mm]
  \dd\psi(x,t)=\ol\psi(t)+{\ol\psi(t)\Gamma_{m_2-\theta_{d_1,m_1}}}d_2^{-1}
  +{\psi_*(x,t)}d_2^{-2},
  \end{cases}\eee
where
  \begin{align*}
&\ol\varphi(t)=\ol\varphi(0)\exp\kk(\int_0^t\big[\,\overline{(m_2
-\theta_{d_1,m_1})}(s)+\mu\big]\mathrm ds\rr)+O\kk(d_2^{-1/2}\rr),\\[.1mm]
&\ol\psi(t)=\ol\psi(0)\exp\kk(-\int_0^t\big[\,\overline{(m_2-\theta_{d_1,m_1})}(s)
+\mu\big]\mathrm ds\rr)+O\kk(d_2^{-1/2}\rr),  \\
&\|\varphi_*\|_{L^2((0,T),H^1(\Omega))}\leq C_M,\; \; \|\varphi_*\|_{L^2((0,T),H^1(\Omega))}\leq C_M, \\[.2mm]
&1/C_M\leq \ol\varphi(t),\ol\psi(t)\leq C_M,\;\; \forall\, t\in[0,T].
  \end{align*}
Enlarge $d_{2,M}$ if necessary. It follows that there exists $\delta_M>0$ such that
$\ol\varphi(0),\ol\psi(0)\ge \delta_M$ for all $d_2\ge d_{2,M}$. Substituting \eqref{4.13} into \eqref{4.10} and using \eqref{4.12}, we deduce that
   \begin{align}
\mu'\int_{Q_T}\varphi\psi
=\;&\int_{Q_T}\varphi\psi\theta_{d_1,m_1}'=\int_{Q_T}\varphi\psi \Psi_1+\int_{Q_T}\varphi\psi (\theta_{d_1,m_1}'-\Psi_1)\nm\\[.1mm]
=\;&\ol\varphi(0)\ol\psi(0)\int_{Q_T}\Psi_1+O\kk(d_2^{-1/2}\rr)
+\rho(d_1)\nm\\[.1mm]
=\;&\ol\varphi(0)\ol\psi(0)\int_{Q_T}\frac{|\nabla \Phi_1|^2}{\Phi_1^2}+O\kk(d_2^{-1/2}\rr)+\rho(d_1),
   \label{4.14}\end{align}
where $\rho(d_1)=\int_{Q_T}\varphi\psi (\theta_{d_1,m_1}'-\Psi_1)\to 0$ as $d_1\to 0$ by \eqref{4.11}. Enlarge $d_{2,M}$ if necessary. Then, by \qq{1.7} and \qq{4.14},  there exists $0<\ep<\bar d_1$ such that
 \bee
 \mu'>0,\;\;\forall\, (d_1,d_2)\in(0,\ep]\times(d_{2,M},\infty).
 \label{4.15}\vspace{-2mm}\eee

{\it Step 3}. Let $d_2^{(\ep,\bar d_1)}$ be determined by Lemma \ref{lm4.2}. Then from the proof of Lemma \ref{lm4.2} we have
 \bee
 \mu(d_2,m_2-\theta_{d_1,m_1})>0,\;\;\forall\, (d_1,d_2)\in[\ep,\bar d_1]\times(d_2^{(\ep,\bar d_1)},\infty).
  \label{4.16}\eee
By Lemma \ref{lm2.5}, we have $\theta_{d_1,m_1}\to \Phi_1$ in $C(\ol Q_T)$ as $d_1\to 0$. In view of the condition {\bf(M2)}, $m_2-\Phi_1$ is spatially heterogeneous. This together with Proposition \ref{lm2.1}(1) implies that
 \[\lim_{d_1\to 0}\mu(d_2,m_2-\theta_{d_1,m_1})=\mu(d_2,m_2-\Phi_1)
    <-\frac{1}{T|\Omega|}\int_{Q_T}(m_2-\Phi_1)\mathrm{d}t=0.\]
Thus, for each $d_2>d_2^{(\ep,\bar d_1)}$, there exists $0<d_1<\ep$ such that $\mu(d_2,m_2-\theta_{d_1,m_1})<0$. Combining this with \qq{4.16}, we see that for each $d_2>d_2^{(\ep,\bar d_1)}$, $\mu(d_2,m_2-\theta_{d_1,m_1})$ changes sign at least once as $d_1$ increases from $0$ to $\ep$.

Let $d_{2,\bar d_1}=\max\{d_{2,M}, d_2^{(\ep,\bar d_1)}\}$. By the implicit function theorem, there exists a unique $C^1$ function $\mathsf{d}_1:(d_{2,\bar d_1},\infty)\to(0,\ep)$ such that for all $(d_1,d_2)\in(0,\ep)\times(d_{2,\bar d_1},\infty)$,
  \bee\label{4.17}
  \mu(d_2,m_2-\theta_{d_1,m_1})=0 \Longleftrightarrow d_1=\mathsf{d}_1(d_2).\eee
Moreover, following the proof of \cite[Claim 5.5]{BHN23}, we see that $\mathsf{d}_1:(d_{2,\bar d_1},\infty)\to(0,\ep)$ is strictly decreasing and
 \[\mathsf{d}_1(d_2)=O\kk(d_2^{-1}\rr)\quad \mbox{as}\;\; d_2\to\infty.\]
It follows from \qq{4.15} and \qq{4.17} that for each $d_2>d_{2,\bar d_1}$, $\mu(d_2,m_2-\theta_{d_1,m_1})<0$ when $d_1\in(0,\mathsf{d}_1(d_2))$ and $\mu(d_2,m_2-\theta_{d_1,m_1})>0$ when $d_1\in(\mathsf{d}_1(d_2),\bar d_1)$. This completes the proof of part (1).

(2) By Lemma \ref{lm4.3}(2), there exist $\bar d_2^1\gg 1$ and $0<\ep\ll 1$ such that $\mu(d_1,m_1-\theta_{d_2,m_2})<0$ for all $(d_1,d_2)\in(0,\ep)\times(\bar d_2^1,\infty)$. Let $d_2^{(\ep,\bar d_1)}$ be determined in Lemma \ref{lm4.2}. Then the proof of Lemma \ref{lm4.2} shows that $\mu(d_1,m_1-\theta_{d_2,m_2})<0$ for all $(d_1,d_2)\in[\ep,\bar d_1]\times(d_2^{(\ep,\bar d_1)},\infty)$. Setting $d_{2,\bar d_1}=\bar d_2^1+d_2^{(\ep,\bar d_1)}$, we obtain the desired conclusion.\vspace{-1mm}
\end{proof}

The following lemmas \ref{lm4.5}-\ref{lm4.6} focus on the properties of $\mu(d_1,m_1-\theta_{d_2,m_2})$ when $d_1$ is large and $d_2$ is small. For each $t\in[0,T]$, let $\Gamma_{12}(\cdot,t)$ and $\Gamma_{a_{d_2}}(\cdot,t)$ denote the unique solution of \eqref{1.10} with $m$ replaced by $m_1-\Phi_{m_2}$ and $a_{d_2}:=m_1-\theta_{d_2,m_2}$, respectively. Define
  \[\mathcal E_{12}:=\frac1{T|\Omega|}\int_{Q_T}|\nabla\Gamma_{12}|^2 \mathrm dx\mathrm dt, \quad \mathcal I_2
  :=\frac1{T|\Omega|}\int_{Q_T}\left|\frac{\nabla\Phi_{m_2}}{\Phi_{m_2}}
  \right|^2\mathrm dx\mathrm dt.\]
It is worth emphasizing that $\mathcal E_{12}$ and $\mathcal I_2$ are independent of both  $d_1$ and $d_2$.\vspace{-1mm}

\begin{lemma}[A mixed large-$d_1$ and small-$d_2$ expansion] \label{lm4.5}
Assume that conditions {\bf(M1)}--{\bf(M2)} hold. Then there exist positive constants $\delta_0,D_0$ and $C_*$ independent of $d_1,d_2$, and continuous functions $\rho_1,\rho_2,\eta_1:[0,\delta_0]\to\mathbb R$ satisfying $\rho_1(0)=\rho_2(0)=\eta_1(0)=0$ such that, for every $d_1>D_0$ and $0<d_2\le\delta_0$,
  \begin{align}
 &\mu(d_1,m_1-\theta_{d_2,m_2})=\mathcal I_2d_2-{\mathcal E_{12}}d_1^{-1}
  +d_2\rho_1(d_2)+{\rho_2(d_2)}d_1^{-1}+R(d_1,d_2),\label{4.18}\\[.6mm]
  &\left|\partial_{d_2}\mu(d_1,m_1-\theta_{d_2,m_2})-\mathcal I_2-\eta_1(d_2)\right|
\le{C_*}d_1^{-1},\label{4.19}
 \end{align}
and
  \begin{equation}\label{4.20}
\left|\partial_{d_1}\mu(d_1,m_1-\theta_{d_2,m_2})-\frac1{T|\Omega|d_1^2}
\int_{Q_T}|\nabla\Gamma_{a_{d_2}}|^2\,\mathrm dx\mathrm dt\right|
 \le{C_*}d_1^{-3},
 \end{equation}
where $\mathcal E_{12}>0,\mathcal I_2>0$ and
\begin{equation}\label{4.21}
  \sup_{0<d_2\le\delta_0}|R(d_1,d_2)|\leq {C_*}d_1^{-2}.
\end{equation}
In particular, as $d_1\to\infty$ and $d_2\to0^+$,
  \begin{align}\label{4.22}
 &\mu(d_1,m_1-\theta_{d_2,m_2})=\mathcal I_2d_2-{\mathcal E_{12}}d_1^{-1}
+o(d_2)+o\left(d_1^{-1}\right),\\[.1mm]
 \label{4.23}
&\partial_{d_2}\mu(d_1,m_1-\theta_{d_2,m_2})=\mathcal I_2+o(1),
\end{align}
and whenever $d_2=O(d_1^{-1})$,
\begin{equation}\label{4.24}
\partial_{d_1}\mu(d_1,m_1-\theta_{d_2,m_2})
={\mathcal E_{12}}d_1^{-2}+o\left(d_1^{-2}\right).
\end{equation}
\end{lemma}

\begin{proof}\; We divide the proof into four steps.

{\it Step 1}. We first establish the positivity of $\mathcal E_{12}$ and $\mathcal I_2$. By the condition {\bf(M2)}, the function $m_1-\Phi_{m_2}$ is spatially
heterogeneous, which implies that
  \[m_1(x,t)-\Phi_{m_2}(x,t)-\overline{(m_1-\Phi_{m_2})}(t)\not\equiv 0\;\;\;\text{in}\;\;Q_T.\]
This combined with the definition of $\Gamma_{12}$ yields that $\Gamma_{12}$ is spatially heterogeneous. Hence,
  \[\mathcal E_{12}=\frac1{T|\Omega|}\int_{Q_T}|\nabla\Gamma_{12}|^2
\,\mathrm dx\mathrm dt>0.\]

We next prove that $\mathcal I_2>0$. Suppose, to the contrary, that $\mathcal I_2=0$. Then $\nabla\Phi_{m_2}\equiv 0$ in $Q_T$, and hence $\Phi_{m_2}$ is independent of $x$. Therefore,
  \[m_2(x,t)=\frac{(\Phi_{m_2})_t(t)}{\Phi_{m_2}(t)}+\Phi_{m_2}(t)\]
is independent of $x$. This contradicts the fact that $\nabla m_2\not\equiv 0$ on $\overline Q_T$. Thus $\mathcal I_2>0$.

{\it Step 2}. We next derive the expansion of the space-time average. Recall that $\theta_{d_2,m_2}$ satisfies \eqref{1.2} with $d=d_2,m=m_2$. Dividing the equation of $\theta_{d_2,m_2}$ by $\theta_{d_2,m_2}$ and integrating the result over $Q_T$, we find
  \[\int_{Q_T}(m_2-\theta_{d_2,m_2})\,\mathrm dx\mathrm dt=-d_2\int_{Q_T}\left|
\frac{\nabla\theta_{d_2,m_2}}{\theta_{d_2,m_2}}\right|^2\,\mathrm dx\mathrm dt.\]
Using again the assumption {\bf(M)}, we obtain
  \begin{equation*}
-\frac1{T|\Omega|}\int_{Q_T}(m_1-\theta_{d_2,m_2})\,\mathrm dx\mathrm dt=
\frac{d_2}{T|\Omega|}\int_{Q_T}\left|
\frac{\nabla\theta_{d_2,m_2}}{\theta_{d_2,m_2}}\right|^2\,\mathrm dx\mathrm dt.
  \end{equation*}
Define
 \begin{align*}
 \rho_1(d_2):=\frac1{T|\Omega|}\int_{Q_T}\left|\frac{\nabla\theta_{d_2,m_2}}
{\theta_{d_2,m_2}}\right|^2\,\mathrm dx\mathrm dt-\mathcal I_2\;\;\;\text{for}\;\; d_2>0.\end{align*}
Then we have
 \begin{equation}\label{4.25}
-\frac1{T|\Omega|}\int_{Q_T}(m_1-\theta_{d_2,m_2})\,\mathrm dx\mathrm dt
=\mathcal I_2d_2+d_2\rho_1(d_2).
\end{equation}
Moreover, by Proposition \ref{lm2.8}, $\theta_{d_2,m_2}\to \Phi_{m_2}$ in $C([0,T];C^1(\overline\Omega))$ as $d_2\to 0^+$, which implies $\rho_1(d_2)\to 0$ as $d_2\to 0^+$. Thus $\rho_1$ is continuous on $[0,\delta_0]$ by setting $\rho_1(0)=0$.

{\it Step 3}. We next establish the uniform expansion for large $d_1$. Given  $a_{d_2}=m_1-\theta_{d_2,m_2}$, by Lemma \ref{lm2.7} we have
  \begin{align*}
  \|a_{d_2}\|_{L^\infty(Q_T)}&\le\|m_1\|_{L^\infty(Q_T)}+\|m_2\|_{L^\infty(Q_T)},\\
\|(a_{d_2})_t\|_{L^2(0,T;L^2(\Omega))}&\le\|(m_1)_t\|_{L^2(0,T;L^2(\Omega))}+
C\|m_2\|_{L^\infty(Q_T)}^2,\end{align*}
where $C>0$ is independent of $d_2$. Thus, 
  \[ M_0:=\sup_{0\le d_2\le\delta_0}\left\{\|a_{d_2}\|_{L^\infty(Q_T)}
+\|(a_{d_2})_t\|_{L^2(0,T;L^2(\Omega))}\right\}<\infty.\]\zzz

Let $\varphi\in C^{2+\alpha,1+\alpha/2}(\overline Q_T)$ be the positive principal eigenfunction corresponding to $\mu(d_1,m_1-\theta_{d_2,m_2})$, normalized by $\|\varphi\|_{L^2(Q_T)}=1$. Then $\varphi$ satisfies
  \begin{equation}\label{4.26}
 \begin{cases}
 \varphi_t-d_1\Delta\varphi-a_{d_2}\varphi=
\mu(d_1,m_1-\theta_{d_2,m_2})\varphi, & \text{in } Q_T, \\
\partial_{\nu}\varphi=0, & \text{on }S_T, \\
\varphi(x,0)=\varphi(x,T), & \text{on }\overline\Omega.
\end{cases}
 \end{equation}
By applying Proposition \ref{p2.2} with the common bound $M_0$, we obtain constants  $D_0>0$ and $C_M>0$, depending only on $M_0$, $\Omega$ and $T$, such that for all $d_1>D_0$ and $0\leq d_2\le\delta_0$,
\begin{equation}\label{4.27}
\varphi(x,t)=\overline\varphi(t)+{\overline\varphi(t)
\Gamma_{a_{d_2}}(x,t)}d_1^{-1}+{\varphi_2(x,t)}d_1^{-2},
\end{equation}
where
\begin{equation}\label{4.28}
1/{C_M}\leq \overline\varphi(t) \leq C_M,\quad
\|\varphi_2\|_{L^2(0,T;H^1(\Omega))}\leq C_M.
\end{equation}
Averaging the first equation of \eqref{4.26} over $\Omega$, dividing by $\overline\varphi(t)$, and integrating over $[0,T]$, we obtain
  \begin{equation}\label{4.29}
\mu(d_1,m_1-\theta_{d_2,m_2})=-\frac1{T|\Omega|}\int_{Q_T}a_{d_2}\,\mathrm dx\mathrm dt
-\frac1{T|\Omega|d_1}\int_{Q_T}a_{d_2}\Gamma_{a_{d_2}}\,\mathrm dx\mathrm dt
+R(d_1,d_2),\end{equation}
where
\begin{equation}\label{4.30}
  R(d_1,d_2)=-\frac1{T|\Omega|d_1^2}\int_{Q_T}a_{d_2}(x,t)\frac{\varphi_2(x,t)}
{\overline\varphi(t)}\,\mathrm dx\mathrm dt.
\end{equation}
Noticing that $\Gamma_{a_{d_2}}$ satisfies \eqref{1.10} with $m$ replaced by $a_{d_2}$, we have
  \begin{equation}\label{4.31}
  \int_{Q_T}a_{d_2}\Gamma_{a_{d_2}}\,\mathrm dx\mathrm dt=
\int_{Q_T}|\nabla\Gamma_{a_{d_2}}|^2\,\mathrm dx\mathrm dt.
  \end{equation}
By use of \eqref{4.28}, H\"older's inequality and the definition of $M_0$, there exists a constant $C_*>0$, depending only on $M_0$, $\Omega$ and $T$, such that \eqref{4.21} holds.

Using Lemma \ref{lm2.5} or Proposition \ref{lm2.8}, we have $a_{d_2}\to a_0:=m_1-\Phi_{m_2}$ in $L^2(Q_T)$ as $d_2\to 0^+$. The elliptic estimate for the Neumann problem gives
  \[\|\nabla(\Gamma_{a_{d_2}}-\Gamma_{a_0})\|_{L^2(Q_T)}\leq C\|a_{d_2}-a_0\|_{L^2(Q_T)}\to 0\quad \text{as }d_2\to 0^+.\]
As $\Gamma_{a_0}=\Gamma_{12}$, it follows that $\frac1{T|\Omega|}\int_{Q_T}|\nabla\Gamma_{a_{d_2}}|^2\,\mathrm dx\mathrm dt\to \mathcal E_{12}$ as $d_2\to 0^+$. Define
  \begin{align*}
 \rho_2(d_2)=\mathcal E_{12}-\frac1{T|\Omega|}
\int_{Q_T}|\nabla\Gamma_{a_{d_2}}|^2\,\mathrm dx\mathrm dt\;\;\;\text{for}\;\; d_2>0.\end{align*}
Then $\rho_2(d_2)\to 0$ as $d_2\to 0^+$, and $\rho_2$ is continuous on $[0,\delta_0]$ by setting $\rho_2(0)=0$.

Combining \eqref{4.25} and \eqref{4.29}--\eqref{4.31}, we obtain \eqref{4.18}. The estimate \eqref{4.21} and the limits of $\rho_1(d_2),\rho_2(d_2)$ imply
\eqref{4.22}. 

{\it Step 4}. We finally deduce the rigorous derivative estimates. We first consider differentiation with respect to $d_2$. Set
 \[Z_{d_2}:=\partial_{d_2}\theta_{d_2,m_2}.\]
By virtue of the assumption {\bf(M1)} and Proposition \ref{lm2.8}, 
  \[\theta_{d_2,m_2}\to\Phi_{m_2}, \;\;\Delta\theta_{d_2,m_2}\to\Delta\Phi_{m_2}\;\;\;\text{in}\;\; C(\overline Q_T)\]
as $d_2\to 0^+$, where for each $x\in\overline\Omega$, $\Psi_{m_2}$ is the unique $T$-periodic solution of \eqref{2.20} with $\Theta_m=\Phi_{m_2}$. Employing the same argument for \eqref{4.11} and \eqref{4.12}, we derive that
  \begin{equation}\label{4.32}
 \lim_{d\to 0}\|Z_{d_2}-\Psi_{m_2}\|_{C(\overline Q_T)}=0\quad \text{and}\quad \frac1{T|\Omega|}\int_{Q_T}\Psi_{m_2}\,\mathrm dx\mathrm dt=\mathcal I_2.
\end{equation} 

Let $\psi\in C^{2+\alpha,1+\alpha/2}(\overline Q_T)$ be the positive adjoint principal eigenfunction associated with
$\mu(d_1,m_1-\theta_{d_2,m_2})$ normalized by $\|\psi\|_{L^2(Q_T)}=1$. Then $\psi$ satisfies
\begin{equation}\label{4.33}
\begin{cases}
-\psi_t-d_1\Delta\psi-a_{d_2}\psi=\mu(d_1,m_1-\theta_{d_2,m_2})\psi, & \text{in } Q_T, \\
\partial_{\nu}\psi=0, & \text{on }S_T, \\
\psi(x,0)=\psi(x,T), & \text{on }\overline\Omega.
\end{cases}
\end{equation}
Differentiating the first equation in \eqref{4.26} with respect to $d_2$, taking the pairing with $\psi$, and utilizing $\partial_{d_2}a_{d_2}=-\partial_{d_2}\theta_{d_2,m_2}=-Z_{d_2}$, we arrive at 
 \begin{align}
 \partial_{d_2}\mu(d_1,m_1-\theta_{d_2,m_2})
=-\frac{\dd\int_{Q_T}(\partial_{d_2}a_{d_2})
\varphi\psi\,\mathrm dx\mathrm dt}{\dd\int_{Q_T}
\varphi\psi\,\mathrm dx\mathrm dt}
=\frac{\dd\int_{Q_T}Z_{d_2}\varphi\psi\,\mathrm dx\mathrm dt}{\dd\int_{Q_T}\varphi\psi\,\mathrm dx\mathrm dt}.\label{4.34}
 \end{align}
By applying Proposition \ref{p2.2} with the common bound $M_0$, we have that for every $d_1>D_0$ and $0\leq d_2\le\delta_0$,
\begin{equation}\label{4.35}
\psi(x,t)=\overline\psi(t)+{\overline\psi(t)\Gamma_{a_{d_2}}(x,t)}d_1^{-1}+
{\psi_2(x,t)}d_1^{-2},
\end{equation}
where
  \begin{equation*}
1/{C_M}\leq\overline\psi(t)\leq C_M,\quad\|\psi_2\|_{L^2(0,T;H^1(\Omega))}\leq C_M.
\end{equation*}

Averaging the first equation of \eqref{4.26} and \eqref{4.33} over $\Omega$
gives
 \[\begin{cases}\overline\varphi_t=\big(\overline{a_{d_2}}
 +\mu\big)\overline\varphi+r(t),&t\in(0,T],\\[1mm] -\overline{\psi}_t=\big(\overline{a_{d_2}}+\mu\big)
\overline{\psi}+r^*(t),&t\in(0,T],\end{cases}\]
where
\[r(t)=\frac1{|\Omega|}\int_{\Omega} a_{d_2}\big(\varphi-\overline\varphi\big)\,\mathrm dx,\quad r^*(t)=\frac1{|\Omega|}\int_{\Omega}a_{d_2} \big(\psi-\overline{\psi}\big)\,\mathrm dx. \]
The uniform eigenfunction expansions \eqref{4.27} and \eqref{4.35} imply that
\[\|r\|_{L^2(0,T)}+\|r^*\|_{L^2(0,T)}
\leq {C_1}d_1^{-1}\quad\text{for}\;\; d_1>D_0,\;0\leq d_2\leq\delta_0. \]

Set
 \[B(t):=\overline\varphi(t)\,\overline{\psi}(t),\;\;\;\langle\varphi,\psi\rangle=
\int_{Q_T}\varphi\psi\,\mathrm dx\mathrm dt.\]
Then $B'(t)=r\,\overline{\psi}-\overline\varphi\,r^*$, and therefore $\|B'\|_{L^1(0,T)}\leq {C_2}d_1^{-1}$ for $d_1>D_0$ and $0\leq d_2\le\delta_0$. Recalling that $B$ is $T$-periodic, we have 
\begin{equation}\label{4.36}
\|B-\widehat B\|_{L^\infty(0,T)}\leq \|B'\|_{L^1(0,T)} \leq {C_3}d_1^{-1}\quad \text{for}\;\;d_1>D_0,0\leq d_2\le\delta_0,
\end{equation}
where $C_3$ is independent of $d_1$ and $d_2$. On the other hand, the uniform eigenfunction expansions \eqref{4.27} and \eqref{4.35} yield that for $d_1>D_0$ and $0\leq d_2\le\delta_0$,
  \[\|\varphi\psi-B\|_{L^1(Q_T)}=\|\varphi(\psi-\overline\psi)
  +\overline\psi(\varphi-\overline\varphi)\|_{L^1(Q_T)}\leq{C_4}d_1^{-1}.\]
Recall that $\int_\Omega\Gamma_{a_{d_2}}(x,t)\mathrm dx\equiv 0$. Combining \eqref{4.27} and \eqref{4.35}, we conclude that for all $d_1>D_0$ and $0\leq d_2\leq\delta_0$,
  \[\langle\varphi,\psi\rangle=|\Omega|\int_0^T B(t)\mathrm dt+O\left(d_1^{-2}\right)=
T|\Omega|\widehat B+O\left(d_1^{-2}\right).\]
Moreover, the uniform positive lower bounds on
$\overline\varphi$ and $\overline\psi$ yield $\langle\varphi,\psi\rangle\ge c_0>0$ for all sufficiently large $d_1$, uniformly for $0\le d_2\le\delta_0$. Consequently, for $d_1>D_0,0\leq d_2\le\delta_0$,
\begin{align}
\left\|\frac{\varphi\psi}{\langle\varphi,\psi\rangle}-\frac1{T|\Omega|}\right\|_{L^1(Q_T)}
\leq{}&\frac{\|\varphi\psi-B\|_{L^1(Q_T)}}{\langle\varphi,\psi\rangle}+\frac{|\Omega|}{\langle\varphi,\psi\rangle}\|B-\widehat B\|_{L^1((0,T))}+T|\Omega|\left|\frac{\widehat B}{\langle\varphi,\psi\rangle}-\frac1{T|\Omega|}\right|\notag\\
\leq{}&{C_5}d_1^{-1}.\label{4.37}\vspace{-2mm}
\end{align}

Note that $Z_{d_2}\to\Psi_{m_2}$ in $C(\overline Q_T)$ as $d_2\to 0^+$ owing to  \eqref{4.32}. After reducing $\delta_0$ if necessary, the family $\{Z_{d_2}:0<d_2\le\delta_0\}$ is uniformly bounded in $L^\infty(Q_T)$. Therefore, by \eqref{4.34} and \eqref{4.37},
\[\partial_{d_2}\mu(d_1,m_1-\theta_{d_2,m_2})=
\frac1{T|\Omega|}\int_{Q_T}Z_{d_2}\,\mathrm dx\mathrm dt+O\left(d_1^{-1}\right)\quad \text{as } d_1\to\infty,\]
uniformly for $0<d_2\le\delta_0$. Define
 \begin{align*}\eta_1(d_2)=\frac1{T|\Omega|}\int_{Q_T}Z_{d_2}\,\mathrm dx\mathrm dt
-\mathcal I_2\quad\text{for}\;\; d_2>0.\end{align*}
Then \eqref{4.19} holds by enlarging $C_*$. Moreover, $\eta_1(d_2)\to 0$ as $d_2\to 0^+$ by \eqref{4.32}, and $\eta_1$ is continuous on $[0,\delta_0]$ by setting $\eta_1(0)=0$. Therefore, \eqref{4.23} is valid.

It remains to estimate the partial derivative of $\mu(d_1,m_1-\theta_{d_2,m_2})$ with respect to $d_1$. Differentiating the first equation of \eqref{4.26} with respect to $d_1$ and pairing it with \eqref{4.33} gives
  \begin{equation*}
 \partial_{d_1}\mu(d_1,m_1-\theta_{d_2,m_2})=\frac{\int_{Q_T}
\nabla\varphi\cdot\nabla\psi\,\mathrm dx\mathrm dt}{\int_{Q_T}\varphi\psi
\,\mathrm dx\mathrm dt}.
\end{equation*}
The uniform eigenfunction expansions \eqref{4.27} and \eqref{4.35} yield that
  \[\nabla\varphi={\overline\varphi\,\nabla\Gamma_{a_{d_2}}}d_1^{-1}
  +{\nabla\varphi_2}d_1^{-2},\quad
\nabla\psi={\overline{\psi}\,\nabla\Gamma_{a_{d_2}}}d_1^{-1}
+{\nabla\psi_2}d_1^{-2},\;\;\forall\; d_1>D_0,\]
uniformly for $0\le d_2\le\delta_0$. Using also the estimate for $B(t)$ in \eqref{4.36}, we obtain
  \[\partial_{d_1}\mu(d_1,m_1-\theta_{d_2,m_2})
=\frac1{T|\Omega|d_1^2}\int_{Q_T}|\nabla\Gamma_{a_{d_2}}|^2\,\mathrm dx\mathrm dt+O\left(d_1^{-3}\right),\;\;\forall\; d_1>D_0,\]
uniformly for $0\le d_2\le\delta_0$. Hence, \eqref{4.20} holds by enlarging $C_*$ if necessary.

If $d_2=O(d_1^{-1})$, then $d_2\to0^+$ as $d_1\to\infty$, and
  \[\frac1{T|\Omega|}\int_{Q_T}|\nabla\Gamma_{a_{d_2}}|^2\,\mathrm dx\mathrm dt
=\mathcal E_{12}+o(1).\]
Therefore,
\[\partial_{d_1}\mu(d_1,m_1-\theta_{d_2,m_2})={\mathcal E_{12}}d_1^{-2}+o\left(d_1^{-2}\right)\quad \text{whenever } d_2=O(d_1^{-1}),\]
which proves \eqref{4.24}. 
\end{proof}\zzz

\begin{lemma}[The small-$d_2$ critical curve]\label{lm4.6}
Assume that conditions {\bf(M1)}--{\bf(M2)} hold, and let $\delta_0>0$ be the constant given in Lemma \ref{lm4.5}. Then there exist constants $\delta_2^*\in(0,\delta_0]$ and $d_1^0>0$, and a strictly decreasing $C^1$ function $\mathsf d_2:(d_1^0,\infty)\to(0,\delta_2^*)$ such that, for every $d_1>d_1^0$,
  \begin{equation}\begin{cases}\label{4.38}
 \mu\bigl(d_1,m_1-\theta_{\mathsf d_2(d_1),m_2}\bigr)=0,\\[1mm]
 \mu(d_1,m_1-\theta_{d_2,m_2})<0
\quad\text{for }0<d_2<\mathsf d_2(d_1),\\[1mm]
\mu(d_1,m_1-\theta_{d_2,m_2})>0
\quad\text{for }\mathsf d_2(d_1)<d_2\le\delta_2^*.
 \end{cases}\end{equation}
Moreover, as $d_1\to\infty$, 
  \begin{align}\label{4.39}
 d_1\mathsf d_2(d_1)\to {\mathcal E_{12}}/{\mathcal I_2},\;\;\;
 \mathsf d_2'(d_1)=-({\mathcal E_{12}}/{\mathcal I_2})d_1^{-2}+o\left(d_1^{-2}\right).
\end{align}
In particular, $\mathsf d_2(d_1)=O(d_1^{-1})$ as $d_1\to\infty$.
\end{lemma}\zzz

\begin{proof}\; For $d_1,d_2>0$, we set $\mathcal F(d_1,d_2):=\mu(d_1,m_1-\theta_{d_2,m_2})$. The simplicity of the principal eigenvalue and the $C^1$-dependence of $\theta_{d_2,m_2}$ on $d_2>0$ imply that $\mathcal F\in C^1((0,\infty)\times(0,\infty))$.

We divide the proof into four steps.

{\it Step 1}. We first investigate the sign behavior of $F(d_1,d_2)$ for $d_2$ of order $d_1^{-1}$. Choose constants $c_-$ and $c_+$ such that
\begin{equation}\label{4.41}
0<c_-<{\mathcal E_{12}}/{\mathcal I_2}<c_+.
\end{equation}
We now determine the signs of $\mathcal F\left(d_1,{c_-}d_1^{-1}\right)$ and $\mathcal F\left(d_1,{c_+}d_1^{-1}\right)$ for sufficiently large $d_1>0$.

Let $c>0$ be fixed and set $d_2=cd_1^{-1}$. By the precise expansion \eqref{4.18} in Lemma \ref{lm4.5}, we see that for sufficiently large $d_1>0$,
 \[d_1\mathcal F\left(d_1,cd_1^{-1}\right)=\mathcal I_2c-\mathcal E_{12}+c\rho_1\left(cd_1^{-1}\right)+\rho_2\left(cd_1^{-1}\right)
 +d_1R\left(d_1,cd_1^{-1}\right).\]
From Lemma \ref{lm4.5}, $\rho_1(cd_1^{-1})\to 0$, $\rho_2(cd_1^{-1})\to 0$ and $|d_1R(d_1,cd_1^{-1})|\leq cd_1^{-1}\to 0$ as $d_1\to\infty$. Therefore, we obtain
  \begin{equation*}
  d_1\mathcal F\left(d_1,cd_1^{-1}\right)=\mathcal I_2c-\mathcal E_{12}+o(1)
\quad\text{as }d_1\to\infty.\end{equation*}
This combined with \eqref{4.41} yields that there exists $\mathfrak D_1>D_0$ such that
  \begin{equation}\label{4.42}
\mathcal F\left(d_1,{c_-}d_1^{-1}\right)<0,\quad
\mathcal F\left(d_1,{c_+}d_1^{-1}\right)>0,\;\;\forall\; d_1>\mathfrak D_1.
\end{equation}

{\it Step 2}. We next establish the existence and uniqueness of the critical curve. To this end, we first verify the strict monotonicity of $\mathcal F(d_1,\cdot)$ on a sufficiently small interval that is independent of $d_1$. Combining \eqref{4.19} with the limit $\lim_{d_2\to 0^+}\eta_1(d_2)=0$, we obtain constants $\mathfrak D_2>\mathfrak D_1$ and $0<\delta_2^*<\delta_0$ such that
  \[\left|\partial_{d_2}\mathcal F(d_1,d_2)-\mathcal I_2-\eta_1(d_2)\right|
\le{C_*}d_1^{-1}\le {\mathcal I_2}/{4},\;\;\forall\;d_1>\mathfrak D_2,\;0<d_2\le\delta_2^*,\]
and
  \[|\eta_1(d_2)|\leq{\mathcal I_2}/{4},\;\;\forall\;0\le d_2\le\delta_2^*.\]
It follows that
  \begin{equation}\label{4.43}
\partial_{d_2}\mathcal F(d_1,d_2) \geq{\mathcal I_2}/{2}>0,\;\;\forall\;d_1>\mathfrak D_2,\;0<d_2\le\delta_0.
\end{equation}

Choose $d_1^0>\mathfrak D_2$ such that ${c_+}/d_1^0<\delta_2^*$.
This, together with \eqref{4.43} and \eqref{4.42}, implies that for every $d_1>d_1^0$, there exists a unique $\mathsf d_2(d_1)\in\left({c_-}d_1^{-1},{c_+}d_1^{-1}\right)\subset(0,\delta_2^*)$ such that $\mathcal F(d_1,\mathsf d_2(d_1))=0$. Consequently, \eqref{4.38} holds for $d_1>d_1^0$.
Moreover, by \eqref{4.43}, $\partial_{d_2}\mathcal F\bigl(d_1,\mathsf d_2(d_1)\bigr) \geq{\mathcal I_2}/{2}>0$ for every $d_1>d_1^0$. As $\mathcal F\in
C^1\bigl((d_1^0,\infty)\times(0,\delta_2^*)\bigr)$, the implicit function theorem applies. For each fixed $d_1^*>d_1^0$, it yields an open interval  $I_{d_1^*}\subset(d_1^0,\infty)$ containing $d_1^*$, an open interval $J_{d_1^*}\subset(0,\delta_2^*)$ containing $\mathsf d_2(d_1^*)$, and a unique $C^1$ function $g_{d_1^*}:I_{d_1^*}\to J_{d_1^*}$ such that $\mathcal F\bigl(d_1,g_{d_1^*}(d_1)\bigr)=0$ for all $d_1\in I_{d_1^*}$, and so $g_{d_1^*}(d_1^*)=\mathsf d_2(d_1^*)$ by the uniqueness.

In particular, any two such local $C^1$ representations agree on the overlap of their domains. Therefore they patch together to yield $\mathsf d_2\in C^1((d_1^0,\infty))$.

{\it Step 3}. We next characterize the asymptotic behavior of the critical curve. The location of the zero gives
 \begin{equation}\label{4.44}
  {c_-}d_1^{-1}<\mathsf d_2(d_1)<{c_+}d_1^{-1},\;\;\forall\; d_1>d_1^0.
\end{equation}
In particular, $\mathsf d_2(d_1)=O(d_1^{-1})$ and $\mathsf d_2(d_1)\to 0$ as $d_1\to\infty$.

Evaluating \eqref{4.18} at $d_2=\mathsf d_2(d_1)$ and using $\mathcal F(d_1,\mathsf d_2(d_1))=0$, we obtain
  \[ 0=\mathcal I_2\mathsf d_2(d_1)-{\mathcal E_{12}}d_1^{-1}+\mathsf d_2(d_1)\rho_1(\mathsf d_2(d_1))+{\rho_2(\mathsf d_2(d_1))}d_1^{-1}+R(d_1,\mathsf d_2(d_1)). \]
Multiplying this equality by $d_1$ gives
\[ 0=\mathcal I_2d_1\mathsf d_2(d_1)-\mathcal E_{12}+d_1\mathsf d_2(d_1)\rho_1(\mathsf d_2(d_1))+\rho_2(\mathsf d_2(d_1))+d_1R(d_1,\mathsf d_2(d_1)). \]
In view of the estimates \eqref{4.44}, $d_1\mathsf d_2(d_1)$ is bounded. Moreover, $\rho_1(\mathsf d_2(d_1))\to 0,\rho_2(\mathsf d_2(d_1))\to 0$ and $\left|
d_1R(d_1,\mathsf d_2(d_1))\right|
\leq {C_*}d_1^{-1}\to 0$ as $d_1\to\infty$. Therefore, $\mathcal I_2d_1\mathsf d_2(d_1)
-\mathcal E_{12}\to 0$ as $d_1\to\infty$, which gives the first relation in \eqref{4.39}.

{\it Step 4}. We finally establish the strict monotonicity of the critical curve. 
Differentiating $\mathcal F(d_1,\mathsf d_2(d_1))=0$ with respect to $d_1$, we obtain
\begin{equation}\label{4.45}
\mathsf d_2'(d_1)=-\frac{\partial_{d_1}\mathcal F(d_1,\mathsf d_2(d_1))}{\partial_{d_2}\mathcal F(d_1,\mathsf d_2(d_1))}.
\end{equation}
The estimates \eqref{4.44} yield $\mathsf d_2(d_1)=O(d_1^{-1})$ for $d_1>d_1^0$. Consequently, the derivative estimate \eqref{4.24} from Lemma \ref{lm4.5} yields
  \begin{equation}\label{4.46}
\partial_{d_1}\mathcal F(d_1,\mathsf d_2(d_1))
={\mathcal E_{12}}d_1^{-2}+o\left(d_1^{-2}\right),\;\;\forall\; d_1>d_1^0,
  \end{equation}
while \eqref{4.23} yields that, enlarging $d_1^0$ if necessary,
  \begin{equation}\label{4.47}
\partial_{d_2}\mathcal F(d_1,\mathsf d_2(d_1))=\mathcal I_2+o(1),\;\;\forall\; d_1>d_1^0.
  \end{equation}
Substituting \eqref{4.46} and \eqref{4.47} into \eqref{4.45} gives the second relation in \eqref{4.39}. Furthermore, in view of $\mathcal E_{12}>0$ and $\mathcal I_2>0$, we conclude $\mathsf d_2'(d_1)<0$ for all sufficiently large $d_1$.\vspace{-1mm}
\end{proof}

The next lemma constructs two strictly increasing $C^1$ functions $\widehat{\mathsf{d}}_2:(D_1,\infty)\to(d_0,\infty)$ and $\widehat{\mathsf{d}}_1:(D_2,\infty)\to (d_0,\infty)$, and establishes their sharp asymptotic behavior.\vspace{-1mm}

\begin{lemma}\label{lm4.7} There exist a constant $d_0>0$, which depends only on  $\|\nabla\Gamma_{m_i}\|_{L^2(Q_T)}$ and $\|m_i\|_{L^\yy(Q_T)}$ for $i=1,2$, sufficiently large constants $D_1, D_2>0$, and strictly increasing $C^1$ functions $\widehat{\mathsf{d}}_2:(D_1,\infty)\to(d_0,\infty)$ and $\widehat{\mathsf{d}}_1:(D_2,\infty)\to (d_0,\infty)$ such that
   \begin{align}
\text{for}\;\;(d_1, d_2)\in(D_1,\infty)\times(d_0,\infty),\;\;\;\mu(d_2,m_2-\theta_{d_1,m_1})=0 \Longleftrightarrow d_2=\widehat{\mathsf{d}}_2(d_1),\label{4.48}\\[1mm]
 \text{for}\;\;(d_1, d_2)\in(d_0,\infty)\times(D_2,\infty),\;\;\; \mu(d_1,m_1-\theta_{d_2,m_2})=0 \Longleftrightarrow d_1=\widehat{\mathsf{d}}_1(d_2).\label{4.49} \end{align}
Furthermore,
  \begin{align}\frac{\widehat{\mathsf{d}}_2(d_1)}{d_1}=A+o(1)\;\;\;\text{for}\;\; d_1>D_1,\;\;\;\text{and}\;\;\;
 \frac{\widehat{\mathsf{d}}_1(d_2)}{d_2}=\frac 1A+o(1)\;\;\; \text{for}\;\;  d_2>D_2,\label{4.50}\end{align}
where
 \bee
  A={\mathcal{E}_2}/{\mathcal{E}_1}.\label{4.51}
  \eee
 \end{lemma}\vspace{-1mm}

\begin{proof}\; We only prove the existence and properties of $\widehat{\mathsf{d}}_2$ as  those of $\widehat{\mathsf{d}}_1$ can be proved similarly.

{\it Step 1}. We will use the implicit function theorem to determine $\widehat{\mathsf{d}}_2$. Thanks to the fact that $\Gamma_{m_2-\theta_{d_1,m_1}}=\Gamma_2-\Gamma_{\theta_{d_1,m_1}}$, one has
  \bee\label{4.52}
 \|\nabla\Gamma_{m_2-\theta_{d_1,m_1}}\|_{L^2(Q_T)}\geq
  \|\nabla\Gamma_2\|_{L^2(Q_T)}-\|\nabla\Gamma_{\theta_{d_1,m_1}} \|_{L^2(Q_T)}.
   \eee
Multiplying the equation of $\Gamma_{\theta_{d_1,m_1}}$ by $\Gamma_{\theta_{d_1,m_1}}$ and integrating the result over $Q_T$ firstly, and using the H\"{o}lder and Poincar\'{e} inequalities and the fact that $\overline{\Gamma}_{\theta_{d_1,m_1}}(t)\equiv 0$ secondly, we derive that
  \begin{align}
\int_{Q_T}|\nabla\Gamma_{\theta_{d_1,m_1}}|^2&=\int_{Q_T}
\Gamma_{\theta_{d_1,m_1}}(\theta_{d_1,m_1}-\bar\theta_{d_1,m_1})\notag\\
&\leq \varepsilon\int_{Q_T}\Gamma_{\theta_{d_1,m_1}}^2+C(\varepsilon)
\int_{Q_T}(\theta_{d_1,m_1}-\bar\theta_{d_1,m_1})^2\notag\\
&\leq \frac{1}{2}\int_{Q_T}|\nabla\Gamma_{\theta_{d_1,m_1}}|^2+C(\varepsilon)
\int_{Q_T}(\theta_{d_1,m_1}-\bar\theta_{d_1,m_1})^2.\label{4.53}
  \end{align}
Multiplying the equation of $\theta_{d_1,m_1}$ by $\theta_{d_1,m_1}$ and integrating the result over $Q_T$, we obtain
  \[  d_1\int_{Q_T}|\nabla\theta_{d_1,m_1}|^2=
 \int_{Q_T}\theta_{d_1,m_1}(m_1-\theta_{d_1,m_1}).\]
This together with the Poincar\'{e} inequality and Lemma \ref{lm2.7}(1) yields
  \[ \int_{Q_T}(\theta_{d_1,m_1}-\bar\theta_{d_1,m_1})^2\leq C\int_{Q_T}|\nabla\theta_{d_1,m_1}|^2\leq C d_1^{-1}.\]
It follows from \eqref{4.53} that $\|\nabla\Gamma_{\theta_{d_1,m_1}}\|_{L^2(Q_T)}\leq Cd_1^{-1/2}$. Note that $\|\nabla\Gamma_2\|_{L^2(Q_T)}>0$. Using \eqref{4.52}, we can find a constant $D_1>0$, such that
   \[\|\nabla\Gamma_{m_2-\theta_{d_1,m_1}}\|_{L^2(Q_T)}\geq \frac{1}{2}\|\nabla\Gamma_2\|_{L^2(Q_T)}>0,\;\;\forall\, d_1\ge D_1.\]
Now, in view of Proposition \ref{p2.2}, there exists $d_0>0$ depending only on $\|\nabla\Gamma_2\|_{L^2(Q_T)}$ and $\|m_1\|_{L^\yy(Q_T)}+\|m_2\|_{L^\yy(Q_T)}$ such that
  \bee\label{4.54}
  \frac{\partial \mu(d_2,m_2-\theta_{d_1,m_1})}{\partial d_2}>0,\;\; \forall\, d_1\ge D_1,\, d_2\ge d_0. \eee

Make using of Lemma \ref{lm2.10} and Proposition \ref{lm2.1}(1), we have
  \[\lim_{d_1\to\infty}\mu(d_2,m_2-\theta_{d_1,m_1})
  =\mu(d_2,m_2- l_1)<-\frac{1}{T|\Omega|}\int_{Q_T}(m_2- l_1)=0.\]
By the continuity and enlarge $D_1$ if necessary, we deduce that $\mu(d_2,m_2-\theta_{d_1,m_1})<0$ for all $(d_1,d_2)\in(D_1,\infty)\times[d_0/2,d_0]$. Notice that
  \[ \lim_{d_2\to\infty}\mu(d_2,m_2-\theta_{d_1,m_1})=-\frac{1}{T|\Omega|}
 \int_{Q_T}(m_2-\theta_{d_1,m_1})>0.\]
For each $d_1>D_1$, $\mu(d_2,m_2-\theta_{d_1,m_1})$ changes sign at least once when $d_2$ increases from $d_0$ to $\infty$.

The application of the implicit function theorem yields a unique strictly increasing $C^1$ function $\widehat{\mathsf{d}}_2:(D_1,\infty)\to(d_2^*,\infty)$ such that \eqref{4.48} holds.

{\it Step 2}. We now establish the first estimate of \eqref{4.50}. To this end, we first prove
  \[\lim_{d_1\to\yy}\widehat{\mathsf{d}}_2(d_1)=\infty. \]
Suppose for the contradiction that there exist $d_1^k$ with $d_1^k\to\infty$ such that  $\widehat{\mathsf{d}}_2(d_1^k)\to d_2\in(d_0,\infty)$ as $k\to\infty$. Then by Proposition \ref{lm2.1}(1) and Lemma \ref{lm2.10},
   \[0=\lim_{k\to\infty}\mu(\widehat{\mathsf{d}}_2(d_1^k),m_2-\theta_{d_1^k,m_1})
   =\mu(d_2,m_2- l_1)<-\frac{1}{T|\Omega|}\int_{Q_T}(m_2- l_1)=0.\]
This ia a contradiction.

Let $\varphi$ and $\psi$ be the positive normalized eigenfunction and adjoint eigenfunction, respectively, associated with the principal eigenvalue $\mu(\widehat{\mathsf{d}}_2(d_1), m_2-\theta_{d_1,m_1})=0$. Then we have
  \[\begin{cases}
 \varphi_t-\widehat{\mathsf{d}}_2(d_1)\Delta \varphi=(m_2-\theta_{d_1,m_1})\varphi &\text{in}\;\;Q_T,\\
\partial_{\nu}\varphi=0 &\text{on}\;\;S_T,\\
\varphi(x,0)=\varphi(x,T) &\text{on}\;\;\boo,
  \end{cases}\]
and
  \[\begin{cases}
-\psi_t-\widehat{\mathsf{d}}_2(d_1)\Delta\psi=(m_2-\theta_{d_1,m_1})\psi &\text{in}\;\;Q_T,\\
\partial_{\nu}\psi=0 &\text{on}\;\;S_T,\\
\psi(x,0)=\psi(x,T) &\text{on}\;\;\boo.
  \end{cases}\]
By virtue of Lemma \ref{lm2.10}, $\lim_{d_1\to\infty}\theta_{d_1,m_1}= l_1$ in $C(\ol Q_T)$. Owing to the limit $\lim_{d_1\to\infty}\widehat{\mathsf{d}}_2(d_1)=\infty$, an argument analogous to the proof of \cite[Lemma 3.7]{HMP01} shows that  $\varphi\to\varphi_*(t)$ and $\psi\to\psi_*(t)$ in $C(\ol Q_T)$ as $d_1 \to\infty$. Here $\varphi_*(t)$ and $\psi_*(t)$ are, respectively, the unique positive solutions to
  \[ \begin{cases}
\varphi_*'(t)=(\ol m_2- l_1)\varphi_*(t), \;\; t\in[0,T],\\[0.5mm]
\varphi_*(0)=\varphi_*(T), \;\;\|\varphi_*\|_{L^{\infty}((0,T))}=1.
  \end{cases}\]
and
  \[ \begin{cases}
\psi_*'(t)=-(\ol m_2- l_1)\psi_*(t), \;\; t\in[0,T],\\[0.5mm]
\psi_*(0)=\psi_*(T), \;\;\|\psi_*\|_{L^{\infty}((0,T))}=1.
  \end{cases}\]
It is easy to verify that $\varphi_*\psi_*$ is a positive constant. Recall that  $\Gamma_m$, $\mathcal{E}(m)$, $l_m$ and $h_m$ are defined by \eqref{1.10}, \eqref{1.12}, \eqref{1.11} and \eqref{2.21}, respectively. By virtue of
  \bee\int_0^T\kk(\frac{1}{A|\Omega|}\int_{\Omega}|
  \nabla\Gamma_2|^2-l_1h_1\rr)=\int_{Q_T}\big(|\nabla\Gamma_1|^2
 -l_1h_1\big)=0,\label{4.55}\eee
the periodic problem
  \bee\label{4.56}\begin{cases}
 c'(t)=\dd\frac 1{A|\Omega|}\int_{\Omega}|\nabla\Gamma_2|^2
-l_1h_1, \;\; t\in[0,T],\\[3mm]
c(0)=c(T), \;\;\;\widehat{\varphi_*c}=0
  \end{cases}\eee
possesses a unique solution $c\in C^1([0,T])$. Let
   \[\varphi=\varphi_*+\frac{\varphi_*(\Gamma_2+c)}
   {\widehat{\mathsf{d}}_2(d_1)}+\frac{\sigma}{d_1^2}.\]
A direct calculation gives the following problem satisfied by $\sigma$:
   \[\begin{cases}
\sigma_t-\widehat{\mathsf{d}}_2(d_1)\Delta \sigma-(m_2-\theta_{d_1,m_1})\sigma=R & \text{in}\;\;Q_T,\\[0.1mm]
\partial_{\nu}\sigma=0 & \text{on}\;\;S_T,\\[0.1mm]
\sigma(x,0)=\sigma(x,T), & \text{on}\;\;\overline\Omega,
   \end{cases}\]
where
  \begin{align}
  R=\;&\frac{d_1^2}{\widehat{\mathsf{d}}_2(d_1)}
 \bigg[\widehat{\mathsf{d}}_2(d_1)\varphi_*( l_1-\theta_{d_1,m_1})
+\varphi_*(m_2-\ol m_2+ l_1-\theta_{d_1,m_1})(\Gamma_2+c)\notag\\
&-\varphi_*\kk(\partial_t\Gamma_2+\frac 1{A|\Omega|}
\int_{\Omega}|\nabla\Gamma_2|^2- l_1h_1\rr)\bigg]. \label{4.57}
   \end{align}
Multiplying the equation of $\sigma$ by $\psi$ and the equation of $\psi$ by $\sigma$, and integrating the results over $Q_T$, we arrive at 
 \bee\int_{Q_T}R\psi=0,\;\;\;\text{i.e.},\;\;\int_{Q_T}
  \frac{\widehat{\mathsf{d}}_2(d_1)}{d_1^2}R\psi=0.\label{4.58}\eee
The relation \qq{2.24} yields
 \[ \theta_{d_1,m_1}- l_1=l_1(\Gamma_1+h_1)d_1^{-1}
 +o\!\kk(d_1^{-1}\rr)\;\;\;\text{as}\;\;d_1\gg 1. \]
Inserting this into \qq{4.57} and using \qq{4.58}, \qq{4.55}, together with the $T$-periodicity of $\Gamma_2$, the facts that $\varphi_*\psi_*$ is a positive constant, that $\psi\to\psi_*$ in $C(\ol Q_T)$ as $d_1\to\infty$, and that
 \[\int_{Q_T}\kk(\partial_t\Gamma_2+\frac 1{A|\Omega|}
\int_{\Omega}|\nabla\Gamma_2|^2- l_1h_1\rr)=0, \]
we arrive at the limit
  \begin{align}
 \lim_{d_1\to\infty}\frac{\widehat{\mathsf{d}}_2(d_1)}{d_1}=\;& \frac{\dd\int_{Q_T}\varphi_*\psi_*(m_2-\ol m_2)(\Gamma_2+c)}
 {\dd\int_{Q_T}\varphi_*\psi_* l_1(\Gamma_1+h_1)}-\frac{\dd\int_{Q_T}\varphi_*\psi_*\kk( \partial_t\Gamma_2+\frac 1{A|\Omega|}\int_{\Omega}|\nabla\Gamma_2|^2- l_1h_1\rr)}
{\dd\int_{Q_T}\varphi_*\psi_* l_1(\Gamma_1+h_1)}\nm\\[1mm]
=\;&\frac{\dd\int_{Q_T}(m_2-\ol m_2)(\Gamma_2+c)}
 {\dd\int_{Q_T} l_1(\Gamma_1+h_1)}.\label{4.59}
  \end{align}
Notice that
   \begin{align*}
\int_{Q_T}(m_2-\ol m_2)\Gamma_2
=\;&\int_{Q_T}|\nabla\Gamma_2|^2=\mathcal{E}_2,\;\;
 \int_{Q_T}(m_2-\ol m_2)c=0,\\[1mm]
\int_{Q_T} l_1(\Gamma_1+h_1)
=\;&\int_0^T l_1\bigg(\int_{\Omega}\Gamma_1 \bigg)+|\Omega|\int_0^T l_1h_1=\int_{Q_T}|\nabla\Gamma_1|^2=\mathcal{E}_1\;\;\;\text{by}\;\eqref{2.21}.
  \end{align*}
The first estimate of \eqref{4.50} follows from \qq{4.59}, and the proof is complete.
\end{proof}

To identify the positions of $\widehat{\mathsf{d}}_1(d_2)$ and $\widehat{\mathsf{d}}_2(d_1)$, we further explore their asymptotic profiles. Observe that 
  \[\int_{\Omega}\kk((m_2-\ol m_2)(\Gamma_2+c)
  -A l_1\Gamma_1-\partial_t\Gamma_2-\frac{1}{|\Omega|}
 \int_{\Omega}|\nabla\Gamma_2|^2\rr)\equiv 0.\]
Thus the problem
   \bee\label{4.60}\begin{cases}
-A\Delta \beta=(m_2-\ol m_2)(\Gamma_2+c)-A l_1\Gamma_1
-\partial_t\Gamma_2-\dd\frac{1}{|\Omega|}\int_{\Omega}|\nabla\Gamma_2|^2, \;\;x\in\Omega,\\
\partial_{\nu}\beta|_{\partial\Omega}=0,\;\;\;\bar\beta(t)=0
  \end{cases}\eee
admits a unique solution $\beta$ for each $t\in[0,T]$. Define
   \begin{align}
   B_1=\;&\frac{1}{\mathcal{E}_2}\int_{Q_T}\!\!\big[(m_2\!-\!\ol m_2)\Gamma_2^2\!-\!2Al_1\Gamma_1\Gamma_2-A^2(m_1\!-\!\ol m_1)\Gamma_1^2+2A^2 l_1\Gamma_1^2\big],\label{4.61}\\[.1mm]
  \Psi(t)=\;&\frac{1}{|\Omega|}\int_{\Omega}\left[(m_2\!-\!l_1)
\left(\beta\!-\!\frac{B_1}{A}\Gamma_2\right)\!-\!l_1\big[
\Gamma_1\Gamma_2\!-\!l_1h_1c\!-\!A l_1k_1\big]\right].\nm
   \end{align}

We will show that $\widehat\Psi=0$, and thereby the periodic problem
 \bee\label{4.62}\begin{cases}
c_t=\Psi(t), \;\; t\in[0,T],\\[.1mm]
c(0)=c(T), \;\;\widehat{\varphi_*c}=0
 \end{cases}\eee
possesses a unique solution $\tilde c$. Indeed, let $\Lambda_1$ and $k_1$ be the unique solutions to \eqref{2.22} and \eqref{2.23}, respectively, with $m=m_1$. Through the careful calculation we have that
  \begin{align*}
\int_{\Omega}(m_2- l_1)\beta
&=\int_{\Omega}(m_2-\overline m_2)\beta=-\int_{\Omega}\beta\Delta\Gamma_2=-\int_{\Omega}\Gamma_2\Delta\beta  \\
&=\int_{\Omega}\left(\frac{1}{A}(m_2-\ol m_2)\Gamma_2^2+\frac{c}{A}|\nabla\Gamma_2|^2- l_1\Gamma_1\Gamma_2-\frac{1}{2A}(\Gamma_2^2)_t\right),
  \end{align*}
and
  \begin{align*}
\int_{Q_T} l_1k_1&=\int_{Q_T}\kk(\Lambda_1(m_1-\ol m_1)-l_1\Gamma_1^2-l_1h_1^2 \rr)\\
&=\int_{Q_T}\kk(-\Lambda_1\Delta\Gamma_1-l_1\Gamma_1^2-l_1h_1^2 \rr)\\
&=\int_{Q_T}\Big(-\Gamma_1\Delta\Lambda_1-l_1\Gamma_1^2-l_1h_1^2\Big)\\
&=\int_{Q_T}\kk[\Gamma_1(\Gamma_1+h_1)(m_1-\ol m_1)-(\Gamma_1^2)_t-2l_1\Gamma_1^2-\frac{\Gamma_1}{|\Omega|}
\int_{\Omega}|\nabla\Gamma_1|^2-l_1h_1^2\rr]\\
&=\int_{Q_T}\kk[\Gamma_1(\Gamma_1+h_1)(m_1-\ol m_1)
 -2 l_1\Gamma_1^2- l_1h_1^2 \rr].
 \end{align*}
We multiply the equation for $c$ in \eqref{4.56} by $c$ and the equation for $h_1$ in \eqref{2.21} by $h_1$, then integrate the resulting equations over $(0,T)$ to obtain
  \begin{align*}
|\oo|\!\int_0^T\!\!\Psi(t)\dt=\;&\int_{Q_T}\kk[(m_2- l_1)\kk(\beta-\frac{B_1}{A}\Gamma_2\rr)
- l_1\Gamma_1\Gamma_2- l_1h_1c-A l_1k_1\rr]\\
=\;&\int_{Q_T}\bigg[\frac{(m_2-\ol m_2)\Gamma_2^2}{A}
+\frac{c}{A}|\nabla\Gamma_2|^2
-2 l_1\Gamma_1\Gamma_2-\frac{B_1}{A}|\nabla\Gamma_2|^2\\
&- l_1h_1c-A\kk(\Gamma_1(\Gamma_1
+h_1)(m_1-\ol m_1)-2 l_1\Gamma_1^2- l_1h_1^2\rr)\bigg]\\
=\;&\int_{Q_T}\bigg(\frac{(m_2-\ol m_2)\Gamma_2^2}{A}
-2 l_1\Gamma_1\Gamma_2-A(m_1-\ol m_1)\Gamma_1^2\\
\;&+2A l_1\Gamma_1^2+\frac{c}{A}|
\nabla\Gamma_2|^2- l_1h_1c-A\big[h_1\Gamma_1(m_1-\ol m_1)- l_1h_1^2\big]\bigg)-B_1\mathcal{E}_1\\
=\;&\int_{Q_T}\!\kk[\frac {(c^2-Ah_1^2)_t}2+\frac {(m_2-\ol m_2)\Gamma_2^2}A-2 l_1\Gamma_1\Gamma_2-A(m_1-\ol m_1)\Gamma_1^2+2A l_1\Gamma_1^2\rr]-B_1\mathcal{E}_1\\
=\;&0.
\end{align*}
Hence \qq{4.62} admits a unique solution $\tilde c$.

Define
  \[\tilde \varphi:=\frac{\|\varphi_*\|_{L^1(Q_T)}}{\|\varphi\|_{L^1(Q_T)}}\varphi.\]
Then $\|\tilde \varphi\|_{L^1(Q_T)}=\|\varphi_*\|_{L^1(Q_T)}$. Using the fact  $\varphi\to \varphi_*$ in $C(\ol Q_T)$ as $d_1\to\infty$, we have $\tilde\varphi\to \varphi_*$ in $C(\ol Q_T)$ as $d_1\to\infty$.
The next lemma provides the expansions of the functions  $\widehat{\mathsf{d}}_2(d_1)$ and $\tilde\varphi(x,t)$.

\begin{lemma}\label{lm4.8} Assume that the condition {\bf(M3)} holds. Define $A$ by \qq{4.51}, and $B_1$ by \qq{4.61}. Let $c$ and $\tilde c$ be the unique solutions of \eqref{4.56} and \eqref{4.62}, respectively, and $\beta$ be the unique solution of \qq{4.60}. Then as $d_1\to\infty$,
 \begin{align}
 \widehat{\mathsf{d}}_2(d_1)=\;&Ad_1+B_1+o(1),\label{4.63}\\[1mm]
 \tilde\varphi=\;&\varphi_*+\frac{\varphi_*(\Gamma_2+c)}{A}d_1^{-1}
-\frac{B_1\varphi_*\Gamma_2}{A^2}d_1^{-2}+\frac{\varphi_*(\beta+\tilde c)}{A}d_1^{-2}+o\!\kk(d_1^{-2}\rr), \label{4.64}
  \end{align}
where the high order term $o(\cdot)$ in \eqref{4.63} is understood in the sense of real number, and the high order term $o(\cdot)$ in \eqref{4.64} is understood in the norm of $C(\ol Q_T)$.
\end{lemma}\vspace{-2mm}

\begin{proof}\; We first analyze the regularities of $\beta$ and $\tilde c$. Note that the condition {\bf(M3)} holds, and $l_1,h_1,c\in C_T^1([0,T])$. Make use of Lemma \ref{lm2.9} it is not hard to show that
 \[\beta\in C_T([0,T],C^{1+\alpha}(\overline\Omega))\cap H^1((0,T),W_p^2(\Omega)),\;\;\tilde c\in C_T([0,T])\cap H^1((0,T)).\]

Let
  \[\widehat{\mathsf{d}}_2(d_1)=Ad_1+f,\quad \tilde\varphi=\varphi_*+\frac{\varphi_*(\Gamma_2+c)}{A}d_1^{-1}
 +\frac{g}{A}d_1^{-2}\]
and
  \[\theta_{d_1,m_1}= l_1+l_1(\Gamma_1+h_1)d_1^{-1}+\omega d_1^{-2}.\]
We will use the implicit function theorem to determine the values of $f$ and $g$.

{\it Step 1}. From Lemma \ref{lm2.10}, we have $\omega\to l_1(\Lambda_1+k_1)$ in $C(\ol Q_T)$ as $d_1\to\infty$. Notice that 
 \[\int_{Q_T}(\tilde\varphi-\varphi_*)=0,\;\;\;\text{and}\;\;\int_{Q_T}\varphi_*(\Gamma_2+c)=0.\]
The direct calculations show that $(f,g)\in \mathbb R\times C_T^1([0,T],W_p^2(\Omega))$ and satisfies, with $s=1/d_1$,
 \bee\label{4.65}\begin{cases}
s g_t-(A+s f)\Delta g-s(m_2- l_1)g=Q(f,g,s) &\;\;\text{in}\;\;Q_T,\\
\partial_{\nu}g=0 &\;\;\text{on}\;\;S_T,\\
g(x,0)=g(x,T) &\;\;\text{on}\;\;\boo,\\
\int_{Q_T}g(x,t)=0 &\;\;\text{on}\;\;[0,T],
  \end{cases} \eee
where
 \begin{align*}
 Q(f,g,s)=\;&f\varphi_*\Delta\Gamma_2+\varphi_*(m_2-\ol m_2)(\Gamma_2+c)
 -\frac{\varphi_*}{|\Omega|}\int_{\Omega}|\nabla\Gamma_2|^2
 -A\varphi_* l_1\Gamma_1-\varphi_*\partial_t\Gamma_2\\
\;&-s\kk[\varphi_* l_1(\Gamma_1+h_1)(\Gamma_2+c)+A\varphi_*\omega
+s l_1(\Gamma_1+h_1)g+s\varphi_*(\Gamma_2+c)\omega+s^2\omega g\rr].
  \end{align*}
Set
  \begin{align*}
 h_1(f,\sigma,r,s)=\;&(A+s f)\Delta \sigma-s(\sigma+r)_t
+s(m_2- l_1)(\sigma+r)+Q(f,\sigma+r,s)\\[1mm]
 h_2(f,\sigma,r,s)=\;&\int_{\Omega}\Big\{(\sigma+r)_t-
 (m_2- l_1)(\sigma+r)+\varphi_*l_1(\Gamma_1\Gamma_2+h_1c)+A\varphi_*\omega\\
\;&+s[ l_1(\Gamma_1+h_1)(\sigma+r)+
 \varphi_*(\Gamma_2+c)\omega]+s^2\omega(\sigma+r)\Big\},\\
 h_3(f,\sigma,r,s)=&\;\int_{Q_T}(\sigma+r),
 \end{align*}
and define an operator $\mathcal H:\mathbb R\times\mathbb X_p\times E\times[0,\infty)\to\mathbb Y_p\times\widetilde E\times\mathbb R$ by
  \[ \mathcal H(f,\sigma,r,s)=\big(h_1(f,\sigma,r,s),\;
    h_2(f,\sigma,r,s),\;h_3(f,\sigma,r,s)\big)^{\text{T}},\]
where $E$, $\widetilde E$, $\mathbb X_p$ and $\mathbb Y_p$ are defined by \qq{2.27}. Note that 
 \[\int_\oo\Delta\Gamma_2=0,\;\;\int_\oo(m_2-\ol m_2)\Gamma_2=
 \frac1{|\Omega|}\int_{\Omega}|\nabla\Gamma_2|^2,\;\;\int_\oo\Gamma_i=0.\]
It is easy to very that a quadruple $(f,\sigma,r,s)\in\mathbb R\times\mathbb X_p\times E\times(0,\infty)$ satisfies $\mathcal H(f,\sigma,r,s)=0$ if and only if  the pair $(f,g)=(f,\sigma+r)$ solves \eqref{4.65}.\vspace{-0.2mm}

{\it Step 2}. Write $\omega= l_1(\Lambda_1+k_1)+o(1)$ and set
  \[ f_0=B_1,\;\;\sigma_0=\varphi_*\beta-\frac{B_1}{A}\varphi_*\Gamma_2,\;\; r_0=\varphi_*\tilde c. \]
Then $(f_0,\sigma_0,r_0)\in \mathbb R\times\mathbb X_p\times E$, and a direct computation gives
   \begin{align*}
  h_1(f_0,\sigma_0,r_0,0)
=\;&A\varphi_*\Delta\kk(\beta-\frac{B_1}{A}\Gamma_2\rr)
+B_1\varphi_*\Delta\Gamma_2+\varphi_*(m_2-\ol m_2)(\Gamma_2+c)\\
&-A\varphi_* l_1\Gamma_1-\varphi_*\partial_t\Gamma_2
-\frac{\varphi_*}{|\Omega|}\int_{\Omega}|\nabla\Gamma_2|^2=0,\\[1mm]
 h_2(f_0,\sigma_0,r_0,0)=&\int_{\Omega}\kk( \varphi_*\tilde c+\varphi_*\beta-\frac{B_1}{A}\varphi_*\Gamma_2\rr)_t\\
&-\int_{\Omega}\varphi_*\kk[ (m_2\!-\! l_1)\kk(\tilde c\!+\!\beta\!-\!\frac{B_1}{A}\Gamma_2\rr)
\!-\! l_1(\Gamma_1\Gamma_2\!+\!h_1c)\!-\!A l_1(\Lambda_1\!+\!k_1)\rr]\\
=\;&|\Omega|\kk[\varphi_*'\tilde c+\varphi_*\tilde c'-(\ol m_2- l_1)\varphi_*\tilde c\rr]-\varphi_*\int_{\Omega}(m_2- l_1)\kk(\beta-\frac{B_1}{A}\Gamma_2\rr)\\
&-\varphi_* l_1\int_{\Omega}\Gamma_1\Gamma_2
-|\Omega|\varphi_* l_1(h_1c-Ak_1)=0,\\[1mm]
h_3(f_0,\sigma_0,r_0,0)=\;&\int_{Q_T}\varphi_*\tilde c +\int_0^T\varphi_*\int_{\Omega}\beta
-\frac{B_1}{A}\int_0^T\varphi_*\int_{\Omega}\Gamma_2=0,
 \end{align*}
i.e., $\mathcal H(f_0,\sigma_0,r_0,0)=0$. A direct calculation determines the Fr\'echet derivative of $\mathcal H$ with respect to $(f,r,\sigma)$ at $(f_0,\sigma_0,r_0,0)$ as follows:
  \[ D_{fr\sigma}\mathcal H(f_0,\sigma_0,r_0,0)(\mu,\phi,b)^{\mathrm{T}}
 =\begin{pmatrix}
 A\Delta \phi+ \mu\varphi_*\Delta\Gamma_2\\[1mm]
\int_{\Omega}(\phi+b)_t-\int_{\Omega}(m_2- l_1)(\phi+b)\\[1mm]
 \int_{Q_T}(\phi+b)
\end{pmatrix}. \]

{\it Step 3}. We show that problem
\bee\label{4.66}
D_{fr\sigma}\mathcal H(f_0,\sigma_0,r_0,0)(\mu,\phi,b)^{\mathrm{T}}=0,
\;\;(\mu,\phi,b)\in\mathbb R\times\mathbb X_p\times E
 \eee
admits only the trivial solution $0$. Suppose that \eqref{4.66} admits a nontrivial solution $(\mu,\phi,b)$. Then for each $t\in[0,T]$, we have
 \bee \begin{cases}
-A\Delta \phi=\mu\varphi_*\Delta\Gamma_2=\mu\varphi_*(\ol m_2-m_2), & x\in\Omega,\\
\partial_{\nu}\phi=0, & x\in\partial\Omega,\\
\int_{\Omega}\phi=0.
  \end{cases}\label{4.67}\eee
  
If $\mu\neq 0$, then $\nabla\phi(\cdot,t)\not\equiv 0$ for every $t\in[0,T]$. This follows from $\nabla m_2\not\equiv 0$, $A>0$, and $\varphi_*>0$ on $[0,T]$. Using \qq{4.67}, we obtain
  \begin{align}
 \int_{\oo}(m_2-l_1)\phi=\;&\int_{\oo}(m_2-\ol m_2)\phi+(\ol m_2-l_1)\int_\oo\phi\nm\\
 =\;&\int_{\oo}(m_2-\ol m_2)\phi
  =-\frac{A}{\mu\varphi_*}\int_{\Omega}|\nabla\phi|^2\neq 0,\;\;\forall\, t\in[0,T].
  \label{4.68}\end{align}
Notice that $\int_{\Omega}\phi_t=(\int_{\Omega}\phi)_t=0$ and $b$ is spatially homogeneous. Thanks to the relation \eqref{4.68}, the second equation of \qq{4.66} reduces to
  \[  0=\int_{\Omega}(\phi+b)_t-\int_{\Omega}(m_2-l_1)(\phi+b)
  =|\Omega|\kk(b_t-(\ol m_2-l_1)b+
  \frac{A}{\mu\varphi_*|\Omega|}\int_{\Omega}|\nabla\phi|^2\rr).\]
It follows that
 \[ b(T){\rm e}^{-\int_0^T(\ol m_2-l_1)}=b(0)-\frac{A}{\mu|\Omega|}\int_0^T{\rm e}^{-\int_0^t(\ol m_2-l_1)}\frac{1}{\varphi_*}\int_{\Omega}|\nabla\phi|^2. \]
Recall that
  \[\int_0^T(\ol m_2-l_1)=0,\;\;\text{and}\;\;\frac{1}{\varphi_*}\int_{\Omega}|\nabla\phi|^2>0\]
for all $t\in[0,T]$. The above equation implies $b(T)\neq b(0)$. We have a contradiction, and so $\mu=0$. 

From  \qq{4.67}, we further deduce $\phi\equiv 0$, whence the third equation of \qq{4.66} reduces to $\int_{Q_T}b=0$, i.e., $\int_0^Tb=0$. This together with the second equation of \qq{4.66} yields that  $b$ satisfies
  \[  b_t=(\ol m_2-l_1)b,\;\; t\in[0,T];\quad\int_0^Tb=0,\]
which forces $b\equiv 0$. Hence \eqref{4.66} admits only trivial solution $0$ in $\mathbb R\times E\times \mathbb X_p$, and thereby the linear operator $D_{fr\sigma}\mathcal H(f_0,\sigma_0,r_0,0)$ is non-degenerate.

{\it Step 4}. By the implicit function theorem, there exist a constant $0<s_0\ll 1$ and a map
  \[  s\mapsto (f(s),\sigma(s),r(s))
 \in C\big([0,s_0),\mathbb R\times E\times \mathbb X_p\big)\]
such that $(f(0),\sigma(0),r(0))=(f_0,\sigma_0,r_0)$ and $(f(s),\sigma(s),r(s))$ solves uniquely $\mathcal H(f,\sigma,r,s)=0$ for $s\in[0,s_0)$. Therefore, $(f,g)=(f(s), \sigma(s)+r(s))$ for $s\in[0,s_0)$, and $\lim_{d_1\to\infty}f=\lim_{s\to 0}f(s)=B_1$, which implies \eqref{4.63}. By the continuous embedding $H^1((0,T), W_p^2(\Omega))\hookrightarrow C(\ol Q_T)$ as  $p>n$, we obtain that
  \[  \lim_{d_1\to\infty}g=\lim_{s\to 0}(\sigma(s)+r(s))=\varphi_*\mu-\frac{B_1}{A}\varphi_*\Gamma_2+
  \varphi_*\tilde c\quad \mbox{in }C(\ol Q_T),\]
which implies \eqref{4.64}. The proof is complete.
\end{proof}

By an argument similar to the proof of Lemma \ref{lm4.8}, we have the following lemma.\vspace{-2mm}

\begin{lemma}\label{lm4.9}
Assume that the condition {\bf(M3)} holds. Then as $d_2\to\infty$,
   \bee\label{4.69}
\widehat{\mathsf{d}}_1(d_2)=\frac 1Ad_2+B_2+o(1),
 \eee
where\vskip 2pt
 \bee
B_2=\frac{1}{\mathcal{E}_1}\int_{Q_T}\kk((m_1-\ol m_1)\Gamma_1^2-\frac{2\mathcal{E}_1}{\mathcal{E}_2} l_2\Gamma_1\Gamma_2
+\frac{\mathcal{E}^2_1}{\mathcal{E}^2_2}\Gamma_2^2[2 l_2-(m_2-\ol m_2)]\rr)\mathrm dx\mathrm dt.\;\label{4.70}
 \eee
\end{lemma}

We now present the following key results, which characterize the signs of $\mu(d_2,m_2-\theta_{d_1,m_1})$ and $\mu(d_1,m_1-\theta_{d_2,m_2})$, together with the positions of $\widehat{\mathsf{d}}_1(d_2)$ and $\widehat{\mathsf{d}}_2(d_1)$.

\begin{lemma}\label{lm4.10} Let $d_0, D_1, D_2>0$ and strictly increasing $C^1$ functions $\widehat{\mathsf{d}}_2:(D_1,\infty)\to(d_0,\infty)$ and $\widehat{\mathsf{d}}_1:(D_2,\infty)\to (d_0,\infty)$ be given in Lemma \ref{lm4.7}. Then for every $d_1>D_1$,
  \bee
  \begin{cases}\mu(d_2,m_2-\theta_{d_1,m_1})<0,\;\;\forall\;d_2\in (d_0,\widehat{\mathsf{d}}_2(d_1)),\\[0.5mm]
\mu(d_2,m_2-\theta_{d_1,m_1})>0,\;\;\forall\;d_2\in(\widehat{\mathsf{d}}_2(d_1),\infty),
   \end{cases}\label{4.71}\eee
while for every $d_2>D_2$,
   \bee
  \begin{cases}
   \mu(d_1,m_1-\theta_{d_2,m_2})<0,\;\;\forall\;d_1\in (d_0,\widehat{\mathsf{d}}_1(d_2)),\\[0.5mm] \mu(d_1,m_1-\theta_{d_2,m_2})>0,\;\;\forall\;d_1\in(\widehat{\mathsf{d}}_1(d_2),\infty).
  \end{cases}\label{4.72}\eee
Moreover, the following statements hold:
\begin{enumerate}
\item[\rm(1)] If $m_1-m_2$ is spatially homogeneous, then $\widehat{\mathsf{d}}_2(d_1)=d_1$ for all $d_1\in(D_1,\infty)$, and $\widehat{\mathsf{d}}_1(d_2)=d_2$ for all $d_2\in(D_2,\infty)$.

\item[\rm(2)] Suppose further that the condition {\bf(M3)} is satisfied. Then we have
\begin{enumerate}
\item[\rm({2}a)] If $\ol{m_1-m_2}\not\equiv 0$ on $[0,T]$ and $\Pi=0$, then as $d_1,d_2\to\infty$, the distance between the curves $\widehat{\mathsf{d}}_2(d_1)$ and $\widehat{\mathsf{d}}_1(d_2)$ approaches to zero.
\item[\rm({2}b)] If either $\ol{m_1-m_2}\not\equiv 0$ on $[0,T]$ and $\Pi>0$, or $\ol{m_1-m_2}\equiv 0$ on $[0,T]$, then as $d_1,d_2\to\infty$, the curve $\widehat{\mathsf{d}}_2(d_1)$ lies above the curve $\widehat{\mathsf{d}}_1(d_2)$ with distance: $\frac{\mathcal{E}_1\Pi}{\sqrt{\mathcal{E}^2_1+\mathcal{E}^2_2}}>0$.
\item[\rm({2}c)] If $\ol{m_1-m_2}\not\equiv 0$ on $[0,T]$ and $\Pi<0$, then as $d_1,d_2\to\infty$, the curve $\widehat{\mathsf{d}}_2(d_1)$ lies below the curve $\widehat{\mathsf{d}}_1(d_2)$ with distance: $-\frac{\mathcal{E}_1\Pi}
 {\sqrt{\mathcal{E}^2_1+\mathcal{E}^2_2}}>0$.
\end{enumerate}
\end{enumerate}
\end{lemma}

\begin{remark} When $m_1(x,t)-m_2(x,t)$ is spatially homogeneous, by the assumption $\int_{Q_T}(m_1-m_2)=0$ and $m_1\not\equiv m_2$, we have $m_1-m_2\equiv \ol m_1-\ol m_2$ and $\ol m_1\not\equiv \ol m_2$ on $[0,T]$. Thus, $m_1-\ol m_1\equiv m_2-\ol m_2$ on $[0,T]$. Combined with \eqref{1.10}, \eqref{1.12} yields that $\Gamma_1\equiv \Gamma_2$ on $\overline Q_T$ and $\mathcal{E}_1=\mathcal{E}_2$, which gives $\Pi=0$. Hence, the condition that $m_1-m_2$ is spatially homogeneous is a special situation of Case (2a) in Lemma \ref{lm4.10} without the condition {\bf(M3)}.
\end{remark}

\begin{proof}[Proof of Lemma \ref{lm4.10}]\; It follows from \eqref{4.48} and \eqref{4.54} that \eqref{4.71} holds for all $d_1>D_1$. Similarly, \eqref{4.72} holds for all $d_2>D_2$.

We first prove the assertion (1). Given that $m_1(x,t)-m_2(x,t)$ is spatially homogeneous and $\int_{Q_T}(m_1-m_2)=0$, we have $\int_0^T(m_1-m_2)=0$. It then follows from Proposition \ref{lm2.1}(4) and Lemma \ref{lm4.7} that
 \[\mu(d_2,m_1-\theta_{d_1,m_1})=\mu(d_2,m_2-\theta_{d_1,m_1}),\;\;
 \mu(d_2,m_2-\theta_{d_1,m_1})=0 \Longleftrightarrow d_2=\widehat{\mathsf{d}}_2(d_1).\]
Since $\mu(d_1,m_1-\theta_{d_1,m_1})=0$, the above implies $\widehat{\mathsf{d}}_2(d_1)=d_1$ for all $d_1\in(D_1,\infty)$. Similarly, $\widehat{\mathsf{d}}_1(d_2)=d_2$ for all $d_2\in(D_2,\infty)$.

We next prove the assertion (2). Curves $\widehat{\mathsf{d}}_2(d_1)$ and $\widehat{\mathsf{d}}_1(d_2)$ can be written as
 \[ L_1: d_2=Ad_1+B_1+o(1),\;\;\; L_2: d_2=Ad_1-AB_2+o(1),\;\;\text{as}\;\;d_1\to\yy\]
by \eqref{4.63} and \eqref{4.69}. The distance between the curves $\widehat{\mathsf{d}}_2(d_1)$ and $\widehat{\mathsf{d}}_1(d_2)$ is
 \bee
 \mathrm{dist}(L_1, L_2)=\frac{|B_1+AB_2|}{\sqrt{1+A^2}}+o(1)=
 \frac{\mathcal{E}_1}{\sqrt{\mathcal{E}^2_1+\mathcal{E}^2_2}}|B_1+AB_2|
 +o(1),\;\;\text{as}\;\;d_1\to\yy.\label{4.73}\eee
Let $\Pi$ be defined by \eqref{1.13}. Since the condition {\bf(M3)} holds. By the expressions of $B_1$ and $B_2$ (cf. \eqref{4.61}, \eqref{4.70}), we derive that by the careful computation
 \begin{align}
 B_1+AB_2=\;&\int_{Q_T} l_1\kk(\frac{\sqrt{\mathcal{E}_2}}{\mathcal{E}_1}\Gamma_1-\frac{1}{\sqrt{\mathcal{E}_2}}
 \Gamma_2\rr)^2+\frac{1}{\mathcal{E}_2}\int_{Q_T}(m_2-\ol m_2- l_1)\Gamma_2^2\nm\\
 \;&-\frac{\mathcal{E}_2}{\mathcal{E}^2_1}\int_{Q_T}(m_1-\ol m_1- l_1)\Gamma_1^2+A\Bigg[ \frac{1}{\mathcal{E}_1}\int_{Q_T}(m_1-\ol m_1- l_2)\Gamma_1^2\nm\\
\;&-\frac{\mathcal{E}_1}{\mathcal{E}^2_2}\int_{Q_T}(m_2-\ol m_2- l_2)\Gamma_2^2+\int_{Q_T}l_2\kk(\frac{\sqrt{\mathcal{E}_1}}{\mathcal{E}_2}\Gamma_2
-\frac{1}{\sqrt{\mathcal{E}_1}}\Gamma_1\rr)^2 \Bigg]\nm\\
=\;&\int_{Q_T}(l_1+l_2)\kk(\frac{\sqrt{\mathcal{E}_2}}{\mathcal{E}_1}\Gamma_1
 -\frac{1}{\sqrt{\mathcal{E}_2}}\Gamma_2\rr)^2+\frac{\mathcal{E}_2}{\mathcal{E}^2_1}
 \int_{Q_T}(l_1-l_2)\Gamma_1^2+\frac{1}{\mathcal{E}_2}\int_{Q_T}(l_2-l_1)\Gamma_2^2\nm\\
=\;&\frac 2{\mathcal{E}_1}\int_{Q_T}\kk(A l_1\Gamma_1^2-(l_1+l_2)\Gamma_1\Gamma_2+\frac 1Al_2\Gamma_2^2 \rr)\nm\\[2mm]
 =\;&\Pi. \label{4.74}
  \end{align}

In case (2a), as $\Pi=0$, we have $B_1+AB_2=0$, which implies $\mathrm{dist}(L_1, L_2)\to 0$ as $d_1\to\infty$.

In case (2b), if $\overline m_1\not\equiv\overline m_2$ on $[0,T]$ and $\Pi>0$, the assertion follows from \eqref{4.73} and \qq{4.74}.
If $\overline m_1\equiv\overline m_2$ on $[0,T]$, then $l_1=l_2$ for all $t\in[0,T]$, and
   \[  B_1+AB_2=2\int_{Q_T}l_1\kk(\frac{\sqrt{\mathcal{E}_2}}{\mathcal{E}_1}\Gamma_1
   -\frac{1}{\sqrt{\mathcal{E}_2}}\Gamma_2\rr)^2\ge 0.\]
We claim that $B_1+AB_2>0$. Otherwise, one has $A\Gamma_1\equiv \Gamma_2$,
which together with \eqref{1.10} and \eqref{1.12} implies
$A(m_1-\ol m_1)\equiv m_2-\ol m_2$.
Building on this, we obtain $\mathcal{E}_1=\mathcal{E}_2$, i.e., $A=1$, and so $m_1\equiv m_2$. This contradicts the condition {\bf(M)}. Thus $B_1+AB_2>0$, and the assertion follows from \eqref{4.73} and \qq{4.74}.

In case (2c), since $\Pi<0$, the curve $\widehat{\mathsf{d}}_2(d_1)$ lies blow the curve $\widehat{\mathsf{d}}_1(d_2)$, and the assertion follows from \eqref{4.73} and \qq{4.74}.
\end{proof}

\section{At least one diffusion rate is large: Proofs of Theorems \ref{th1.3}-\ref{th1.5}}\label{sec5}

This section is devoted to establishing the results concerning large diffusion rates and mixed-scale asymptotics. We begin with Theorem \ref{th1.3} on global stability, then proceed to the detailed classification in Theorem \ref{th1.4}, and conclude with the asymptotic profile result in Theorem \ref{th1.5}.

\subsection{Proof of Theorem \ref{th1.3}}

(1) By Lemma \ref{lm4.2}(2), there exists $\bar d\gg 1$, depending only on $m_i$, such that $(\theta_{d_1,m_1},0)$ is globally asymptotically stable for all $d_2>d_1>\bar d$. Now take $a=\ep$ and $b=\bar d$ in Lemma \ref{lm4.2}(1). Then there exists $d_2^{\ep}=d_2^{\ep,\bar d}>\bar d$ such that $(\theta_{d_1,m_1},0)$ is globally asymptotically stable for all $(d_1,d_2)\in[\ep,\bar d]\times(d_2^{\ep},\infty)$. Therefore, $(\theta_{d_1,m_1},0)$ is globally asymptotically stable for all $(d_1,d_2)\in[\ep,d_2)\times(d_2^{\ep},\infty)$.

(2) This follows directly from Lemma \ref{lm4.2}(1).

\subsection{Proof of Theorem \ref{th1.4}}

Assume that conditions {\bf(M1)}--{\bf(M2)} hold and $\widehat{m}_i(x)\not\equiv\text{constant}$ on $\boo$ for $i=1,2$. Let $ d_0, D_1, D_2>0$ and strictly increasing $C^1$ functions $\widehat{\mathsf{d}}_2:(D_1,\infty)\to( d_0,\infty)$ and $\widehat{\mathsf{d}}_1:(D_2,\infty)\to ( d_0,\infty)$ be obtained in Lemma \ref{lm4.7}. The estimate \eqref{4.50} gives $\widehat{\mathsf d}_1'(d_2)>0$ for $d_2>D_2$. Let $\widetilde{\mathsf{d}}_2(d_1)$ denote the inverse function of $\widehat{\mathsf{d}}_1(d_2)$. Define $D_1^\sharp :=\lim\limits_{d_2\searrow D_2}\widehat{\mathsf d}_1(d_2)$. From \qq{4.49} and the second relation in \qq{4.50} we see that for all $(d_1,d_2)\in(D_1^\sharp,\infty)\times(D_2,\infty)$,
\[\mu(d_1,m_1-\theta_{d_2,m_2})=0 \Longleftrightarrow d_1=\widehat{\mathsf{d}}_1(d_2)\Longleftrightarrow d_2=\widetilde{\mathsf{d}}_2(d_1), \]
and $\lim_{d_1\to\infty}\frac{\widetilde{\mathsf{d}}_2(d_1)}{d_1}=A$.

By virtue of Lemma \ref{lm4.7}, it is concluded that
\begin{equation}\label{5.1}
\begin{cases}
\mu(d_2,m_2-\theta_{d_1,m_1})<0
&\text{if }d_1>D_1,
\quad  d_0<d_2<\widehat{\mathsf d}_2(d_1),\\
\mu(d_2,m_2-\theta_{d_1,m_1})>0
&\text{if }d_1>D_1,
\quad d_2>\widehat{\mathsf d}_2(d_1),\\
\mu(d_1,m_1-\theta_{d_2,m_2})>0
&\text{if }d_1>D_1^\sharp,
\quad D_2<d_2<\widetilde{\mathsf d}_2(d_1),\\
\mu(d_1,m_1-\theta_{d_2,m_2})<0
&\text{if }d_1>D_1^\sharp,
\quad d_2>\widetilde{\mathsf d}_2(d_1).
\end{cases}
\end{equation}
Furthermore,
\begin{equation}\label{5.2}
\lim_{d_1\to\infty}
\frac{\widehat{\mathsf d}_2(d_1)}{d_1}
=\lim_{d_1\to\infty}
\frac{\widetilde{\mathsf d}_2(d_1)}{d_1}
=A=\frac{\mathcal E(m_2)}{\mathcal E(m_1)}.
\end{equation}
Thanks to Lemma \ref{lm4.6}, there exist constants
$\delta_2^*>0$ and $\bar d_1^0>0$, together with a strictly decreasing $C^1$ function $\mathsf d_2:(\bar d_1^0,\infty)\to(0,\delta_2^*)$ such that, for every $d_1>\bar d_1^0$,
\begin{equation}\label{5.3}
\begin{cases}
\mu(d_1,m_1-\theta_{d_2,m_2})<0,&\forall\;0<d_2<\mathsf d_2(d_1),\\
\mu(d_1,m_1-\theta_{d_2,m_2})>0, &\forall\;\mathsf d_2(d_1)<d_2\leq\delta_2^*.
\end{cases}
\end{equation}
Moreover,
  \begin{equation}\label{5.4}
\lim_{d_1\to\infty}d_1\mathsf d_2(d_1)=\frac{\mathcal E_{12}}{\mathcal I_2},
\quad\mathsf d_2'(d_1)=-\frac{\mathcal E_{12}}{\mathcal I_2}d_1^{-2}
+o\left(d_1^{-2}\right).
\end{equation}
Note that $\widehat m_2(x)$ is assumed to be non-constant on $\overline\Omega$ in Theorem \ref{th1.4}. By Lemma \ref{lm4.3}(1), there exist constants $\bar d_1^1>0$ and $\varepsilon_1>0$ such that
 \begin{equation}\label{5.5}
\mu(d_2,m_2-\theta_{d_1,m_1})<0,\;\;\forall\; d_1>\bar d_1^1,\;0<d_2\leq\varepsilon_1.
\end{equation}

Choose a constant $\delta_*>0$, independently of $d_1$, such that
  \begin{equation}\label{5.6}
  \delta_*<\min\{\delta_2^*,\varepsilon_1, d_0,D_2\}.
  \end{equation}
The equation for $l_{m_1}$ and the equal-mass condition {\bf(M)} imply
\[ \frac{1}{T|\Omega|}\int_{Q_T}(m_2(x,t)-l_{m_1}(t))\,\mathrm dx\,\mathrm dt=0. \]
Let $\varphi_{d_2}>0$ be the principal eigenfunction associated with
$\mu(d_2,m_2-l_{m_1})$. Dividing its eigenvalue equation by
$\varphi_{d_2}$ and integrating over $Q_T$, we obtain
  \[\mu(d_2,m_2-l_{m_1})=-\frac{d_2}{T|\Omega|}
\int_{Q_T}|\nabla\log\varphi_{d_2}|^2
\,\mathrm dx\,\mathrm dt<0,\;\;\forall\; d_2>0.\]
The inequality is strict because $m_2-l_{m_1}$ is spatially heterogeneous. Hence,
 \begin{equation}\label{eq5.7}
\max_{d_2\in[\delta_*, d_0]}\mu(d_2,m_2-l_{m_1})<0.
\end{equation}

Set $\xi:=\|\theta_{d_1,m_1}- l_{m_1}\|_{L^\infty(Q_T)}$. Then on $\overline{Q}_T$,
   \[m_2-l_{m_1}-\xi\leq m_2-\theta_{d_1,m_1}\leq m_2-l_{m_1}+\xi\]
By Proposition \ref{lm2.1}(2), $\mu(d,\cdot)$ is decreasing in its second argument, so \[\mu(d,m_2-l_{m_1}-\xi)\geq\mu(d,m_2-\theta_{d_1,m_1})\geq\mu(d,m_2-l_{m_1}+\xi).\]
We further use the identities
 \begin{align*}
 \mu(d,m_2-l_{m_1}-\xi)=\mu(d,m_2-l_{m_1})+\xi,\\ \mu(d,m_2-l_{m_1}+\xi)=\mu(d,m_2-l_{m_1})-\xi,
 \end{align*}
to deduce
\begin{equation}\label{5.8}
|\mu(d,m_2-\theta_{d_1,m_1})-\mu(d,m_2-l_{m_1})|\leq\xi=\|\theta_{d_1,m_1}-l_{m_1}\|_{L^\infty(Q_T)}.
\end{equation}
Combined with the uniform convergence $\theta_{d_1,m_1}\to l_{m_1}$ on $\overline Q_T$ as $d_1\to\infty$, the inequalities \eqref{5.8} and \eqref{eq5.7} yield 
\begin{equation}\label{5.9}
\mu(d_2,m_2-\theta_{d_1,m_1})<0,\;\;\forall\;d_2\in[\delta_*, d_0]
\end{equation}
whenever $d_1$ is sufficiently large.

On the other hand, by Lemma \ref{lm2.7}(1),
  \[\sup_{\delta_*\le d_2\le D_2}\left(\|m_1-\theta_{d_2,m_2}\|_{L^\infty(Q_T)}
  +\|(m_1-\theta_{d_2,m_2})_t\|_{L^2(0,T;L^2(\Omega))} \right)<\infty.\]
Applying Proposition \ref{p2.2} with the above bound, yields a constant $C_*>0$, independent of $d_1$ and $d_2$, such that for sufficiently large $d_1>0$,
  \[\sup_{d_2\in[\delta_*,D_2]}\left|\mu(d_1,m_1-\theta_{d_2,m_2})+
\frac{1}{T|\Omega|}\int_{Q_T}(m_1-\theta_{d_2,m_2})\,\mathrm dx\,\mathrm dt
\right|\le {C_*}d_1^{-1}.\]
It follows that
 \[\lim_{d_1\to\infty}\mu(d_1,m_1-\theta_{d_2,m_2})=
-\frac{1}{T|\Omega|}\int_{Q_T}(m_1-\theta_{d_2,m_2})\mathrm dx\mathrm dt\;\;\;
 \text{uniformly for}\;\;d_2\in[\delta_*,D_2].\]
By {\bf(M)} and the equation for $\theta_{d_2,m_2}$, we see that
  \[ -\frac{1}{T|\Omega|}\int_{Q_T}(m_1-\theta_{d_2,m_2})\,\mathrm dx\,\mathrm dt
  =\frac{d_2}{T|\Omega|}\int_{Q_T}\left|\frac{\nabla\theta_{d_2,m_2}}{\theta_{d_2,m_2}}\right|^2\,\mathrm dx\,\mathrm dt>0. \]
Consequently,
\begin{equation}\label{5.10}
\mu(d_1,m_1-\theta_{d_2,m_2})>0,\;\;\forall\; d_2\in[\delta_*,D_2]
\end{equation}
whenever $d_1$ is sufficiently large.\zzz
\begin{enumerate}
\item[(1)] When $m_1-m_2$ is spatially homogeneous, we choose a constant $\bar d_1>\max\{D_1,D_1^\sharp,\bar d_1^0,\bar d_1^1\}$.
 \item[(2)] When either $\ol m_1(t)\equiv\ol m_2(t)$ on $[0,T]$, or $\ol m_1(t)\not\equiv\ol m_2(t)$ on $[0,T]$ and $\Pi>0$, Lemma \ref{lm4.10}(2b) yields a constant $D_{+}>0$ such that $\widehat{\mathsf d}_2(d_1)>\widetilde{\mathsf d}_2(d_1)$  for all $d_1>D_{+}$. 
In this case, we choose a constant $\bar d_1>\max\{D_1,D_1^\sharp,\bar d_1^0,\bar d_1^1,D_{+}\}$.
\item[(3)] When $\ol m_1(t)\not\equiv\ol m_2(t)$ on $[0,T]$ and $\Pi<0$, Lemma
\ref{lm4.10}(2c) yields a constant $D_->0$ such that $\widehat{\mathsf d}_2(d_1)<\widetilde{\mathsf d}_2(d_1)$ for all $d_1>D_{-}$. 
In this case, we choose a constant $\bar d_1>\max\{D_1,D_1^\sharp,\bar d_1^0,\bar d_1^1,D_{-}\}$.\zzz
\end{enumerate}
Moreover, in the above three cases, $\bar d_1$ can be enlarged so that the sign relations \eqref{5.9} and
\eqref{5.10} hold for every $d_1>\bar d_1$, and
  \begin{equation}\label{5.11}
\mathsf d_2(d_1)<\delta_*,\;\;\forall\; d_1>\bar d_1.
  \end{equation}

Combining \eqref{5.5}, \eqref{5.9}, and the first two relations in \eqref{5.1}, we obtain
\begin{equation}\label{5.12}
\begin{cases}
\mu(d_2,m_2-\theta_{d_1,m_1})<0&\text{when}\;\;0<d_2<\widehat{\mathsf d}_2(d_1),\\
\mu(d_2,m_2-\theta_{d_1,m_1})>0&\text{when}\;\;d_2>\widehat{\mathsf d}_2(d_1)
\end{cases}
\end{equation}
for all $d_1>\bar d_1$. Similarly, \eqref{5.3}, \eqref{5.10}, and the last two relations in \eqref{5.1} yield
\begin{equation}\label{5.13}
\begin{cases}
\mu(d_1,m_1-\theta_{d_2,m_2})<0&\text{when}\;\; 0<d_2<\mathsf d_2(d_1),\\
\mu(d_1,m_1-\theta_{d_2,m_2})>0&\text{when}\;\;\mathsf d_2(d_1)<d_2<\widetilde{\mathsf d}_2(d_1),\\
\mu(d_1,m_1-\theta_{d_2,m_2})<0&\text{when}\;\;d_2>\widetilde{\mathsf d}_2(d_1)
\end{cases}
\end{equation}
for all $d_1>\bar d_1$.

We next apply Lemma \ref{lm4.4} with this fixed value of $\bar d_1$. There exist a constant $d_{2,\bar d_1}>0$ and a strictly decreasing
$C^1$ function $\mathsf d_1:(d_{2,\bar d_1},\infty)\to(0,\infty)$, satisfying $\mathsf d_1(d_2)=O(d_2^{-1})$ as $d_2\to\infty$, such that, whenever $d_2>d_{2,\bar d_1}$,
\begin{equation}\label{5.14}
\begin{cases}
\mu(d_2,m_2-\theta_{d_1,m_1})<0,&\forall\;0<d_1<\mathsf d_1(d_2),\\
\mu(d_2,m_2-\theta_{d_1,m_1})>0,&\forall\;\mathsf d_1(d_2)<d_1\leq\bar d_1,\\
\mu(d_1,m_1-\theta_{d_2,m_2})<0,&\forall\;0<d_1\leq\bar d_1.
\end{cases}
\end{equation}
Only at this stage do we choose $\bar d_2$ so large that
\begin{equation}\label{5.15}
\bar d_2>\max\{d_{2,\bar d_1},D_2, d_0,\bar d_1\}
\end{equation}
and
\begin{equation}\label{5.16}
\mathsf d_1(d_2)<\bar d_1,\;\;\forall\;d_2>\bar d_2.
\end{equation}
This final choice changes only $\bar d_2$ and does not alter the
already fixed value of $\bar d_1$.

\begin{proof}[Proof of Theorem $\ref{th1.4}(1)$] Suppose that $m_1-m_2$ is spatially homogeneous. In view of \eqref{4.8} and the definition of  $\bar d_1$,
\begin{equation}\label{5.17}
\mu(d_2,m_2-\theta_{d_1,m_1})>0,\;\;\;
\mu(d_1,m_1-\theta_{d_2,m_2})<0\quad\text{whenever }d_2>d_1>\bar d_1.
\end{equation}

For $d_2>\bar d_2$ and $0<d_1<\mathsf d_1(d_2)$. Combining \eqref{5.16} with the first and third relations in \eqref{5.14}, we obtain $\mu(d_2,m_2-\theta_{d_1,m_1})<0$ and $\mu(d_1,m_1-\theta_{d_2,m_2})<0$. Consequently, system \eqref{1.4} is uniformly persistent, and \eqref{1.5} possesses a linearly stable positive solution by Proposition \ref{p2.4}(3).

For $d_2>\bar d_2$ and $\mathsf d_1(d_2)<d_1<d_2$. If $d_1\leq\bar d_1$, the second and third relations in \eqref{5.14} yield $\mu(d_2,m_2-\theta_{d_1,m_1})>0$ and $\mu(d_1,m_1-\theta_{d_2,m_2})<0$. If instead $d_1>\bar d_1$, the same sign relations hold by virtue of \eqref{5.17}. It then follows from Proposition \ref{p2.3} that \((\theta_{d_1,m_1},0)\) is linearly stable, whereas \((0,\theta_{d_2,m_2})\) is linearly unstable.
\end{proof}\vskip 4pt

\begin{proof}[Proof of Theorem $\ref{th1.4}(2)$]
Assume that $m_1(x,t)-m_2(x,t)$ is spatially heterogeneous, and condition {\bf(M3)} holds.

(2a) Either $\bar m_1(t)\equiv\bar m_2(t)$ on $[0,T]$, or $\bar m_1(t)\not\equiv\bar m_2(t)$ on $[0,T]$ and $\Pi>0$.

By \eqref{5.11}, \eqref{5.6}, and the range of $\widetilde{\mathsf d}_2$, we have
$\mathsf d_2(d_1)<\widetilde{\mathsf d}_2(d_1)$ for $d_1>\bar d_1$. Lemma \ref{lm4.10}(2b) together with the choice of $\bar d_1$ further gives
  \[0<\mathsf d_2(d_1)<\widetilde{\mathsf d}_2(d_1)<\widehat{\mathsf d}_2(d_1)\quad\text{for }d_1>\bar d_1.\]
From \eqref{5.12}–\eqref{5.14}, we draw the following conclusions.

(1) In the $(d_1,d_2)$‑region defined by \eqref{1.x14},
  \[\mu(d_2,m_2-\theta_{d_1,m_1})<0 \quad\text{and}\quad \mu(d_1,m_1-\theta_{d_2,m_2})<0.\]
Consequently, system \eqref{1.4} is uniformly persistent, and \eqref{1.5} admits a linearly stable positive solution by Proposition \ref{p2.4}(3);

(2) In the $(d_1,d_2)$‑region defined by \eqref{1.x15},
  \[\mu(d_2,m_2-\theta_{d_1,m_1})>0 \quad\text{and}\quad \mu(d_1,m_1-\theta_{d_2,m_2})<0.\]
Hence, by Proposition \ref{p2.3}, $(\theta_{d_1,m_1},0)$ is linearly stable, while $(0,\theta_{d_2,m_2})$ is linearly unstable;

(3) For $(d_1,d_2)\in(\bar d_1,\infty)\times(\mathsf d_2(d_1),\widetilde{\mathsf d}_2(d_1))$,
  \[\mu(d_2,m_2-\theta_{d_1,m_1})<0 \quad\text{and}\quad \mu(d_1,m_1-\theta_{d_2,m_2})>0.\]
Hence, by Proposition \ref{p2.3}, $(\theta_{d_1,m_1},0)$ is linearly unstable, while $(0,\theta_{d_2,m_2})$ is linearly stable.

This proves part (2a).

(2b) $\bar m_1(t)\not\equiv\bar m_2(t)$ on $[0,T]$ and $\Pi<0$.

By \eqref{5.11}, \eqref{5.6}, and the range of $\widehat{\mathsf d}_2$, we have $\mathsf d_2(d_1)<\widehat{\mathsf d}_2(d_1)$ for $d_1>\bar d_1$. Lemma \ref{lm4.10}(2c) together with the choice of $\bar d_1$ further gives $0<\mathsf d_2(d_1)<\widehat{\mathsf d}_2(d_1)<\widetilde{\mathsf d}_2(d_1)$ for $d_1>\bar d_1$. Relations \eqref{5.12}–\eqref{5.14} imply the following.

(1) In the $(d_1,d_2)$ region defined by \eqref{1.x16}, 
 \[\mu(d_2,m_2-\theta_{d_1,m_1})<0 \quad\text{and}\quad \mu(d_1,m_1-\theta_{d_2,m_2})<0.\] Consequently, system \eqref{1.4} is uniformly persistent, and \eqref{1.5} admits a linearly stable positive solution by Proposition \ref{p2.4}(3);

(2) In the $(d_1,d_2)$ region defined by \eqref{1.x17}, 
  \[\mu(d_2,m_2-\theta_{d_1,m_1})>0 \quad\text{and}\quad \mu(d_1,m_1-\theta_{d_2,m_2})<0.\] 
Hence, by Proposition \ref{p2.3}, $(\theta_{d_1,m_1},0)$ is linearly stable, while $(0,\theta_{d_2,m_2})$ is linearly unstable;

(3) In the $(d_1,d_2)$ region defined by \eqref{1.x18}, 
 \[\mu(d_2,m_2-\theta_{d_1,m_1})<0 \quad\text{and}\quad \mu(d_1,m_1-\theta_{d_2,m_2})>0.\] 
Hence, by Proposition \ref{p2.3}, $(\theta_{d_1,m_1},0)$ is linearly unstable, while $(0,\theta_{d_2,m_2})$ is linearly stable;

(4) In the $(d_1,d_2)$ region defined by \eqref{1.x18}, 
 \[\mu(d_2,m_2-\theta_{d_1,m_1})>0 \quad\text{and}\quad \mu(d_1,m_1-\theta_{d_2,m_2})>0.\] 
Hence,  both $(\theta_{d_1,m_1},0)$ and $(0,\theta_{d_2,m_2})$ are linearly stable by Proposition \ref{p2.3}, and \eqref{1.5} admits a linearly unstable positive solution by Proposition \ref{p2.4}(4).

This proves part (2b).
\end{proof}

\subsection{Proof of Theorem \ref{th1.5}}

In view of \qq{5.15} and the first and third inequalities in \qq{5.14},
\begin{align}
\mu(d_2,m_2-\theta_{d_1,m_1})<0,\quad \mu(d_1,m_1-\theta_{d_2,m_2})<0
\quad\text{for }(d_1,d_2)\in(0,\mathsf{d}_1(d_2))\times(\bar d_2,\infty).\label{5.22}
\end{align}
The existence of a linearly stable positive solution $(U,V)$ to \qq{1.5} follows directly from \eqref{5.22} and Proposition \ref{p2.4}(3).

We next derive the asymptotic profile of $(U,V)$ as $d_1\to 0$ and $d_2\to\infty$.
Theorem \ref{th1.5} imposes condition \eqref{1.8}, namely $\widehat{m}_1(x)\not\equiv \widehat{m}_2(x)$ on $\overline{\Omega}$. Let $\{\widetilde u_{k,d_1,d_2}\}_{k\ge 1}$, $\{\undl u_{k,d_1,d_2}\}_{k\ge 1}$, $\{\widetilde v_{k,d_1,d_2}\}_{k\ge 1}$ and $\{\undl v_{k,d_1,d_2}\}_{k\ge 1}$ denote the sequences of functions in $C_T^{2+\alpha,1+\alpha/2}(\overline Q_T)$ defined in \eqref{3.2}. Clearly, relation \eqref{3.3} remains valid.
\vskip 4pt

{\it Step 1: The limits of $\bigl(\widetilde u_{k,d_1,d_2},\,\undl u_{k,d_1,d_2},\,\widetilde v_{k,d_1,d_2},\,\undl v_{k,d_1,d_2}\bigr)$ and $(U,V)$ as $d_1\to 0$}.

Following the proof of Theorem \ref{th1.1}, we deduce that there exist  $\dd\bigl(\widetilde U_{k,d_2},\,\undl U_{k,d_2}, \,\widetilde V_{k,d_2},\,\undl V_{k,d_2}\bigr)$ with $\widetilde U_{k,d_2}, \undl U_{k,d_2}\in C_T^{\alpha,1}(\overline Q_T)$ and $\widetilde V_{k,d_2}, \undl V_{k,d_2}\in C_T^{2+\alpha,1+\alpha/2}(\overline Q_T)$ such that
  \bee\label{5.23}
 \lim_{d_1\to 0}\bigl(\widetilde u_{k,d_1,d_2},\,\undl u_{k,d_1,d_2},\,\widetilde v_{k,d_1,d_2},\,\undl v_{k,d_1,d_2}\bigr)=\bigl(\widetilde U_{k,d_2},\,\undl U_{k,d_2}, \,\widetilde V_{k,d_2},\,\undl V_{k,d_2}\bigr)\;\;\;\text{in}\;\;C(\ol Q_T).
  \eee
Moreover, these functions satisfy
  \bee\label{5.24}\begin{cases}
  \partial_t\widetilde U_{k+1,d_2}=\widetilde U_{k+1,d_2}\bigl(m_1(x,t)-\undl V_{k,d_2}-\widetilde U_{k+1,d_2}\bigr) & \text{in }Q_T,\\[1mm]
  \partial_t\undl V_{k,d_2}=d_2\Delta\undl V_{k,d_2}+\undl V_{k,d_2}\bigl(m_2(x,t)-\widetilde U_{k,d_2}-\undl V_{k,d_2}\bigr)&\text{in }Q_T, \\[1mm]
 \partial_t\undl U_{k,d_2}=\undl U_{k,d_2}\bigl(m_1(x,t)-\widetilde V_{k,d_2}-\undl U_{k,d_2}\bigr) & \text{in }Q_T, \\[1mm]
 \partial_t\widetilde V_{k+1,d_2}=d_2\Delta\widetilde V_{k+1,d_2}+\widetilde V_{k+1,d_2}\bigl(m_2(x,t)-\undl U_{k,d_2}-\widetilde V_{k+1,d_2}\bigr)& \text{in }Q_T,\\[1mm]
 \partial_{\nu}\undl V_{k,d_2}=\partial_{\nu}\widetilde V_{k,d_2}=0  & \text{on }S_T.
 \end{cases}
 \eee
In view of \eqref{3.3} and \eqref{5.23}, and use an analogous argument to the derivation of \qq{3.6} we have
  \begin{align}\label{5.25}
\Theta_{m_1-\theta_{d_2,m_2}}\leq \undl U_{k,d_2}\leq \undl U_{k+1,d_2}\leq \dd\liminf_{d_1\to 0}U \leq \limsup_{d_1\to 0}U\leq \widetilde U_{k+1,d_2}\leq \widetilde U_{k,d_2}\leq \Phi_1,\\
\Theta_{d_2,m_2-\Phi_1}\leq \undl V_{k,d_2}\leq \undl V_{k+1,d_2}\dd\leq \liminf_{d_1\to 0}V \leq\limsup_{d_1\to 0}V\leq \widetilde V_{k+1,d_2}\leq \widetilde V_{k,d_2}\leq \theta_{d_2,m_2}\nm
  \end{align}
uniformly in $\ol Q_T$, where $\Theta_{d,m}$ and $\Theta_{m}$ are the maximal nonnegative solutions to \eqref{1.2} and \qq{1.6}, respectively.

{\it Step 2: The limit of $\bigl(\widetilde U_{k,d_2},\,\undl U_{k,d_2}, \,\widetilde V_{k,d_2},\,\undl V_{k,d_2}\bigr)$ as $k\to\yy$}.

There exist four bounded  and time $T$-periodic functions $\widetilde U_{d_2},\undl U_{d_2},\widetilde V_{d_2}$ and $\undl V_{d_2}$ such that
  \bee\label{5.26}
\lim_{k\to\infty}\bigl(\widetilde U_{k,d_2},\,\undl U_{k,d_2},\,\widetilde V_{k,d_2},\,\undl V_{k,d_2}\bigr)=\bigl(\widetilde U_{d_2},\,\undl U_{d_2},\,\widetilde V_{d_2},\,\undl V_{d_2}\bigr)\quad \text{pointwise on }\; \ol Q_T.
 \eee
The $L^p$ theory for time-periodic parabolic equations (cf. \cite{L99,KS19}) shows that $\{\widetilde V_{k,d_2}\}$ and $\{\undl V_{k,d_2}\}$ are bounded in $W_p^{2,1}(Q_T)$ for any $p>n$. Hence, $\widetilde V_{k,d_2},\undl V_{k,d_2}\in C^{1+\alpha,\frac{1+\alpha}{2}}_T(\overline Q_T)$, and are also bounded in $C^{1+\alpha,\frac{1+\alpha}{2}}(\ol Q_T)$. Applying the variation-of-constants formula to the equations of $\widetilde U_{k,d_2}$ and $\undl U_{k,d_2}$ in \eqref{5.24}, we deduce that $\{\widetilde U_{k,d_2}\}$ and $\{\undl U_{k,d_2}\}$ are bounded in $C^{\alpha,1}(\ol Q_T)$. Then the Schauder theory for time-periodic parabolic equations implies that $\{\widetilde V_{k,d_2}\}$ and $\{\undl V_{k,d_2}\}$ are bounded in $C^{2+\alpha,1+\alpha/2}(\ol Q_T)$. Combined with \eqref{5.26}, this gives $\widetilde V_{d_2},\undl V_{d_2}\in C^{2+\gamma,1+\gamma/2}_T(\overline Q_T)$ and
   \[\lim_{k\to\infty}\|\widetilde V_{k,d_2}-\widetilde V_{d_2}\|_{C^{2+\gamma,1+\gamma/2}(\ol Q_T)}=0, \quad \lim_{k\to\infty}\|\undl V_{k,d_2}-\undl V_{d_2}\|_{C^{2+\gamma,1+\gamma/2}(\ol Q_T)}=0\]
for some $0<\gamma<\alpha$. By \eqref{5.24}, an argument similar to the proof of Theorem \ref{th1.1} and Dini's Theorem, we have that $\widetilde U_{d_2},\undl U_{d_2}\in C^{\gamma,1}_T(\overline Q_T)$,
\[  \lim_{k\to\infty}\|\widetilde U_{k,d_2}-\widetilde U_{d_2}\|_{C(\ol Q_T)}=0, \quad \lim_{k\to\infty}\|\undl U_{k,d_2}-\undl U_{d_2}\|_{C(\ol Q_T)}=0,\]
and $(\widetilde U_{d_2},\undl V_{d_2},\undl U_{d_2},\widetilde V_{d_2})$ is a $T$-periodic solution of
  \bee\label{5.27}\begin{cases}
\partial_t\widetilde U_{d_2}=\widetilde U_{d_2}\big(m_1(x,t)-\undl V_{d_2}-\widetilde U_{d_2}\big)& \text{in }Q_T,\\[1mm]
\partial_t\undl V_{d_2}=d_2\Delta \undl V_{d_2}+\undl V_{d_2}\big(m_2(x,t)-\widetilde U_{d_2}-\undl V_{d_2}\big)& \text{in }Q_T, \\[1mm]
\partial_t \undl U_{d_2}=\undl U_{d_2}\big(m_1(x,t)-\widetilde V_{d_2}-\undl U_{d_2} \big)& \text{in }Q_T, \\[1mm]
\partial_t\widetilde V_{d_2}=d_2\Delta \widetilde V_{d_2}+\widetilde V_{d_2}\big(m_2(x,t)-\undl U_{d_2}-\widetilde V_{d_2}\big)& \text{in }Q_T, \\[1mm]
\partial_{\nu}\undl V_{d_2}=\partial_{\nu}\widetilde V_{d_2}=0
  & \text{on }S_T.
  \end{cases}\eee

{\it Step 3: We claim that $\min_{\overline\Omega}\int_0^T[m_1(x,t)-\undl V_{d_2}(x,t)]\le 0$}.

Suppose to the contrary that $\min_{\overline\Omega}\int_0^T[m_1(x,t)-\undl V_{d_2}(t)]>0$. By \qq{5.22} and Proposition \ref{lm2.1}(1), $\max_{\boo}\int_0^T(m_1-\theta_{d_2,m_2})>0$. There exists $x_0\in\boo$ such that
  \[\int_0^T(m_1(x_0,t)-\theta_{d_2,m_2}(x_0,t))>0,\]
which implies $\Theta_{m_1-\theta_{d_2,m_2}}(x_0,t)>0$ for all $t\in[0,T]$ by \qq{1.6}. It follows from \qq{5.25} and \qq{5.26} that $\widetilde U_{d_2}(x_0,t)>0$  for all $t\in[0,T]$. By the continuity, the set
 \[ \Omega_*=\bigl\{x\in\boo:\widetilde U_{d_2}(x,t)>0,\;\;\forall\, t\in[0,T]\bigr\} \]
is a relatively open subset of $\boo$. We confirm that $\Omega_*$ is also a relatively closed subset of $\boo$. In fact, let $x_i\in\Omega_*$ and $x_i\to\bar x\in\boo$. It then follows from the first equation of \qq{5.27} that
 \[\int_0^T\widetilde U_{d_2}(x_i,t)\dt=\int_0^T(m_1-\undl V_{d_2})(x_i,t)\dt
 \ge \min_{\overline\Omega}\int_0^T(m_1-\undl V_{d_2})>0. \]
Letting $i\to\yy$ yields
 \[\int_0^T\widetilde U_{d_2}(\bar x,t)\dt
 \ge \min_{\overline\Omega}\int_0^T(m_1-\undl V_{d_2})>0. \]
This combined with the first equation of \qq{5.27} implies $\widetilde U_{d_2}(\bar x,t)>0$ for all $t\in[0,T]$, i.e., $\bar x\in\Omega_*$. Thus $\Omega_*=\boo$ and $\widetilde U_{d_2}>0$ on $\ol Q_T$.

Making use of the first equation of \qq{5.27}, we have
  \[ \int_0^T(m_1-\undl V_{d_2}-\widetilde U_{d_2})=0,\;\;\forall\, x\in\overline{\Omega},\]
which yields that
\[ \int_0^T(m_2-\undl V_{d_2}-\widetilde U_{d_2})\mathrm dt=\int_0^T(m_2-m_1)\mathrm dt,\;\;\forall\, x\in\overline{\Omega}.\]
Thanks to the condition {\bf(M2)}, we see that $m_2-\Phi_1$ is spatially heterogeneous. Then by Proposition \ref{lm2.1}(1),
\[ \mu(d_2,m_2-\Phi_1)<-\frac{1}{T|\Omega|}\int_{Q_T}(m_2-\Phi_1)=0,\]
which means that $\Theta_{d_2,m_2-\Phi_1}>0$ on $\overline Q_T$. This together with \eqref{5.25}-\eqref{5.26} yields that $\undl V_{d_2}>0$ on $\overline Q_T$. Hence,
  \[  d_2\int_{Q_T}\bigg| \frac{\nabla \undl V_{d_2}}{\undl V_{d_2}} \bigg|^2=d_2\int_{Q_T}\frac{\Delta \undl V_{d_2}}{\undl V_{d_2}}=-\int_{Q_T}(m_2-\widetilde U_{d_2}-\undl V_{d_2})=\int_{Q_T}(m_1-m_2)=0\]
by the equation of $\undl V_{d_2}$ in \eqref{5.27}. The above identity means that $\undl V_{d_2}$ is spatially homogeneous, and so is $m_2-\widetilde U_{d_2}$ by the second equation of \eqref{5.27}. Consequently, 
 \[\int_0^T(m_2-m_1)=\int_0^T(m_2-\undl V_{d_2}-\widetilde U_{d_2})\]
is spatially homogeneous. This contradicts the assumption that $\widehat{m}_1(x)\not\equiv\widehat{m}_2(x)$ on $\overline\Omega$. The claim follows. 

{\it Step 4: The limit of $\dd\bigl(\widetilde U_{d_2},\,\undl V_{d_2},\,\undl U_{d_2},\,\widetilde V_{d_2}\bigr)$ as $d_2\to\yy$}.

Similar to the proof of Lemma \ref{lm4.2}, by the regularity and boundedness of $(\widetilde U_{d_2},\undl V_{d_2},\undl U_{d_2},\widetilde V_{d_2})$, we derive that passing to a subsequence of $d_2$ if necessary,
 \bee
 \lim_{d_2\to\yy}\bigl(\widetilde U_{d_2},\,\undl V_{d_2},\,\undl U_{d_2},\,\widetilde V_{d_2}\bigr)=(U^*,V_*,U_*,V^*)\;\;\;\text{in}\;\;C(\ol Q_T),
 \label{5.28}\eee
and $(U^*(x,t),V_*(t),U_*(x,t),V^*(t))$ is a $T$-periodic solution to
  \bee\label{5.29}\begin{cases}
U^*_t=U^*\big(m_1(x,t)-V_*-U^*\big), & t>0,\;\;x\in\boo,\\
 V_*'(t)=V_*\big(\ol m_2(t)-U^*-V_*\big), & t>0,\\
 (U_*)_t=U_*\big(m_1(x,t)-V^*-U_*\big), &t>0,\;\;x\in\boo\\
  {V^*}'(t)=V^*\big(\ol m_2(t)-U_*-V^*\big), & t>0.
\end{cases}
 \eee
Moreover, by the equation of $U_*$ we have that
  \[\text{for any}\;\; x\in\boo,\;\;\; \text{either}\;\; U_*(x,t)\equiv 0\;\;\;\text{or}\;\;U_*(x,t)>0\;\;\text{on}\;\;[0,T].\]

{\it Step 5: We claim that $\min_{\overline\Omega}\int_0^T[m_1(x,t)-V^*(t)]\geq 0$}.

It follows from {\it Step 3} that $\min_{\overline\Omega}\int_0^T[m_1(x,t)-V_*(t)]\le 0$, which implies $V_*(t)>0$ in $[0,T]$ by the equation of $V_*(t)$. Then $V^*(t)\geq V_*(t)>0$ in $[0,T]$, and by the condition {\bf(M)} and the equation of $V^*$ in \eqref{5.29}, we have
  \bee\label{5.30}
\int_{Q_T}(m_1-U_*-V^*)=\int_{Q_T}(m_2-U_*-V^*)=0.
 \eee
By virtue of $\widehat{m}_1(x)\not\equiv\text{constant}$ on $\boo$ together with $\int_\oo(\widehat{m}_1-\widehat{\ol m}_1)=0$, we see that the function
\[
\int_0^T(m_1-l_2)=\int_0^T(m_1-\ol m_1)=\widehat{m}_1(x)-\widehat{\ol m}_1
\]
changes sign on $\boo$. Accordingly, the set
\begin{align*}
\Omega'=\left\{ x\in\overline\Omega:\int_0^T[m_1(x,t)-l_2(t)]\mathrm dt> 0 \right\}
\end{align*}
is nonempty, and $\Theta_{m_1-l_2}(x,t)>0$ on $[0,T]$ for all $x\in\oo'$. Notice that
  \[\Theta_{m_1-\theta_{d_2,m_2}}\to\Theta_{m_1-l_2}\;\;\;\text{in}\;\; C(\ol Q_T)\]
as $d_2\to\infty$. For every $x\in\oo'$, by \eqref{5.25}, \qq{5.26} and \eqref{5.28}, we have $U_*(x,t)\ge \Theta_{m_1-l_2}(x,t)>0$ on $[0,T]$. Then $U_*(x,t)=\Theta_{m_1-V^*}(x,t)>0$ on $[0,T]$, and by the equation of $U_*$,
   \[\int_0^T[m_1(x,t)-V^*(t)]\mathrm dt=\int_0^TU_*(x,t)\dt>0. \]

If the claim $\min_{\overline\Omega}\int_0^T[m_1(x,t)-V^*(t)]\geq 0$ fails, then there exists $x_*\in\overline\Omega$ such that
  \[\int_0^T[m_1(x_*,t)-V^*(t)]\mathrm dt<0.\]
Therefore, the sets defined by
  \begin{align*}
\Omega_1&=\left\{ x\in\overline\Omega:\int_0^T[m_1(x,t)-V^*(t)]\mathrm dt< 0 \right\},\\
\Omega_2&=\left\{ x\in\overline\Omega:\int_0^T[m_1(x,t)-V^*(t)]\mathrm dt> 0 \right\}
  \end{align*}
are both nonempty. It is clear that $U_*(x,t)\equiv 0$ on $[0,T]$ for any $x\in\ol\oo_1$. We will show that,
 \bee
 \text{for each}\;\;x\in\oo_2, \; U_*(x,t)>0\;\;\text{and}\;\;U_*(x,t)=\Theta_{m_1-V^*}(x,t)\;\;\text{on}
 \;\;[0,T]. \label{5.31}
 \eee
In fact, take $\ep(x)=\frac 14\int_0^T[m_1(x,t)-V^*(t)]\mathrm dt$. Then we have
 \begin{align*}
 \int_0^T V^*(t)\mathrm dt>\;&\int_0^T\widetilde V_{d_2}(x,t)\dt-\ep(x)
  \;\;\text{fixed}\;\;d_2\gg 1\\
  >\;&\int_0^T\widetilde V_{k,d_2}(x,t)\dt-2\ep(x)
  \;\;\text{fixed}\;\;n\gg 1.
  \end{align*}
Therefore,
  \[\int_0^T[m_1(x,t)-\widetilde V_{k,d_2}(x,t)]>\ep(x), \]
which implies $\Theta_{m_1-\widetilde V_{k,d_2}}(x,t)=\Phi_{m_1-\widetilde V_{k,d_2}}(x,t)>0$, and
 \bee
  \int_0^T\Theta_{m_1-\widetilde V_{k,d_2}}(x,t)\dt=\int_0^T[m_1(x,t)-\widetilde V_{k,d_2}(x,t)]\dt>\ep(x).\label{5.32}\eee
By the construction of $\undl u_{k,d_1,d_2}$, we have $\undt u_k=\Theta_{d_1,m_1-\tilde v_k}$. Since $\widetilde V_{k,d_2}=\lim_{d_1\to 0}\tilde v_k$, it follows form \eqref{5.23} and Lemma \ref{lm2.6} that
  \[ \undl{U}_{k,d_2}=\lim_{d_1\to 0}\undt u_k=\lim_{d_1\to 0}\Theta_{d_1,m_1-\tilde v_k}
  =\Theta_{m_1-\widetilde V_{k,d_2}}\;\;\;\text{in}\;\;C(\ol Q_T). \]
This combined with  \qq{5.32} implies $\int_0^T\undl{U}_{k,d_2}(x,t)\dt>\ep(x)$. Hence, by \qq{5.26} and \qq{5.28}, we have $\int_0^T U_*(x,t)\dt>\ep(x)>0$. This implies $U_*(x,t)>0$  and $U_*(x,t)=\Theta_{m_1-V^*}(x,t)>0$ on $[0,T]$ for all $x\in\oo_2$.

Consequently,
  \begin{align*} \int_{Q_T}U_*&=\int_0^T\int_{\Omega_2}\Theta_{m_1-V^*}
=\int_0^T\int_{\Omega_2}(m_1-V^*) \\
  &>\int_0^T\int_{\Omega_2}(m_1-V^*)+\int_0^T\int_{\Omega_1}(m_1-V^*)
  =\int_{Q_T}(m_1-V^*),
\end{align*}
which contradicts \eqref{5.30}. This claim is proved.

{\it Step 6}. From the results of Steps 3 and 5, together with \eqref{5.25} and \eqref{5.26}, we conclude that $\widehat V^*=\widehat V_*=\min_{\overline\Omega}\widehat{m}_1$. Define
  \begin{align*}
 \Omega_0=\left\{ x\in\overline\Omega: \widehat{m}_1(x)=\min_{\overline\Omega}\widehat{m}_1\right\},\;\;
\Omega_+=\left\{x\in\overline\Omega:\widehat{m}_1(x)>\min_{\overline\Omega}\widehat{m}_1\right\}.
   \end{align*}
In view of the equations of $U^*$ and $U_*$ in \eqref{5.29}, we have that for each $x\in\Omega_0$, $U^*(x,t)=U_*(x,t)\equiv 0$ on $[0,T]$. Same as the proof of \qq{5.31} we can show that for each $x\in\Omega_2$, $U^*(x,t)>0$ and $U_*(x,t)>0$ on $[0,T]$. Hence, for each $x\in\Omega_0$,
  \[\int_0^T(m_1-V_*-U^*)=\int_0^T(m_1-V^*-U_*)=0,\]
while for each $x\in\Omega_+$,
   \[ \int_0^T(m_1-V_*-U^*)=\int_0^T\frac{U^*_t}{U^*}=0,\;\;
   \int_0^T(m_1-V^*-U_*)=\int_0^T\frac{(U_*)_t}{U_*}=0.\]
Thus we have
  \[  \int_0^TU^*=\int_0^TU_*=\int_0^Tm_1(x,t)-\min_{\overline\Omega}
  \int_0^Tm_1(x,t),\;\;\forall\, x\in\overline\Omega. \]
Finally, combining the above argument with \eqref{5.25} and the convergence of  $(\widetilde U_{k,d_2},\undl V_{k,d_2}, \undl U_{k,d_2},\widetilde V_{k,d_2})$, we obtain \qq{1.14}, thereby completing the proof of Theorem \ref{th1.5}.

\begin{appendix}
\renewcommand{\thesection}{Appendix}
\section{Examples for the conditions in Theorem \ref{th1.4}\label{secB}}
\setcounter{equation}{0}

We will construct explicit examples of $m_1$ and $m_2$ such that all hypotheses of Theorem \ref{th1.4} hold.

\noindent{\bf Example}:\, {\it
Let $\Omega=(0,1)$ and $T=2$, $m_i(x,t)=10+r_i\cos(\pi x)+l_i\sin(\pi t)$, where $r_i$ and $l_i$ are positive integers, $r_i+l_i\leq 9$ and $(r_1,l_1)\neq (r_2,l_2)$.}

The conditions {\bf(M)}--{\bf(M2)} are satisfied, and $\widehat{m}_i(x)\not\equiv\mathrm{constant}$ for $i=1,2$. Notice that
  \[ \int_0^1m_i(x,t)\mathrm dx=10+l_i\sin(\pi t),\quad m_i(x,t)-\int_0^1m_i(x,t)\mathrm dx=r_i\cos(\pi x),\quad i=1,2.\]
It follows from \eqref{1.10} and \eqref{1.11} that for $i=1,2$,
  \[\Gamma_{m_i}(x,t)=\frac{r_i}{\pi^2}\cos(\pi x),\]
and
  \[  l_{m_i}(t)=\frac{\kk[\exp\kk(\int_{Q_T}m_i\rr)-1\rr]
  \exp\kk(\int_0^t\int_0^1m_i\rr)}{\int_0^2
  \exp\kk(-\int_0^{\tau}\int_0^1m_i\rr)+\kk[\exp\kk(\int_{Q_T}m_i\rr)-1\rr]
  \int_0^t\exp\kk(-\int_0^{\tau}\int_0^1m_i\rr)}.\]
Hence, for $i=1,2$,
   \[  \mathcal{E}_i=\int_{Q_T}|\partial_x\Gamma_{m_i}|^2
   =\frac{2r_i^2}{\pi^2}\int_0^1\sin^2(\pi x)=\frac{r_i^2}{\pi^2}.\]
\begin{enumerate}
\item[(1)] Choose $r_1=r_2$ and $l_1\neq l_2$. Then $m_1(x,t)-m_2(x,t)$ is spatially homogeneous.
\item[(2)] Choose $r_1\neq r_2$ and $l_1=l_2$. Then $m_1(x,t)-m_2(x,t)$ is spatially heterogeneous and $\int_0^1(m_1-m_2)\equiv 0$.
\item[(3)] Choose $r_1=7,r_2=6,l_1=1$ and $l_2=2$. Then $m_1(x,t)-m_2(x,t)$ is spatially heterogeneous and $\int_0^1(m_1-m_2)\not\equiv 0$. Using \eqref{1.13} and the tool of Matlab software, we calculate that $\Pi\approx 485.0953$.
\item[(4)] Choose $r_1=7,r_2=5,l_1=2$ and $l_2=1$. Then $m_1(x,t)-m_2(x,t)$ is spatially heterogeneous and $\int_0^1(m_1-m_2)\not\equiv 0$. Using \eqref{1.13} and the tool of Matlab software, we calculate that $\Pi\approx -75.8628$.
\end{enumerate}
\end{appendix}

{\bf Acknowledgement}: The authors would like to thank Professors Binxiang Dai and Zhi-An Wang for their valuable comments and suggestions.

{\bf Credit Author Statement}

{\bf Ethical approval}: The manuscript is original and hasn't been published elsewhere in any form or language.

{\bf Data Availability}:  Not applicable.

{\bf Conflicts of Interest}:
The authors declare that there are no known competing financial interests or personal relationships that could have influenced the work reported in this paper.

\end{document}